\documentclass[12pt,reqno]{amsart}
\usepackage{hyperref}
\usepackage{color, bm, amscd, tikz-cd}
\newcommand{\cB}{{\mathcal B}}

\newcommand{\cD}{{\mathcal D}}

\newcommand{\cF}{{\mathcal F}}

\newcommand{\cH}{{\mathcal H}}
\newcommand{\cI}{{\mathcal I}}

\newcommand{\cK}{{\mathcal K}}
\newcommand{\cL}{{\mathcal L}}
\newcommand{\cM}{{\mathcal M}}

\newcommand{\cO}{{\mathcal O}}

\newcommand{\cQ}{{\mathcal Q}}
\newcommand{\cR}{{\mathcal R}}
\newcommand{\cS}{{\mathcal S}}

\newcommand{\cU}{{\mathcal U}}

\newcommand{\cZ}{{\mathcal Z}}
\newcommand{\bH}{{\mathbb{H}}}

\newcommand{\sbm}[1]{\left[\begin{smallmatrix} #1
		\end{smallmatrix}\right]}
\newcommand{\biota}{{\boldsymbol \iota}}

\newcommand{\bTheta}{{\boldsymbol{\Theta}}}

\theoremstyle{plain}

\newtheorem{thm}{Theorem}[section]

\newtheorem{corollary}[thm]{Corollary}
\newtheorem{lemma}[thm]{Lemma}

\newtheorem{obs}[thm]{Observation}

\newtheorem{proposition}[thm]{Proposition}
\newtheorem{definition}[thm]{Definition}
\newtheorem{remark}[thm]{Remark}

\newtheorem{example}[thm]{Example}
\numberwithin{equation}{section}

\def\textmatrix#1&#2\\#3&#4\\{\bigl({#1 \atop #3}\ {#2 \atop #4}\bigr)}
\def\dispmatrix#1&#2\\#3&#4\\{\left({#1 \atop #3}\ {#2 \atop #4}\right)}
\numberwithin{equation}{section}

\def\textmatrix#1&#2\\#3&#4\\{\bigl({#1 \atop #3}\ {#2 \atop #4}\bigr)}
\def\dispmatrix#1&#2\\#3&#4\\{\left({#1 \atop #3}\ {#2 \atop #4}\right)}

\begin{document}
\title[Pairs of Commuting Contractions]{Functional Models and Invariant Subspaces for Pairs of Commuting Contractions}
\author[J. A. Ball]{Joseph A. Ball}
\address{Department of Mathematics, Virginia Tech, Blacksburg, VA 24061-0123, USA\\ joball@math.vt.edu}
\author[H. Sau]{Haripada Sau}
\address{Department of Mathematics, Virginia Tech, Blacksburg, VA 24061-0123, USA\\ sau@vt.edu, haripadasau215@gmail.com}
\thanks{The research of the second author is supported by SERB Indo-US Postdoctoral Research Fellowship, 2017.}

\begin{abstract}  A seminal result in operator theory is the Sz.-Nagy--Foias model theory for a completely nonunitary
Hilbert-space contraction operator $T$:  in short, any completely nonunitary contraction operator $T$ is unitarily
equivalent to its functional model $T(\Theta)$ on a certain functional Hilbert space (just a vectorial Hardy space
over the unit disk in the simplest case) which can be constructed explicitly from the so-called characteristic
function  $\Theta = \Theta_T$ of $T$.  The constructions also incorporate explicit geometric constructions
for the minimal isometric and unitary dilation of the operator $T$.  The goal of the present paper is to push
this theory to the case of a commuting pair of contraction operators $(T_1, T_2)$ having product $T  = T_1 T_2$
which is completely nonunitary.   The idea is to use the Sz.-Nagy-Foias functional model  for $T$ as the model space
also for the commutative tuple ($T_1, T_2)$ with $T = T_1 T_2$ equal to the usual Sz.-Nagy--Foias
model operator, and identify what added structure is required to classify such commutative
contractive factorizations $T = T_1 T_2$ up to unitary equivalence.
In addition to the characteristic function $\Theta_T$, we identify additional invariants
$({\mathbb G}, {\mathbb W})$ which can be used to construct a functional model for the commuting pair
$(T_1, T_2)$ and which have good uniqueness properties:  if two commutative contractive pairs $(T_1, T_2)$
and $(T'_1, T'_2)$ are unitarily equivalent, then their characteristic triples $(\Theta, {\mathbb G}, {\mathbb W})_T$
and $(\Theta, {\mathbb G}, {\mathbb W})_{T'}$ coincide in a natural sense.   We illustrate the theory with
several simple cases where the characteristic triples can be explicitly computed.  This work extends earlier
results of Berger-Coburn-Lebow \cite{B-C-L} for the case where $(T_1, T_2)$ is a pair of commuting isometries,
and of Das-Sarkar \cite{D-S}, Das-Sarkar-Sarkar \cite{D-S-S} and the second author \cite{sauAndo} for the case where $T = T_1T_2$ is pure (the operator sequence $T^{*n}$ tends strongly to $0$).  Finally we use the model to study
the structure of joint invariant subspaces for a commutative, contractive operator pair, extending results of
Sz.-Nagy--Foias for the single-operator case.
 \end{abstract}
\maketitle

\section{Introduction} \label{S:intro}
The starting point for many future developments in nonselfadjoint operator theory was the Sz.-Nagy dilation theorem
from 1953 \cite{sz-nagy}:   {\em if $T$ is a contraction operator on a Hilbert space $\cH$, then there is a
unitary operator $\cU$ on a  larger Hilbert space  $\widetilde \cK \supset \cH$ such that $T^n = P_\cH \cU^n |_\cH$
for all $n=0,1,2,\dots$.}
While the original proofs were more existential than constructive, there followed more concrete constructive proofs
(e.g., the Sch\"affer-matrix construction to be discussed below)  which evolved into a detailed geometric picture
of the dilation space (see \cite[Chapter II]{Nagy-Foias}).  Analysis of how the original Hilbert space $\cH$ fit into
the dilation space $\widetilde \cK$ and the discovery of appropriate transforms to convert the abstract spaces to
spaces of functions (holomorphic or measurable as the case may be) led to the discovery of the
characteristic function $\Theta_T$ of any c.n.u.\ contraction operator $T$ and how the c.n.u.\  contraction
operator $T$ can be represented  (up to unitary equivalence) as a compressed multiplication operator on a
functional-model Hilbert space constructed directly from $\Theta_T$.

The And\^o dilation theorem \cite{ando}, coming ten years later, provides a 2-variable analogue of the
Sz.-Nagy dilation theorem:  {\em given a commuting pair of contraction operators $(T_1, T_2)$ on a
Hilbert space $\cH$, there is a commutative pair of unitary operators $(\cU_1, \cU_2)$ on a larger Hilbert
space $\widetilde \cK \supset \cH$ so that, for all $n,m \ge 0$, $T_1^n T_2^m = P_\cH \cU_1^n \cU_2^m |_\cH$.}
The proof was an ad hoc expanded extension of the Sch\"affer-matrix construction for the single-operator case
which shed no light on the geometry of the dilation space (a consequence of the lack of uniqueness up to a
notion of unitary equivalence for And\^o dilations).  Consequently there has been essentially no follow-up
to the And\^o result in the direction of a Sz.-Nagy--Foias-type model theory for a commuting pair of contraction
operators as there was in the single-operator setting, although there have been some preliminary results
(see \cite{A-M-Dist-Var, BV}).

In an independent development, Berger-Coburn-Lebow \cite{B-C-L} obtained a model for a commutative-tuple
of isometries $(V_1, \cdots, V_d)$ by considering the Wold decomposition for the product $V= V_1 \cdots V_d$ and
understanding what form the factors $V_1, \dots, V_d$ must take in the shift-part of $V$ so as (i) to be themselves
commuting isometries, and (ii) to have product equal to the shift operator $V$.

Much of this paper can be seen as an effort to extend the Berger-Coburn-Lebow results for the case $d=2$
to the setting where the commutative pair of  isometries $(V_1, V_2)$ is replaced by a commutative pair of
contractions $(T_1, T_2)$ such that the product $T = T_1 T_2$ is c.n.u.  We can then place $T$ into its
Sz.-Nagy--Foias functional model determined by the characteristic function $\Theta_T$ of $T$, and look for
the form a pair of operators $(T_1, T_2)$ also defined on the functional model space must have so that (i)
each of $T_1$ and $T_2$ is a contraction operator, and (ii) $T_1 T_2 = T_2 T_1 = T$.  By using recent progress
on  construction of isometric And\^o dilations on a non-minimal extended version of the Sz.-Nagy--Foias
functional model space for the minimal isometric lift of $T$, we are able to identify the invariant
(the characteristic triple $({\mathbb G}, {\mathbb W}, \Theta_T)$---an expanded version of the Sz.-Nagy--Foias
invariant $\Theta_T$) which converts the study of the operator triple $(T_1, T_2, T= T_1 T_2)$ to
function theory on a Sz.-Nagy-Foias model space $(H^2(\cD_{T^*}) \oplus \overline{\Delta_T L^2(\cD_T) })
\ominus \{ \Theta f \oplus \Delta_T f \colon f \in H^2(\cD_T) \} $.
The ingredient ${\mathbb G}$ in the characteristic triple is closely connected to the {\em fundamental operators}
coming up recently in the study of $\Gamma$-contractions (commutative operator pairs having the
symmetrized bidisk as a spectral set \cite{B-P-SR}) as well as in the study of tetrablock contractions
(commutative triples of operators having the tetrablock ${\mathbb E}$ as a spectral set \cite{sir's tetrablock paper}).
Let us mention that earlier of Das-Sarkar-Sarkar \cite{D-S-S} found such a  model for a pair of commuting contractions
and earlier work of the second author \cite{sauAndo} worked out the invariant $({\mathbb G}, \Theta_T)$, but
only for the case where $T = T_1 \cdot T_2$ is {\em pure}  (i.e., SOT-$\lim_{n \to \infty} T^{*n} = 0$), in which case the second component ${\mathbb W}$ of the characteristic triple is vacuous.

The construction of the piece ${\mathbb G}$ for the characteristic triple $({\mathbb G}, {\mathbb W}, \Theta_T)$
for a given commutative contractive pair $(T_1, T_2)$ depends crucially on a collection of spaces and operators
$(\cF,  \Lambda, P, U)$ constructed from $(T_1, T_2)$ which we shall call an {\em And\^o tuple} for $(T_1, T_2)$
(also playing a key role in \cite{sauAndo}).
Let us note that this structure of {\em And\^o tuple} appears at least implicitly already in the original construction by
And\^o of a joint isometric lift for a commutative pair of contractions \cite{ando}.
Let us now discuss the notion of And\^o tuple more precisely.

We first recall the notion of {\em regular factorization} for a general (not necessarily square or commutative)
factorization $T  = T_1 T_2$ where $T \colon \cH \to \cH^{\prime \prime}$, $T_1 \colon \cH' \to \cH^{\prime \prime}$
and $T_2 \colon \cH \to \cH'$ are all contraction operators
with $\cH$, $\cH'$, $\cH^{\prime \prime}$ all possibly different Hilbert spaces.  We let
$$
  D_T = (I - T^* T)^{\frac{1}{2}},  \quad D_{T_1}  = (I - T_1^* T_1)^{\frac{1}{2}}, \quad
  D_{T_2}  = (I - T_2^* T_2)^{\frac{1}{2}}
 $$
 denote the associated defect operators with ranges denoted as
$$
 \cD_T = \overline{\operatorname{Ran}}\, D_T  \subset \cH, \quad
 \cD_{T_1} = \overline{\operatorname{Ran}}\, D_{T_1} \subset \cH',  \quad
  \cD_{T_2} = \overline{\operatorname{Ran}}\, D_{T_2} \subset \cH.
$$
Then the identity
\begin{equation}   \label{id1}
D_T^2 = I - T_2^* T_2 +T_2^* T_2 - T_2^* T_1^* T_1 T_2 = D_{T_2}^2 + T_2^* D_{T_1}^2 T_2
\end{equation}
shows that the operator $\Lambda \colon \cD_T \to \cD_{T_1} \oplus \cD_{T_2}$ defined densely by
\begin{equation}   \label{defLambda}
  \Lambda \colon D_T h = D_{T_1} T_2 h \oplus D_{T_2} h \text{ for all } h \in \cH
\end{equation}
is an isometry from $\cD_T$ into $\cD_{T_1} \oplus \cD_{T_2}$.  In case this isometry is onto
(so $\Lambda$ is in fact a unitary transformation from $\cD_T$ onto $\cD_{T_1} \oplus \cD_{T_2}$), we say that
the contractive factorization $T = T_1 \cdot T_2$ is {\em regular}.

\begin{remark}  \label{R:regfact}
{\rm In addition to the context of And\^o tuples, isometries of the form \eqref{defLambda}
play a key role in characterizing invariant subspaces of a contraction operator $T$ in terms of its
Sz.-Nagy--Foias model.  Indeed, invariant subspaces correspond to factorizations of the characteristic
function $\Theta = \Theta_1 \Theta_2$ (where $\Theta_1$ and $\Theta_2$ are also contractive
analytic operator-valued functions) which are {\em pointwise regular on the unit circle}, i.e., the operator $Z$
defined by
\begin{equation}   \label{Z}
Z \colon D_{\Theta}(\zeta)u \mapsto D_{\Theta_1}(\zeta) \Theta_2(\zeta) h \oplus D_{\Theta_2}(\zeta) h
\end{equation}
which is always an isometry from $\cD_{\Theta(\zeta)} $ into $\cD_{\Theta_1(\zeta)} \oplus \cD_{\Theta_2(\zeta)}$
is actually onto and hence unitary for almost all $\zeta \in{\mathbb T}$ (see \cite[Chapter VII]{Nagy-Foias}).
Whenever this happens, following \cite[Chapter VII]{Nagy-Foias}, we shall say that $\Theta =
\Theta_1 \cdot \Theta_2$ is a {\em regular factorization} of the contractive operator function $\Theta$.
We shall present an extension of Sz.-Nagy--Foias characterization of invariant subspaces to the context of
a commutative pair $(T_1, T_2)$ of contractions in Section \ref{S:invsub} below.
}\end{remark}

We now restrict to the square case where $\cH = \cH' = \cH^{\prime \prime}$ and $T$ has a commutative
contractive factorization $T = T_1 \cdot T_2  = T_2 \cdot T_1$ with $T_1$ and $T_2$ also contraction
operators on $\cH$.   Let us introduce the notation
\begin{align}
& \cD_{U_0} : =\operatorname{clos.}  \{ D_{T_1} T_2 h \oplus D_{T_2} h \colon h \in \cH\} \subset \cD_{T_1} \oplus \cD_{T_2}, \notag \\
& \cR_{U_0}:=\operatorname{clos.}  \{ D_{T_1} h \oplus D_{T_2} T_1 h \colon h \in \cH\} \subset \cD_{T_1} \oplus \cD_{T_2}.
\label{U0dom-codom}
\end{align}
Making use of the assumption that now $T_1$ and $T_2$ commute,  in addition to the identity \eqref{id1}
one can verify the following additional identity:
\begin{equation}   \label{id2}
D_{T_2}^2 + T_2^* D_{T_1}^2 T_2 = D_{T_1}^2 + T_1^* D_{T_2}^2 T_1.
\end{equation}
We conclude that the operator $U_0 \colon \cD_{U_0} \to \cR_{U_0}$
defined densely by
\begin{equation}  \label{U0}
U_0 \colon D_{T_1} T_2 h \oplus D_{T_2} h \mapsto  D_{T_1} h \oplus D_{T_2} T_1 h \text{ for all } h \in \cH
\end{equation}
is unitary.  If
\begin{equation}   \label{equal-codim}
\dim (\cD_{U_0})^\perp =  \dim  (\cR_{U_0})^\perp
\end{equation}
(orthogonal complements with respect to the ambient space $\cD_{T_1} \oplus \cD_{T_2}$),  we can extend
$U_0$ to a unitary operator  $U$
on all of $ \cF:= \cD_{T_1} \oplus \cD_{T_2}$.  In case condition \eqref{equal-codim} fails, we may enlarge the
ambient space $\cD_{T_1} \oplus \cD_{T_2}$
to the space $\cF: = \cD_{T_1} \oplus \cD_{T_2} \oplus \ell^2$ (here $\ell^2$ can be taken to be any separable
infinite-dimensional Hilbert space) with respect to which condition  \eqref{equal-codim} does hold (with value
$\infty$ equal to the common co-dimension), and thereby come up with a unitary extension $U$ of $U_0$
on the larger space $\cF$.  Finally, given $\cF$, $\Lambda$, $U$ constructed
in this way, we let $P$ be the orthogonal projection onto the first component as an operator on $\cF$:
\begin{align*}
 &  P \colon f \oplus g \mapsto f \oplus 0 \text{ in case } \cF = \cD_{T_1} \oplus \cD_{T_2}, \\
 & P \colon f \oplus g \oplus h \mapsto f \oplus 0 \oplus 0 \text{ in case } \cF = \cD_{T_1} \oplus \cD_{T_2}  \oplus \ell^2.
 \end{align*}

\begin{definition} \label{D:Ando-tuple}  {\rm Given a commutative contractive factorization
$T = T_1 \cdot T_2 = T_2\cdot T_1$
for a contraction operator $T$, we say that the collection of spaces and operators $(\cF, \Lambda, P , U)$
is an {\em And\^o tuple} for the pair $(T_1, T_2)$ if it arises via the construction given in the preceding paragraph.
} \end{definition}

\begin{remark} \label{R:Ando-tuple-unique} {\rm Let us note that in general the And\^o tuple depends
nontrivially on the choice of unitary extension $U$ of the partially defined isometry $U_0$.  An And\^o
tuple is uniquely determined from $(T_1, T_2)$ exactly when both $\cD_{U_0}$ and $\cR_{U_0}$
are equal to all of  $\cD_{T_1} \oplus \cD_{T_2}$.  Note that $\cD_{U_0} =  \cD_{T_1} \oplus \cD_{T_2}$
corresponds to $T = T_1 \cdot T_2$
being a regular factorization while the $\cR_{U_0} = \cD_{T_1} \oplus \cD_{T_2}$ corresponds to
$T = T_2 \cdot T_1$ being a regular factorization.  We conclude that {\em an And\^o tuple of the
commutative pair of contractions $(T_1, T_2)$ is uniquely determined exactly when both
$T = T_1 \cdot T_2$ and $T = T_2 \cdot T_1$ are regular factorizations}.   We shall see that in the finite-dimensional
setting it is enough to assume that one of these factorizations is regular while in the infinite-dimensional setting
this is not the case (see Theorem \ref{T:left-right} below.
}\end{remark}

We shall actually need a couple of variations of the notion of And\^o tuple ({\em coordinate-free And\^o tuple
for $(T_1, T_2)$} and {\em coordinate-free And\^o tuple for $(T_1^*, T_2^*)$})---see Definitions \ref{D:Ando-c}
and \ref{D:Ando-c*} below.

The paper is organized as follows.  Following this Introduction,
in Section \ref{S:1D} we review the geometric structure associated with a unitary dilation $\cU$ or isometric lift
$V$ of a contraction operator $T$ from \cite[Chapter II]{Nagy-Foias}, with special additional attention
to the case where $\cU$ or $V$ may not be minimal.  We then explain how particular choices of coordinates
in this structure lead to three distinct functional models for the minimal unitary-dilation or isometric-lift
space associated with the names of Sch\"affer, Douglas, and Sz.-Nagy--Foias; the term {\em functional} is
somewhat loose and strictly applies only to the Sz.-Nagy--Foias case (see  Section \ref{S:epilogue}).
In any case this analysis enables us to find an explicit identification between the Douglas isometric-lift space
and the Sz.-Nagy-Foias isometric-lift space for a given c.n.u.\ contraction operator; this in turn is crucial
for defining the second component ${\mathbb W}$ in a characteristic triple $({\mathbb G}, {\mathbb W}, \Theta_T)$
for a given pair of commuting contraction operators  ($T_1, T_2)$ to come in Section \ref{S:NFmodel}.

Section \ref{S:c} is an attempt to mimic Chapter II of \cite{Nagy-Foias} for the case of a contractive commutative
operator pair $(T_1, T_2)$ in place of a single contraction operator $T$.  We use the existence of
a And\^o isometric lift $(V_1, V_2)$ for a commutative contractive pair $(T_1, T_2)$ together with the
ingredients from modified "coordinate-free" And\^o tuples  to arrive at three functional-model
forms (of Sch\"affer, Douglas, and Sz.-Nagy--Foias type) for an And\^o dilation.
Unlike the single-operator case,
the resulting And\^o  dilations need not have much to do with the original assumed (coordinate-free) dilation,
nor with each other, as the construction depends on a choice of And\^o tuple which in turns depends (except in
the nongeneric case where both the factorization $T = T_1 \cdot T_2 = T_2 \cdot T_1$ are regular)  on
an arbitrary choice of unitary extension $U$ of the partially defined isometry $U_0$.
 While it appears that
the Sch\"affer And\^o isometric lift does not have much to do with the Douglas or Sz.-Nagy--Foias isometric lift
(even when one tries to match the respective unitary operators $U$), it does appear that the Douglas and
Sz.-Nagy--Foias isometric lifts can be arranged to be unitarily equivalent with appropriate matching choices of unitary
extensions $U$ in the respective constructions of an And\^o tuple.
Strictly speaking, the construction here
(based on so-called coordinate-free And\^o tuple for $(T_1, T_2)$ and coordinate-free And\^o tuple for
$(T_1^*, T_2^*)$ and assumed coordinate-free And\^o isometric lift $(V_1, V_2)$ of $(T_1, T_2)$) does
not prove the existence of an And\^o isometric lift; however, the proof can be rearranged, based on
the ``coordinate-dependent" definition of And\^o tuple (Definition \ref{D:Ando-tuple} above) to prove from scratch
the existence of a Sch\"affer-type and Douglas-type And\^o isometric lift; this is done in \cite{sauAndo}.

In Section \ref{S:NFmodel} we introduce the characteristic triple $({\mathbb G}, {\mathbb W}, \Theta_T)$
for a commutative, contractive pair ${\mathbf T} = (T_1, T_2)$ with $T = T_1 \cdot T_2 = T_2 \cdot T_1$,
and show that this has all the invariance properties for the
commutative, contractive pair $(T_1, T_2)$ as the Sz.-Nagy--Foias characteristic function $\Theta_T$
has for a single c.n.u.\ contraction operator $T$.  In particular, there is a functional-model pair of
commutative contractions $({\mathbf T}_1, {\mathbf T}_2)$ acting on the Sz.-Nagy--Foias functional
model space $\cH_{NF}$ such that the original abstract commutative, contractive pair $(T_1, T_2)$
in unitarily equivalent to the concrete functional-model commutative, contractive pair
$({\mathbf T}_1, {\mathbf T}_2)$, and the characteristic triple for $T$ serves as a complete unitary
invariant for $T$.  We include some simple examples of characteristic triples to illustrate the ideas.

The final Section \ref{S:invsub} extends the Sz.-Nagy--Foias analysis of invariant subspaces in the functional
model to the commutative-contractive-pair setting.

Finally, let us mention that we are in the process of extending the framework of this paper to the setting of
commutative $d$-tuples $(T_1, \dots, T_d)$ of contraction operators on a Hilbert space $\cH$
\cite{BSprep}.

\section{Models for a unitary and isometric dilation of a contraction operator}  \label{S:1D}

\subsection{The coordinate-free version}
Let $T$ be a contraction operator on a Hilbert space $\cH$.  We say that the pair $(\widetilde \Pi, \cU)$ is a {\em unitary dilation} of $T$
if (i) $\widetilde \Pi$ is  an isometric embedding of $\cH$ into a Hilbert space $\widetilde \cK$, and (ii) $\cU$ is a unitary operator on
$\widetilde \cK$ such that
$$\widetilde \Pi^*  \cU^n \widetilde \Pi  =  \begin{cases}   T^n &\text{if } n \ge 0, \\
   T^{*n} &\text{if } n < 0, \end{cases}
$$
or, equivalently in local form,  for all $h, h' \in \cH$ and integers $n,m$ we have
$$
\langle \cU^n \widetilde  \Pi h, \, \cU^m  \widetilde \Pi h' \rangle =
\begin{cases} \langle  T^{n-m}  h, \, h' \rangle &\text{if } n \ge m,\\
  \langle h, T^{m-n} h' \rangle &\text{if }  n < m.  \end{cases}
$$
We say that $(\widetilde \Pi, \cU)$ is a {\em minimal unitary dilation} of $T$ if in addition
the smallest $\cU$-reducing subspace of $\widetilde \cK$ containing $\widetilde \Pi \cH$ is all of $\widetilde \cK$.

We say that the pair $(\Pi, V)$ is an {\em isometric dilation} of the contraction operator $T$ if
(i) $\Pi$ is an isometric embedding of $\cH$ into  the Hilbert space $\cK$, and (ii) $V$ is an isometric operator on $\cK$ such that
the dilation property holds:
$$
\Pi^* V^n \Pi = T^n \text{ for } n=0,1,2,\dots,
$$
or, in local form, for all $h,h' \in \cH$ and $n=1,2,3,\dots$ we have
$$
\langle V^n \Pi h,  \Pi h' \rangle = \langle T^n h, h' \rangle.
$$
We say that $(\Pi, V)$ is a {\em minimal isometric dilation} if the smallest invariant subspace if $\cK$ containing $\cH$ is all of $\cK$

In case $(\Pi, V)$ is a minimal
isometric dilation of $T$, then in fact $(\Pi, V)$ is a {\em isometric lift} (after isometric embedding) of $T$, meaning that
$V^*$ is a  {\em coisometric extension} (after isometric embedding) of $T^*$, i.e., $V^*\Pi = \Pi T^*$.
When we speak about isometric dilations, we are usually interested in minimal isometric dilations, and therefore, as is usually done,
we assume at the outset that the isometry is actually a lift.  Thus we abuse the terminology slightly and say that
$(\Pi, V)$ is an
{\em isometric dilation}  of $T$ if (i)     $\Pi$ is an isometric embedding of $\cH$ into a Hilbert space $\cK$, and (ii) $V$ is an isometric
operator on $\cK$ such that $\Pi T^*= V^* \Pi$.  We say that $(\Pi, V)$ is a {\em minimal isometric dilation} of $T$ if in addition
$\cK$ is the smallest invariant subspace of $V$ containing $\Pi \cH$.
We say that two isometric dilations $(\Pi, V)$ and $(\Pi', V')$ are {\em unitarily equivalent} if there is a unitary operator $U$ from $\cK$ to $\cK'$
so that $\Pi' = U \Pi$ and $U V = V' U$.

The classical formulation of unitary/isometric dilation is the case where one takes $\cH \subset \cK  \subset \widetilde \cK$ and
$\Pi = \iota_{\cH \to \cK} \colon \cH \to \cK$ and $\widetilde \Pi = \iota_{\cH \to \widetilde \cK} \colon \cH \to \widetilde \cK$ are the inclusion maps. We can construct
unitary/isometric  dilations of $T$ via e.g.\  the Sch\"affer-matrix construction, so existence of
unitary/isometric dilations is not an issue.  Let us assume that $\cU$ on $\widetilde \cK$ and $V$ on $\cK$ is a
unitary (respectively, isometric) dilation  of $T$ in the classical sense ($\cH \subset \cK \subset \widetilde \cK$ with $\Pi \colon \cH \to \cK$ and
$\widetilde \Pi \colon \cH \to \widetilde \cK$ equal to the respective inclusion maps).
For this case there is a nice coordinate-free description of the geometry behind any minimal unitary/isometric dilation as follows (see \cite[Chapter II]{Nagy-Foias}).

The space $\widetilde \cK$ has two internal orthogonal direct-sum decompositions
\begin{align}
\widetilde \cK & = M(\cL_*) \oplus \cR   \label{decom1} \\
& = M_-(\cL_*) \oplus  \cH \oplus M_+(\cL)  \oplus \cR_0 \label{decom2}
\end{align}
where
\begin{align*}
& \cL_* = \overline{\cU(\cU^* - T^*) \cH} = \overline{(I - \cU T^*) \cH}, \quad  \cL = \overline{(\cU- T) \cH},   \\
& M(\cL_*) = \bigoplus_{n=-\infty}^\infty \cU^n \cL =
M_-(\cL_*) \oplus M_+(\cL_*)
\end{align*}
where we set
\begin{align*}
& M_-(\cL_*) = \bigoplus_{n=-\infty}^{-1} \cU^n \cL_*, \quad  M_+(\cL_*) = \bigoplus_{n=0}^\infty \cU^n \cL_*,  \\
& M_-(\cL) = \bigoplus_{n=-\infty}^{-1} \cU^n \cL, \quad    M_+(\cL) = \bigoplus_{n=0}^\infty \cU^n \cL
\end{align*}
and where
\begin{equation}   \label{cR}
  \cR = \widetilde \cK \ominus M(\cL_*), \quad \cR_0 =  [M_+(\cL_*) \oplus \cR] \ominus [ \cH \oplus M_+(\cL)].
\end{equation}

We note that $\cU|_{M(\cL_*)}$ and $\cU|_{M(\cL)}$ are bilateral shifts while $\cU|_{M_+(\cL_*)}$ and $\cU|_{M_+(\cL)}$
are unilateral shift operators, with respective wandering subspaces equal to $\cL_*$ and $\cL$.

We next collect a few additional properties concerning the geometry of the minimal isometric lift of $T$ of the form
$(\iota_{\cH \to \cK_+}, V)$.

\begin{proposition}  \label{P:dil}
Let $T$ be a contraction operator on a Hilbert space with unitary dilation $\cU$ on $\widetilde \cK$ and isometric dilation $V$ on $\cK$ as described above.
Then:
\begin{enumerate}
\item  The maps
$ \iota_* \colon \cL_* \to \cD_{T^*}$ and $\iota \colon \cL \to \cD_T$ given densely by
$$
  \iota_* \colon (I - V T^*) h \mapsto D_{T^*} h, \quad
  \iota \colon (V - T) h \mapsto D_T h
 $$
 for $h \in \cH$ extend to define unitary identification maps from $\cL_*$ onto $\cD_{T^*}$ and from $\cL$ onto $\cD_T$ respectively.

\item  The projection $P_{\cL_*}$ onto the wandering subspace $\cL_*$  restricted to $\cH$ is given by
\begin{equation}   \label{ortho-proj}
P_{\cL_*}|_\cH=   (I - V T^*)|_\cH.
\end{equation}

\item For $h \in \cH$, the orthogonal projection $P_{M_+(\cL_*)} h$ of $h$ onto $M_+(\cL_*)$ is given by
\begin{equation}   \label{PM+L}
  P_{M_+(\cL_*)} h = \sum_{n=0}^\infty V^n (I - V T^* ) T^{*n} h.
\end{equation}
with
\begin{equation}   \label{proj-norm1}
\| P_{M_+(\cL_*)} h \|^2 = \sum_{n=0}^\infty \| D_{T^*} T^{*n} h \|^2.
\end{equation}

\item For $h \in \cH$ we have
\begin{equation}  \label{proj-norm2}
  \| P_{\cR} h \|^2 =   \langle Q^2 h, h \rangle
\end{equation}
where we set $Q^2 = \operatorname{SOT-lim}_{n \to \infty} T^n T^{*n}$, i.e.,
for all $h \in \cH$,  $\| (T^n T^{*n} - Q^2)h \| \to 0$ as $n \to \infty$.

 \item $\overline{P_\cR \cH}$ is invariant under $V^*$ and  $V^*|_{\overline{P_\cR \cH}}$ is isometric.

\item Either of the following conditions is necessary and sufficient for $\cU$ on $\widetilde \cK$ to be a minimal unitary
dilation of $T$, or
equivalently for $V$ on $\cK$ to be minimal isometric dilation of $T$:

\smallskip

\begin{enumerate}

\item  $\widetilde \cK = M_-(\cL_*) \oplus \cH \oplus M_+(\cL)$ (so the subspace $\cR_0$ in \eqref{cR} is zero), or
$\cK = \overline{ \cup_{n=0}^\infty V^n \cH} = \cH \oplus M_+(\cL)$;

\smallskip

\item  the linear manifold $\bigcup_{n=0,1,2,\dots} V^n P_\cR \cH$ is dense in $\cR$.

\end{enumerate}

\smallskip

\noindent
Furthermore the condition

\begin{enumerate}
\item[(c)] $M(\cL_*) + M(\cL)$ is dense in $\widetilde \cK$
\end{enumerate}

\smallskip

\noindent
is sufficient for the minimality of $(\iota_{\cH \to \widetilde \cK}, \cU)$ as a unitary dilation of $T$ (respectively,  minimality of
$(\iota_{\cH \to \cK}, V)$ as an isometric dilation of $T$).

\smallskip

\item  Suppose that $T$ is completely nonunitary.  Then condition (c) above is both necessary and sufficient for minimality
of $(\iota_{\cH \to \widetilde \cK}, \cU)$ as a unitary dilation of $T$ (respectively, the minimality of
$(\iota_{\cH \to \cK}, V)$ as an isometric dilation of $T$).

In this case the decompositions \eqref{decom1}--\eqref{decom2} simplify to
\begin{equation}  \label{decom3}
\widetilde \cK = M(\cL_*) \oplus \cR = M_-(\cL_*) \oplus \cH \oplus M_+(\cL)
\end{equation}
with corresponding decompositions for the space $\cK: = \widetilde \cK \ominus M_-(\cL_*)$
\begin{equation}   \label{decom4}
\cK = M_+(\cL_*) \oplus \cR = \cH \oplus M_+(\cL),
\end{equation}
and  we have the following density condition:
\begin{equation} \label{cRDeltaTheta}
\overline{ (I - P_{M(\cL_*)}) M(\cL)}  = \overline{ P_\cR M(\cL) }= \cR.
\end{equation}

\end{enumerate}
\end{proposition}

\begin{proof}
(1) and (2) follow from elementary computations as in \cite[Chapter II Section 1]{Nagy-Foias}.

As $\cL_*$ is the wandering subspace for $M_+(\cL_*)$,  we have in general that
$P_{M_+(\cL_*)} = \sum_{n=0}^\infty V^n P_\cL V^{*n}$.
Given $h \in \cH$, by making use of the fact that $T^* = V^*|_\cH$ together with the formula \eqref{ortho-proj} for
$P_{\cL_*} |\cH$, the formula \eqref{PM+L} for   $P_{M_+(\cL_*)} h$ follows immediately.
It then follows that
\begin{align*}
   \| P_{M_+(\cL_*)} h \|^2 &  = \sum_{n=0}^\infty \| (I - V T^*) T^{*n} h \|^2 \\
  &  = \sum_{n=0}^\infty \| D_{T^*} T^{*n} h \|^2 \text{ by part (1) above}
\end{align*}
verifying  \eqref{proj-norm1} and  part (3) follows.

As for (4), note that
$\operatorname{SOT-lim}_{n \to \infty} T^n T^{*n}$ exists since $T^n T^{*n}$ is a monotonically decreasing sequence of positive-semidefinite operators
since $T^*$ is a contraction operator and hence we may define a positive-semidefinite operator $Q$ on $\cH$ by
\begin{equation}   \label{Q}
 Q^2 = \operatorname{SOT-lim}_{n \to \infty} T^n T^{*n}.
 \end{equation}
 From \eqref{decom1}--\eqref{decom2} we see that $\cH \subset M_+(\cL_*) \oplus \cR$, and hence, for $h \in \cH$,
 $$
  \| P_\cR h \|^2 = \| h \|^2 - \| P_{M_+(\cL_*)} \|^2.
 $$
 Thus
\begin{align*}
& \| P_\cR h \|^2 = \| h \|^2 - \| P_{M_+(\cL_*)} h \|^2   = \| h \|^2 - \sum_{n=0}^\infty \| D_{T^*} T^{*n} h \|^2  \\
 & \quad \quad  = \| h \|^2 - \lim_{N\to \infty}  \sum_{n=0}^N \langle T^n (I - T T^*) T^{*n} h, h \rangle \\
 & \quad \quad = \| h \|^2 - \left( \| h \|^2 - \lim_{N \to \infty} \langle T^{N+1} T^{*(N+1)} h, h \rangle \right) \\
 & \quad \quad = \lim_{N \to \infty} \langle T^{N+1} T^{*(N+1)} h, h \rangle = \langle Q^2 h, h \rangle
\end{align*}
(4) follows.

\smallskip

As for (5), note that
$$
V^* P_\cR h = P_\cR V^* h = P_\cR T^* h,
$$
we conclude that $\overline{P_\cR \cH}$ is invariant for $V^*$.  Furthermore, since $\cR$ is reducing for $\cU$ we have
$$
 \| V^* P_\cR h \| = \| \cU^* P_\cR h \| = \| P_\cR h \|
$$
and we see that $V^*|_{\overline{P_\cR \cH}}$ is isometric, and (4) follows.

We now discuss the characterizations of minimality in (6).
Note first that any $\cU$-reducing subspace containing $\cH$ must contain $M_-(\cL_*)$, since $\cU^* \cL_* =
\overline{(\cU^* - T^*) \cH}$.  Similarly, any $\cU$-invariant subspace containing $\cH$ (i.e., any $V$-invariant subspace containing $\cH$)
must contain $M_+(\cL)$ since $\cL = \overline{( \cU - T) \cH}$.

From the decompositions
$$
(I - \cU T^*) h = (I - T T^*) h - (\cU - T) T^* h, \quad \cU h = Th + (\cU - T) h
$$
we see that
$$
\cL_* + \cU \cH \subset \cH + \cL.
$$
Similarly, from the decompositions
$$
  h = (I - \cU T^*) h + \cU T^* h, \quad (\cU - T) h = -(I - \cU T^*)T h + \cU (I - T^* T) h
$$
we see that we have the reverse containment
$$
 \cH + \cL \subset \cL_* + \cU \cH.
$$
From the decomposition \eqref{decom2} we see that these decompositions are in fact orthogonal, so we have the identity
of subspaces of $\cK$:
$$
 \cL_* \oplus \cU \cH = \cH \oplus \cL.
$$
We may use this identity to read off that the subspace
$$
  \widetilde \cK_{\rm min}: = M_-(\cL_*) \oplus \cH \oplus M_+(\cL)
$$
is reducing for $\cU$, and that
$$
 \cK_{\rm min}:= \cH \oplus M_+(\cL)
$$
is invariant for $V$.  From the preceding discussion  we know that any reducing subspace for $\cU$ containing $\cH$ must contain
$\widetilde \cK_{\rm min}$ and that any invariant subspace for $V$ containing $\cH$ must contain $\cK_{\rm min}$.  In this way
we see that condition  (a) is necessary and sufficient for minimality.

We have observed above that  $\widetilde \cK_{\rm min}: = M_-(\cL_*) \oplus \cH \oplus M_+(\cL)$
is $\cU$-reducing, and hence $\cR_0 = \widetilde \cK \ominus \widetilde \cK_{\rm min}$ is $\cU$-reducing as well.  From
\eqref{decom2} we see that $\cR_0 \perp M_-(\cL_*)$.  As $\cR_0$ is $\cU$-reducing this forces that in fact
$\cR_0 \perp M(\cL_*)$, i.e. (from \eqref{decom1}), $\cR_0 \subset \cR$.  Thus the characterization
of $\cR_0$ in \eqref{cR} can be rewritten as
$$
  \cR_0 = \{ k \in \cR \colon k \perp \cH \oplus \cM_+(\cL) \}.
$$
From the fact that $\cL = (V - T) \cH$, one can see that
$\cH \oplus M_+(\cL) = \overline{ \cup_{n=0}^\infty V^n \cH}$.  Thus we may rewrite our characterization of $\cR_0$ as
\begin{equation}  \label{cR0}
 \cR_0 = \{ k \in \cR \colon k \perp \bigcup_{n=0}^\infty V^n \cH \}.
\end{equation}

Suppose now that $k \in \cR$ is such that $k \perp \bigcup_{n=0}^\infty V^n P_\cR \cH$. Then $k \in \cR$
is such that for all $h \in \cH$ and  $n=0,1,2,\dots$ we have
$$
0 = \langle k, V^n P_\cR h \rangle = \langle  k, \cU^n P_\cR h \rangle = \langle k, P_\cR \cU^n h \rangle =
\langle k, V^n h \rangle,
$$
i.e.,  $k \in \cR$ is such that $k \perp \bigcup_{n=0}^\infty V^n \cH$, or $k \in \cR_0$ by \eqref{cR0}.
 The argument is reversible:  if
$k \in \cR_0$, then $k$ is orthogonal to
$\overline{\cup_{n=0}^\infty V^n P_\cR \cH}$.  We have thus verified
$$
  \cR_0 =  \{k \in \cR \colon k \perp \bigcup_{n=0}^\infty V^n P_\cR \cH\}.
 $$
In particular, $\cR_0 = \{0\}$ (i.e., $\cU$ is a minimal unitary dilation of $T$ or equivalently $V$ is a minimal isometric lift of $T$
by criterion (a) already proved) if and only if $\bigcup_{n=0}^\infty V^n P_\cR \cH$ is dense in $\cR$, and
it follows that (b) is also a criterion for minimality.

Recall that an argument in the penultimate paragraph above shows that $\cR_0 \subset \cR = \widetilde \cK \ominus M(\cL_*)$,
so we get $\cR_0 \perp M(\cL_*)$.  From the decomposition \eqref{decom2} we see that $\cR_0 \perp M_+(\cL)$; as
$\cR_0$ is $\cU$-reducing, this implies that furthermore $\cR_0 \perp M(\cL)$.  Thus we always have
$$
  \cR_0 \subset ( M(\cL_*) + M(\cL) )^\perp.
$$
In particular, if $M(\cL_*) + M(\cL)$ is dense in $\widetilde \cK$, it follows that $\cR_0 = \{0\}$ and $\cU$ is a minimal unitary dilation
(as well as $V$ is a minimal isometric lift) of the contraction operator $T$.  Thus criterion (c) is sufficient for minimality.
This completes the proof of part (6) of Proposition \ref{P:dil}.

Suppose next that $T$ is completely nonunitary contraction operator and that minimality holds, so
$\cR_0 = \{0 \}$.  Suppose that $k \in \widetilde \cK$ is orthogonal to $M(\cL_*) + M(\cL)$.  In particular, we see that
$k \perp M_-(\cL_*)$ and $k \perp M_+(\cL)$.  From the decomposition \eqref{decom2} with $\cR_0 = \{0\}$, we see that
$k \in \cH$.  As $k$ is also orthogonal to $M_-(\cL)$, we also have, for all $h' \in \cH$ and $n=1,2,3,\dots$,
\begin{align*}
0 & = \langle h, \cU^{*n} ( \cU - T) h' \rangle = \langle \cU^{n-1} h, h' \rangle - \langle \cU^n h, T h' \rangle  \\
& = \langle T^{n-1} h, h' \rangle - \langle T^n h, Th \rangle = \langle (I - T^* T) T^{n-1} h, h' \rangle
\end{align*}
from which we conclude that $(I - T^* T)T^{n-1} h = 0$ for $n=1,2,3,\dots$.  This in turn gives
$$
 \| h \|^2 = \| T h \|^2 = \| T^2 h \|^2 = \dots = \| T^n h \|^2 = \dots,
 $$
 or $h$ is in the isometric subspace for $T$ (the largest invariant subspace $\cH_i$ for $T$ such that
 $T|_{\cH_{\rm i}}$ is isometric.  Similarly $k \in \cH$ and $k \perp M_+(\cL_*)$ implies that
 $k \in \cH_{\rm ci}$ where $\cH_{\rm ci}$ is the largest $T^*$-invariant subspace such that $T^*|_{\cH_{\rm ci}}$ is
 isometric.  Putting all this together means that $h \in \cH_{\rm i} \cap \cH_{\rm ci} =: \cH_{\rm u}$, where $\cH_{\rm u}$
 is the largest $T$-reducing subspace such that $T|_{\cH_{\rm u}}$ is unitary.  The hypothesis that $T$ is
 completely nonunitary amounts to saying that $\cH_{\rm u} = \{0\}$.   Thus finally $k = 0$ and we conclude
 $M(\cL) + M(\cL_*)$ is dense in $\widetilde \cK$.  It follows that criterion (c) is also necessary for minimality
 in case $T$ is completely nonunitary.

 Finally the validity of \eqref{cRDeltaTheta} for the case where $M(\cL_*) + M(\cL)$ is dense in $\widetilde \cK$ is
 part of Theorem II.2.1 in \cite{Nagy-Foias}.  This concludes the proof of part (7) of Proposition \ref{P:dil}.
\end{proof}

We close this section with a useful uniqueness result  for minimal isometric lifts of a c.n.u.\ contraction operator $T$.

\begin{thm} \label{T:min-iso-lift}  Suppose that $T$ is a c.n.u.\ contraction operator on a Hilbert space $\cH$ and that
$(\cK, \Pi, V)$ and $(\cK', \Pi', V')$ are two minimal isometric lifts of $T$.  Then $(\cK, \Pi, V)$ and $(\cK', \Pi', V')$
are unitarily equivalent as isometric lifts of $T$,  i.e., there is a unitary operator $U \colon \cK \to \cK'$ such that
\begin{equation}  \label{define-equiv}
  U \Pi = \Pi', \quad  U V  = V' U,
\end{equation}
and furthermore $U$ so specified is unique.
\end{thm}

\begin{proof}   This is an adaptation of Theorem I.4.1 in \cite{Nagy-Foias} where the classical case
is considered.
The minimality of $(\cK, \Pi, V)$  and $(\cK', \Pi', V')$ as isometric lifts of $T$ forces the spaces
 $\cK$ and $\cK'$ to be given by
$$
   \cK = \overline{\operatorname{span}} \{ V^n \Pi \cH \colon n=0,1,2,\dots\}, \quad
    \cK' = \overline{\operatorname{span}} \{ V^{\prime n} \Pi' \cH \colon n=0,1,2,\dots\}.
$$
If $(\cK, \Pi, V)$ is any minimal isometric lift, one can check that
$$
\langle V^n \Pi h, V^m \Pi \widetilde h \rangle_\cK =
\begin{cases}  \langle V^{n-m} \Pi h,  \Pi \widetilde h \rangle_\cK
  = \langle T^{n-m} h, \widetilde h \rangle_\cH & \text{if } n \ge m \ge 0, \\
  \langle h, V^{m-n} \widetilde h \rangle_\cK = \langle h, T^{m-n} \widetilde h \rangle_\cH &\text{if }
  m \ge n \ge 0,  \end{cases}
 $$
 and thus $\langle V^n h, V^m \widetilde h \rangle_\cK$ does not depend on the choice of minimal isometric
 lift.  Thus, if $(\cK, \Pi, V)$ and $(\cK', \Pi', V')$ are two minimal isometric lifts, the map defined densely by
 $$
  U \colon \sum_{n=0}^N V^n \Pi h \mapsto \sum_{n=0}^N V^{\prime n} \Pi' h
 $$
 extends by continuity to a unitary map from $\cK$ to $\cK'$ implementing a unitary equivalence between
 the minimal isometric lifts $(\cK, \Pi, V)$ and $(\cK', \Pi', V')$.  Furthermore from the defining intertwining conditions
 for unitary equivalence, we see that  any such unitary equivalence
 must be of this form.
\end{proof}

\subsection{Function-theoretic models}

We shall assume in the sequel that $T$ is a completely nonunitary contraction operator on $\cH$.

\subsubsection{The Sch\"affer model for the minimal isometric dilation}  \label{S:Schaffer}
The Sch\"affer model is based on the second of the decompositions \eqref{decom4} for the minimal isometric-dilation
space $\cK$ for $T$.  Note that we can extend the unitary identification $\iota \colon \cL \to \cD_{T}$ in part (1) of
Proposition \ref{P:dil} to a unitary
identification $\biota  \colon M_+(\cL) \to H^2(\cD_{T^*})$ according to the formula
\begin{equation}   \label{biota}
  \biota \colon \sum_{n=0}^\infty  V^n \ell_{n} \mapsto \sum_{n=0}^\infty  (\iota \ell_{n}) z^n.
\end{equation}
  Let us write $\cK_{S}$ for the {\em Sch\"affer isometric lift space}
  $$
   \cK_{S} = \begin{bmatrix} \cH \\ H^2(\cD_T) \end{bmatrix}
  $$
  and  $\Pi_S$ for the isometric embedding operator
  $$
  \Pi_S = \sbm{ I_\cH \\ 0 } \colon \cH \to \cK_{S}
  $$
  and let $V_S$ on $\cK_{S}$ be the operator given by
  $$
  V_S \colon \begin{bmatrix} h \\ f \end{bmatrix} \mapsto \begin{bmatrix} T & 0 \\ D_T & M_z \end{bmatrix}
   \begin{bmatrix} h \\ f \end{bmatrix} = \begin{bmatrix} T h \\ D_T h + z f(z) \end{bmatrix}
 $$
 where $D_T$ in the lower left corner of the block matrix $\sbm{T & 0 \\ D_T & M_z}$ is to be interpreted as the operator
 $D_T \colon \cH \to H^2(\cD_T)$ mapping the vector $h \in \cH$ to the constant function $D_T h$ in $H^2(\cD_T)$.
 Then one can check that $V_S$ is isometric on $\cK_{S}$ and that  $\Pi_S T^* = V_S^* \Pi_S$ (due to the block lower-triangular
 form in the matrix representation of $V_S$).  We conclude that
 $(\Pi_S, V_S)$ is an isometric dilation of $T$, called the {\em Sch\"affer-matrix isometric dilation of $T$}.

 Let us next write $U_S \colon \cK  \to \cK_{S}$ for the unitary identification map
  $$
  U_S = \begin{bmatrix} P_\cH \\  \iota P_{M_+(\cL)} \end{bmatrix}.
  $$
  Then one can check the intertwinings
  $$
   \Pi_S = U_S\,  \iota_{\cH \to \cK}.
  $$
   Furthermore, the calculation
 \begin{align*}
  &    U_S V (h + \sum_{n=0}^\infty V^n \ell_n )   = U_S \left( T h + (V-T) h + \sum_{n=0}^\infty V^{n+1} \ell_n \right) \\
      & \quad \quad  = \begin{bmatrix} T h \\ D_T h + M_z \cdot \sum_{n=0}^\infty (\iota \ell_n) z^n \end{bmatrix}
       = \quad \quad \begin{bmatrix}  T & 0 \\ D_T & M_z \end{bmatrix} U_S ( h + \sum_{n=0}^\infty V^n z^n )
\end{align*}
shows that
\begin{equation}   \label{Schaffer-dil}
U_S V = V_S U_S.
\end{equation}
We conclude that $U_S$ implements a unitary equivalence between the minimal coordinate-free isometric dilation
$(\iota_{\cH \to \cK}, V)$ and the Sch\"affer-matrix isometric dilation $(\Pi_S, V_S)$.  As the coordinate-free isometric dilation
$(\iota_{\cH \to \cK_+}, V)$ is minimal, we conclude that the Sch\"affer-matrix isometric dilation $(\Pi_S, V_S)$ is also minimal.

\subsubsection{The Douglas model for the minimal isometric dilation}   \label{S:Douglas}

  We now show how to use the first decomposition  in \eqref{decom4} for the minimal isometric-dilation space of $T$
to arrive at the Douglas model for the  minimal isometric dilation, as derived by Douglas from first principles in \cite{Doug-Dilation}.

Let us now introduce the  extension of the unitary identification  $\iota_* \colon \cL_* \to \cD_{T^*}$ in part (1) of Proposition \ref{P:dil}
 to the unitary identification
$\biota_* \colon M_+(\cL_*) \to H^2(\cL_*)$ given by
\begin{equation}   \label{biota*}
 \biota_* \colon \sum V^n \ell_{*n} \mapsto \sum_{n=0}^\infty (\iota_* \ell_{*n}) z^n.
\end{equation}
By  part (3) of Proposition \ref{P:dil}  (see formula \eqref{PM+L}, we know that, for $h \in \cH$,
$$
P_{M_+(\cL_*)} h  = \sum_{n=0}^\infty V^n (I - V T^*) T^{*n} h
$$
and hence
\begin{equation}   \label{biota*M+L}
\biota_* P_{M_+(\cL_*)} h = \sum_{n=0}^\infty \iota_* (I - V T^*) T^{*n} h =
\sum_{n=0}^\infty (D_{T^*} T^{*n} h) z^n = : \widehat \cO_{D_{T^*}, T^*}(z) h
\end{equation}
where $\widehat \cO_{D_{T^*}, T^*}$ is the frequency-domain observability operator associated with the state/output linear system
$$
\left\{ \begin{array}{rcl} x(t+1) & = & T^* x(t) \\ y(t) & = & D_{T^*} x(t), \end{array}  \right. t = 0, 1, 2, \dots.
$$
From the construction we see that
$$
   \widehat \cO_{D_{T^*}, T^*} T^* h = (M_z)^*  \widehat \cO_{D_{T^*}, T^*} h.
$$
and from formulas \eqref{proj-norm1} and \eqref{isom-id}  we see that
\begin{equation}  \label{isom-id}
  \| \widehat \cO_{D_{T^*}, T^*} h \|^2_{H^2 (\cD_{T^*})}  = \| P_{M_+(\cL_*)} h \|^2, \quad
  \| Q h \|^2 = \langle Q^2 h, h \rangle = \| P_\cR h \|^2
\end{equation}
where $Q^2  = \operatorname{SOT-lim}_{n \to \infty} T^n T^{*n}$.
Let us define the subspace $\cR_0$ of $\cR$ by
$$
\cR_0 =  \overline{P_\cR \cH}.
$$
Hence we can define an isometric map $\omega_D \colon \cR_0 \to \overline{\operatorname{Ran}}\, Q$ by
action on the dense subset $P_\cR \cH \subset \cR_0$ given by
\begin{equation}  \label{omega}
  \omega_D  \colon P_\cR h \mapsto Q h.
\end{equation}

We now note that $T Q^2 T^* = Q^2$; hence the
formula
\begin{equation}  \label{defX}
X^*Qh = Q T^* h
\end{equation}
defines an isometry on $\overline{\operatorname{Ran}} \, Q$ which (if not already unitary) has a
minimal unitary extension on a space $\cR_D \supset \overline{\operatorname{Ran}} \, Q$ which we denote by $(W^{D})^*$.
A dense subspace of $\cR_D$ is $\bigcup_{n=0}^\infty W_D^{n} \operatorname{Ran} Q$ and then the extension $W^*$ is given densely by
\begin{align}\label{IntofQ}
  W_D^* W_D^n Q h = W_D^{n-1} Q h \text{ for } n \ge 1,  W_D^* Q h = X^* Q h = Q T^* h.
\end{align}
As we are assuming that $(\iota_{\cH \to \cK}, V)$ is a minimal isometric lift of $T$, from part  (6b) of Proposition \ref{P:dil} we
see that the linear manifold $\bigcup_{n=0}^\infty V^n P_\cR \cH$ is dense in $\cR$.
Hence we can extend the map $\omega_D$ given by \eqref{omega} to a unitary map, still denoted as $\omega_D$, from all of
$\cR$ onto all of $\cR_D$ densely defined
according to the formula
\begin{equation}   \label{omega-ext}
  \omega_D \colon V^n P_\cR h \mapsto W_D^n Q h \text{ for } n = 0,1,2,\dots.
\end{equation}
One can check that $\omega_D$ satisfies the intertwining relation
$$
  \omega_D (V|_\cR) = W_D \omega_D.
$$
As $\omega_D$ is  unitary, this relation can equivalently be written as
$$
 \omega_D (V^*|_\cR) = W_D^* \omega_D.
$$

Let us introduce the Hilbert space $\cK_D = \sbm{ H^2(\cD_{T^*}) \\ \cR_D}$ and define an  isometric operator
$V_D$ on $\cK_{D}$ by
\begin{align}\label{Vd}
  V_D = \begin{bmatrix} M_z & 0 \\ 0 & W_D \end{bmatrix}.
\end{align}

There is a canonical isometric embedding operator $\Pi_D \colon \cH \to \cK_D$ given by
\begin{align}\label{PiD}
  \Pi_D \colon h \mapsto \begin{bmatrix} \cO_{D_{T^*}, T^*} h \\ Q h \end{bmatrix}.
\end{align}
Furthermore we have the intertwining relation
$$
    \Pi_D T^* = (V_D)^*  \, \Pi_D.
$$
Thus $(\Pi_D, V_D)$ is an isometric dilation of $T$.

We can now define a map $U_D \colon \cK \to H^2(\cD_{T^*}) \oplus \cR_D$ by
$$
  U_D \colon k  \mapsto \biota_* P_{M_+(\cL_*)} k  \oplus \omega_D P_\cR k
$$
or, in operator form with  column notation, $U_D \colon \cK \to \cK_{D}: = \begin{bmatrix} H^2(\cD_{T^*}) \\ \cR_D \end{bmatrix}$
is given by
\begin{align}\label{Ud}
  U_D = \begin{bmatrix} \biota_* P_{M_+(\cL_*)}  \\ \omega_D P_\cR \end{bmatrix}.
\end{align}
For $k \in \cK$, since $\cK = M_+(\cL_*) \oplus \cR$ by the first decomposition in \eqref{decom4}, we have
$\| k \|^2 = \| P_{M_+(\cL_*)} k \|^2 + \| P_\cR k \|^2$. Since $\biota_* \colon M_+(\cL_*)\to H^2(\cD_{T^*})$ and $\omega_D \colon \cR \to \cR_D$
are unitary maps, we then see that $\Pi_D \colon \cK \to \sbm{ H^2(\cD_{T^*}) \\ \cR_D }$  is unitary.
 Furthermore one can check the intertwinings:
\begin{align}\label{Dintwin}
   \Pi_D = U_D \,  \iota_{\cH \to \cK}, \quad    U_D V = V_D U_D.
\end{align}
Thus $(\Pi_D, V_D)$ and $(\iota_{\cH \to \cK}, V)$ are unitarily equivalent as isometric dilations of $T$.  As $(\iota_{\cH \to \cK}, V)$
is minimal, it follows that $(U_D, V_D)$ is minimal as well.

\subsubsection{The Sz.-Nagy--Foias functional model for the minimal isometric dilation} \label{S:SNFmodel}

The Sz.-Nagy--Foias isometric dilation for a completely nonunitary contraction operator can be derived as follows.
Let $(\iota_{\cH \to \cK}, V)$ be the coordinate-free minimal isometric dilation of the completely nonunitary contraction operator
as in Proposition \ref{P:dil}.   Define the operator $\bTheta \colon M(\cL) \to M(\cL_*)$ as the restricted projection
$$
    \bTheta = P_{M(\cL_*)}|_{M(\cL)}.
$$
Let $\biota_* \colon M(\cL_*) \to L^2(\cD_{T^*})$ and $\biota \colon M(\cL) \to L^2(\cD_T)$ be the Fourier representation operators
\begin{equation}   \label{biota-biota*-2sided}
\biota_* \colon \sum_{n=-\infty}^\infty \cU^n \ell_{*n} \mapsto \sum_{n=-\infty}^\infty \ell_{*n} e^{int}, \quad
\biota \colon \sum_{n=-\infty}^\infty \cU^n \ell_{n} \mapsto \sum_{n=-\infty}^\infty \ell_{n} e^{int},
\end{equation}
i.e., the operator  $\biota_*$ in \eqref{biota*} extended in the natural way to a unitary identification of the bilateral-shift space
 $M(\cL)$ to the $L^2$-space $L^2(\cD_{T^*})$ and similarly for the operator $\biota$ in \eqref{biota}).  Then it is easily checked that
 $$
   \bTheta (\cU|_{M(\cL)}) = (\cU|_{M(\cL_*)}) \bTheta.
 $$
 Let $\widehat \bTheta = \biota_* \bTheta \biota \colon L^2(\cD_T) \to L^2(\cD_{T^*})$.  Then the previous intertwining relation becomes
 the function-space intertwining
 $$
    \widehat \bTheta M_{\zeta} = M_{\zeta} \widehat \bTheta.
 $$
 By a standard result (see e.g.\ \cite{Nagy-Foias}, it follows that $\widehat \bTheta$  is a multiplication operator
 $$
    \widehat \bTheta \colon h(\zeta) \mapsto \Theta(\zeta) \cdot h(\zeta)
 $$
 for a measurable $\cL(\cD_T, \cD_{T^*})$-valued function $\zeta \mapsto \Theta(\zeta)$.  As $\bTheta$ is a restricted projection,
 it follows that $\| \bTheta \| \le 1$, and also $\| M_\Theta \| \le 1$ as an operator from $L^2(\cD_T)$ to $L^2(\cD_{T^*})$, from
 which it follows that $\| \Theta(\zeta) \| \le 1$ for almost all $\zeta$ in the unit circle.  Furthermore, from the second decomposition
 in \eqref{decom3}, we see that $M_\Theta$ maps $H^2(\cD_T)$ into $H^2(\cD_{T^*})$; thus in fact $\Theta$ is a contractive
 $H^\infty$-function with values in $\cL(\cD_T, \cD_{T^*})$---known as the {\em Sz.-Nagy-Foias characteristic function} of $T$, i.e.,
 \begin{align}\label{charcfunction}
 \Theta(z)=\Theta_T(z)=-T+zD_{T^*}(I_{\cH}-zT^*)^{-1}D_{T}.
 \end{align}

 Suppose next that $k \in \cR$ has the form $k = P_\cR \ell$ for some $\ell \in M(\cL)$.  Then
 $$
 \| k \|^2  = \| P_\cR \ell \|^2 = \| \ell \|^2 - \| \bTheta \ell \|^2 =
 \| \biota \ell \|^2 - \| \Theta_T \cdot  \biota \ell \|^2 = \| \Delta_{T} \cdot \biota \ell \|^2
 $$
 where we let $\Delta_{T}$ be the $\cD_T$-valued operator function on the unit circle ${\mathbb T}$ given by
 $$
   \Delta_{T}(\zeta) := (I - \Theta_T(\zeta)^* \Theta_T(\zeta))^{1/2}.
 $$
By part (7) of Proposition \ref{P:dil} we know that the space $(I - P_{M(\cL_*)}) M(\cL) = P_\cR M(\cL)$ is dense in $\cR$.
Hence we can define a unitary map $\omega_{NF}$ from $\cR$ to $\overline{ \Delta_{T}L^2(\cD_T) }$
densely defined on $P_\cR M(\cL)$ by
\begin{equation}   \label{omegaNF}
 \omega_{NF} \colon P_\cR \ell \mapsto \Delta_{T} \cdot \biota \ell.
\end{equation}
From this formula we can read off the validity of the intertwining relation
\begin{equation}   \label{intertwine1}
  \omega_{NF} (V|_\cR) = M_{\zeta}\, \omega_{NF}.
\end{equation}

We now can define a unitary identification map $U_{NF}$ from $\cK = M_+(\cL_*) \oplus \cR$ to $\cK_{NF} = \sbm{ H^2(\cD_{T^*}) \\
\overline{ \Delta_{T}L^2(\cD_T) } }$  by
\begin{align}\label{Unf}
  U_{NF} k = \begin{bmatrix} \biota_* P_{M_+(\cL_*)} k \\ \omega_{NF} P_\cR k \end{bmatrix}.
\end{align}
Since $\biota_*$ is unitary from $M_+(\cL_*)$ to $H^2(\cD_{T^*})$, $\omega_{NF}$ is unitary from $\cR$ to
$\overline{ \Delta_{T}L^2(\cD_T) }$, and $\cK$ has the internal orthogonal decomposition
$\cK = M_+(\cL_*) \oplus \cR$, we see that $U_{NF}$ so defined is unitary from $\cK$ onto
$\cK_{NF}$.  Observing the intertwining relation $M_z \biota_*|_{M_+((\cL_*)} = \biota V|_{M_+(\cL_*)}$
and recalling \eqref{intertwine1}, we arrive at the intertwining relation
\begin{equation}   \label{NFintwin}
  U_{NF} V = V_{NF} U_{NF}
\end{equation}
where we set $V_{NF}$ equal to the isometric operator on $\cK_{NF}$ given by
\begin{align}\label{Vnf}
  V_{NF} = \begin{bmatrix} M_z & 0 \\ 0 & M_{\zeta} \end{bmatrix}.
\end{align}
Define the isometric embedding $\Pi_{NF}$ of $\cH$ into $\cK_{NF}$ as
\begin{align}\label{Unf&Pinf}
  \Pi_{NF} = U_{NF} \iota_{\cH \to \cK}.
\end{align}
The space $U_{NF} M_+(\cL)$ can be identified explicitly as follows:
for $\ell \in M_+(\cL)$,
$$ U_{NF} \ell = \begin{bmatrix}  \Theta_T \cdot \biota \ell \\  \omega_{NF} P_\cR \ell \end{bmatrix} =
\begin{bmatrix} \Theta_T \cdot \biota \ell \\ \Delta_{T} \cdot \biota \ell \end{bmatrix}
$$
and hence the space $\cH_{NF} :=\Pi_{NF}\cH = U_{NF} \cH$ is given by
\begin{align}\label{HNF}
  \cH_{NF} = \begin{bmatrix} H^2(\cD_{T^*}) \\ \overline{ \Delta_{T}L^2(\cD_T)} \end{bmatrix}
  \ominus \begin{bmatrix} \Theta_T \\ \Delta_{T} \end{bmatrix} \cdot H^2(\cD_T).
\end{align}
Note that the subspace $U_{NF} M_+(\cL) = \sbm{ \Theta_T \\ \Delta_{T} } \cdot H^2(\cD_T)$ is invariant for
$V_{NF}$ and hence $\cH_{NF}$ is invariant for $V_{NF}^*$.
Rewrite \eqref{NFintwin} in the form
$$
  U_{NF} V^* = V_{NF}^* U_{NF}
$$
and restrict this identity to $\cH$ to arrive at
$$
   \Pi_{NF} T^* = V_{NF}^* \Pi_{NF}.
$$

We now have all the pieces needed to conclude that $(\Pi_{NF}, U_{NF})$ is an isometric lift of $T$ (the
{\em Sz.-Nagy--Foias functional-model isometric lift} of $T$).  Furthermore
the operator $U_{NF} \colon \cK \to \cK_{NF}$ implements a unitary equivalence of the Sz.-Nagy--Foias isometric lift
$(\Pi_{NF}, U_{NF})$ with the coordinate-free minimal isometric lift $( \iota_{\cH \to \cK}, V)$, and hence
the Sz.-Nagy--Foias functional-model isometric lift $(\Pi_{NF}, V_{NF})$ is also minimal.

It is known that any two minimal isometric dilations of a contraction operator are unitarily equivalent. It also follows that the unitary operator involved in the equivalence of two minimal dilations is unique. We conclude this section by finding the explicit unitary involved in the equivalence of the minimal isometric dilations $(\Pi_{D},V_{D})$ and $(\Pi_{NF},V_{NF})$. Let $\omega_D$ and $\omega_{NF}$ be the unitaries defined in (\ref{omega-ext}) and (\ref{omegaNF}), respectively. Define the unitary
\begin{align}\label{DNFintwin}
  U_{\text{min}}:=I_{H^2(\cD_{T^*})}\oplus\omega_{NF}\omega_D^*:H^2(\cD_{T^*})\oplus\cR_D\to H^2(\cD_{T^*})\oplus \overline{\Delta_{T}L^2(\cD_T)}.
\end{align}From equations (\ref{Dintwin}) and (\ref{NFintwin}), we see that the following diagram is commutative:
$$\begin{tikzcd}
(\iota_{\cH \to \cK}, V, M_+(\cL_*)\oplus \cR)\arrow{r}{\biota_*\oplus\omega_D }  \arrow{rd}{\biota_*\oplus\omega_{NF}}
  & (\Pi_D, V_D, H^2(\cD_{T^*})\oplus\cR_D) \arrow{d}{U_{\text{min}}} \\
    &  (\Pi_{NF},V_{NF}, H^2(\cD_{T^*})\oplus\overline{\Delta_{T}L^2(\cD_T)}).
\end{tikzcd}$$
Therefore the unitary $U_{\text{min}}$ intertwines the Douglas model and the Sz.-Nagy--Foias model of the minimal isometric dilation, i.e.,
\begin{align}\label{DNFint}
  U_{\text{min}} V_{D}=V_{NF}U_{\text{min}}
\end{align}where $V_D$ and $V_{NF}$ are as defined in (\ref{Vd}) and (\ref{Vnf}), respectively. It also follows from the first equation in (\ref{Dintwin}) and (\ref{Unf&Pinf}) that
\begin{align}\label{uniqDNF}
U_{\text{min}}\Pi_D=\Pi_{NF}.
\end{align}
This unitary identification will be used in what follows.

\subsubsection{Epilogue: the de Branges-Rovnyak functional model}  \label{S:epilogue}
The term {\em functional model} for the Sch\"affer model (Section \ref{S:Schaffer}) as well for the
Douglas model (Section \ref{S:Douglas}) is a bit loose since the Sch\"affer isometric-lift space
$\cK_S = \cH \oplus H^2(\cD_T)$ has only its second component $H^2(\cD_T)$ as a space of functions
while the first component remains an abstract Hilbert space, and similarly the Douglas isometric-lift space
$\cK_D = H^2(\cD_{T^*}) \oplus \cR_D$ has only its first component $H^2(\cD_{T^*})$ equal to a space of functions
while the second component $\cR_D$ is an abstract Hilbert space.  On the other hand the Sz.-Nagy--Foias
isometric lift space $\cK_{NF} = H^2(\cD_{T^*}) \oplus \overline{\Delta_T L^2(\cD_T)}$  in Section
 \ref{S:SNFmodel}  has both components equal to spaces of functions, the first  component consisting of
 holomorphic functions with the second component consisting of measurable functions.  Indeed, note that
 the transformation $\omega_{NF} \omega_D^{-1} \colon \cR_D \to \overline{\Delta_T L^2(\cD_T)}$ can be viewed
 as a type of Fourier transform implementing a concrete spectral theory for the unitary operator $W_D$
in the following sense:  $\omega_{NF} \omega_D^{-1}$ converts the abstract Hilbert space $\cR_D$
to the space of measurable operator-valued functions $ \overline{\Delta_T L^2(\cD_T)}$ on the unit circle
while at the same time converting the abstract unitary operator $W_D$
to the operator $M_{\zeta}$ of multiplication by the coordinate function on the function space
$ \overline{\Delta_T L^2(\cD_T)}$.
There is yet another functional model space for c.n.u.\ contraction operators, namely that of de Branges-Rovnyak
(see \cite{dBR1, dBR2} for the original work and \cite{BB, BC, BK, NV1, NV2} for later treatments) consisting
of two (in general) coupled components, each of which is a holomorphic function on the unit disk ${\mathbb D}$,
so the study of c.n.u.\ contraction operators is again reduced to pure holomorphic function theory rather than
a hybrid of holomorphic- and measurable-function theory.
We recommend the survey article \cite{BB} as an entrance to this rich topic.

\section{And\^o dilations: two coordinate-free constructions}  \label{S:c}

The goal of this section is to introduce coordinate-free and functional-model formulations for an And\^o isometric
lift for a pair of commuting contractions $(T_1, T_2)$.   However we first add some additional information concerning
An\^o tuples.

\subsection{More on And\^o tuples}
A piece of unfinished business from the Introduction is to complete the discussion of uniqueness of the And\^o tuple
for a given commutative contractive pair $(T_1, T_2)$ (see Remark \ref{R:Ando-tuple-unique}).  We shall also
need more coordinate-free versions of an An

There we saw that the And\^o tuple is unique if and only if both $T_1 \cdot T_2$ and $T_2 \cdot T_1$ is a regular factorization.
The question is whether it suffices to assume that one of these factorizations is regular.  This issue is resolved
by the following result.

\begin{thm}  \label{T:left-right}  Let $(T_1, T_2)$ be a commutative, contractive pair of contraction operators
on a Hilbert space $\cH$.
\begin{enumerate}
\item Assume that all defect spaces $\cD_{T}$, $\cD_{T_1}$,  $\cD_{T_2}$ are finite-dimensional. Then
$T_1 \cdot T_2$ is a regular factorization if and only if $T_2 \cdot T_1$ is a regular factorization.

\item In the infinite-dimensional setting, it is possible for one of the factorizations $T_1 \cdot T_2$ to be regular
while the other $T_2 \cdot T_1$ is not regular.
\end{enumerate}
\end{thm}

\begin{proof}[Proof of (1):]
Let us write $\Lambda_r$ for the counterpart of $\Lambda$ when the roles of $T_1$ and $T_2$ are interchanged:
$$
\Lambda_r \colon D_T h \mapsto D_{T_1} T_2 h \oplus D_{T_2} h.
$$
Then by the same computation \eqref{id1} with the role of the indices $(1,2)$ interchanged, we see that
$\Lambda_r$ is an isometry from $\cD_T$ onto $\cR(U_0)$.  Moreover,
 Ii $T_1 \cdot T_2$ is a regular factorization, then
\begin{align*}
 \dim \cD_T & =  \dim \cD(U_0) \text{ (since $\Lambda$ is an isometry from $\cD_T$ onto $\cD_{U_0}$)} \\
 & = \dim (\cD_{T_1} \oplus \cD_{T_2}) \text{ (by regularity of $T_1 \cdot T_2$)} \\
 & \ge \dim \cR_{U_0} \text{ (since $\cR_{U_0} \subset \cD_{T_1} \oplus \cD_{T_2}$) }\\
 & = \dim \cD_T \text{ (since $\Lambda_r$ is an isometry from $\cD_T$ onto $\cR_{U_0}$).}
  \end{align*}
 Thus we see that necessarily the inequality in line 3 must be an equality.  If we are in the finite-dimensional setting
 ($\cD_T$, $\cD_{T_1}$, $\cD_{T_2}$ all finite-dimensional), we necessarily then have
 $\cR_{U_0} = \cD_{T_1} \oplus \cD_{T_2}$ which is the statement that the factorization $T_2 \cdot T_1$
 is also regular.

 \smallskip

 \noindent
 {\em Proof of (2):}
Let $(T_1, T_2)$ be the following pair of contractions on
$\cH: = H^2 \oplus H^2$:
$$
  (T_1, T_2) =\left( \begin{bmatrix} 0 & 0 \\ I & 0 \end{bmatrix}, \begin{bmatrix}T_z & 0 \\ 0 & T_z \end{bmatrix} \right)
$$
where $T_z$ is the Toeplitz operator $f(z) \mapsto z f(z)$ on $H^2$.  Note that
$$
 T_1 T_2 = \begin{bmatrix} 0 & 0 \\ T_z & 0 \end{bmatrix} = T_2 T_1 := T.
$$
Then
\begin{align*}
& D_T^2 = \begin{bmatrix} I_{H^2} & 0 \\ 0 & I_{H^2} \end{bmatrix}  -
\begin{bmatrix} 0 & T_z^* \\ 0 & 0 \end{bmatrix} \begin{bmatrix} 0 & 0 \\ T_z & 0 \end{bmatrix}
 = \begin{bmatrix}  0 & 0 \\ 0 & I_{H^2} \end{bmatrix} = D_T, \quad \cD_T = \begin{bmatrix} 0 \\ H^2 \end{bmatrix} \\
& D_{T_1}^2 = \begin{bmatrix} I & 0 \\ 0 & I \end{bmatrix} - \begin{bmatrix} 0 & I \\ 0 & 0 \end{bmatrix}
\begin{bmatrix} 0 & 0 \\ I & 0 \end{bmatrix} = \begin{bmatrix} 0 & 0 \\ 0 & I \end{bmatrix} = D_{T_1}, \quad
\cD_{T_1} = \begin{bmatrix}  0 \\ H^2 \end{bmatrix}, \\
& D_{T_2}^2 = \begin{bmatrix} I & 0 \\ 0 & I \end{bmatrix} - \begin{bmatrix} T_z^* & 0 \\ 0 & T_z^* \end{bmatrix}
\begin{bmatrix} T_z & 0 \\ 0 & T_z \end{bmatrix} = \begin{bmatrix} 0 & 0 \\ 0 & 0 \end{bmatrix}, \quad
\cD_{T_2} = \{0 \},
\end{align*}
and  $\cD_{T_1} \oplus \cD_{T_2} = \sbm{ 0 \\ H^2}$.

We let $\Lambda$ be the map associated with the factorization $T_1 \cdot T_2$ given by \eqref{defLambda}
while $\Lambda_r$ is the same map associated with the factorization
$T_2 \cdot T_1$ (i.e., \eqref{defLambda} but with the indices and then the components interchanged).
For $h = \sbm{h_1 \\ h_2 } \in \cH = \sbm{ H^2 \\ H^2}$
we compute
$$
\Lambda \colon D_T h = \sbm{ 0 \\ h_2 } \mapsto
D_{T_1} T_2 h \oplus D_{T_2} h = \sbm{ 0 \\ T_z h}
$$
and we conclude that
$$
  \operatorname{Ran} \Lambda = \sbm{ 0 \\ zH^2} \underset{\ne}\subset \sbm{ 0 \\ H^2 } = \cD_{T_1} \oplus
  \cD_{T_2}
$$
implying that $T_1 \cdot T_2$ is not a regular factorization.  On the other hand,
$$
\Lambda_r \colon D_T h = \sbm{ 0 \\ h_2 } \mapsto
D_{T_1} h \oplus D_{T_2} T_1 h = \sbm{ 0 \\ h_2} \oplus 0
$$
from which we see that
$$
\operatorname{Ran} \Lambda_r  =  \sbm{ 0 \\ H^2} = \cD_{T_1} \oplus \cD_{T_2}
$$
implying that the factorization $T_2 \cdot T_1$ is regular.
\end{proof}

 We shall also need the following extension of parts of
Proposition \ref{P:dil} to the setting of a commutative contractive operator-pairs.

\begin{proposition}\label{P:LandL*}
Let $(T_1,T_2)$ be a pair of commuting contractions on a Hilbert space $\mathcal{H}$ and let $(V_1,V_2)$ be
an And\^o commutative  isometric dilation for $(T_1, T_2)$ acting on $\mathcal{K} \supset \cH$. Define the following
subspaces of $\cK$:
\begin{align}\label{L}
&\mathcal{L}_1=\overline{(V_1-T_1)\mathcal{H}},\;\mathcal{L}_2=\overline{(V_2-T_2)\mathcal{H}},\\\label{L*}
&\mathcal{L}_{1*}=\overline{(I-V_1T_1^*)\mathcal{H}},\;\mathcal{L}_{2*}=\overline{(I-V_2T_2^*)\mathcal{H}}.
\end{align}
Then for each $k=1,2$, the maps
$$ \iota_{k*} \colon \cL_{k*} \to \cD_{T_k^*}\text{ and }\iota_k \colon \cL_k \to \cD_{T_k}$$ given densely by
$$
  \iota_{k*} \colon (I - V_k T_k^*) h \mapsto D_{T_k^*} h, \text{ and }
  \iota_k \colon (V_k - T_k) h \mapsto D_{T_k} h
 $$
 for $h \in \cH$ extend to define unitary identification maps from $\cL_{k*}$ onto $\cD_{T_k^*}$ and from $\cL_k$ onto $\cD_{T_k}$ respectively.
\end{proposition}

\begin{proof}  This amounts to statement (1) of Proposition \ref{P:dil} with the single contraction operator
$T$ replaced by the pair of contraction operators $(T_1, T_2)$.
The proof comes down to the  simple inner product computation, for $k=1,2$ and $h\in \mathcal{H}$,
\begin{eqnarray*}
\|(I - V_k T_k^*) h\|^2 &=& \langle (I - V_k T_k^*) h, (I - V_k T_k^*) h \rangle \\
&=&\|h\|^2-\langle V_kT_k^*h,h \rangle -\langle h, V_kT_k^*h \rangle+\|T_k^*h\|^2\\
&=&\|h\|^2-\|T_k^*h\|^2=\|D_{T_k^*}h\|^2.
\end{eqnarray*}
A similar computation shows that $\|(V_k - T_k) h\|^2=\|D_{T_k}h\|^2$.
\end{proof}

We next discuss the following "coordinate-free" adjustment of the
notion of And\^o tuple discussed in the Introduction (see Definition \ref{D:Ando-tuple}).
Using part (1) of Proposition \ref{P:dil} and Proposition \ref{P:LandL*}, we observe that for $h\in\mathcal{H}$,
\begin{align}
\nonumber\|(V-T)h\|^2  =\|D_Th\|^2&=&\|D_{T_1}T_2h\|^2+\|D_{T_2}h\|^2=\|(V_1-T_1)T_2h\|^2+\|(V_2-T_2)h\|^2\\\label{IsoUniGen}
&=&\|D_{T_1}h\|^2+\|D_{T_2}T_1h\|^2=\|(V_1-T_1)h\|^2+\|(V_2-T_2)T_1h\|^2,
\end{align}
which implies that $\Lambda_c:\cL\to\cL_1\oplus\cL_2$ (here the subscript $c$ suggests {\em coordinate-free}) densely defined by
\begin{eqnarray}\label{Lambda-c}
\Lambda_c:(V-T)h\mapsto (V_1-T_1)T_2h\oplus (V_2-T_2)h
\end{eqnarray}
is an isometry.   Let us introduce the notation
\begin{align*}
&  \cD_{U_{c0}} = \operatorname{Ran} \Lambda_c = \operatorname{clos.}
\{ (V_1-T_1)T_2h\oplus (V_2-T_2)h \colon h \in \cH\},  \\
& \cR_{U_{c0}} = \operatorname{clos.} \{(V_1 - T_1) h \oplus (V_2 - T_2) T_1 h \colon
h \in \cH\}.
\end{align*}
Then a consequence of the equality between the fourth and sixth terms in the chain of equalities
\eqref{IsoUniGen} is there is a unitary map $U_{c0} \colon \cD_{U_{c0}} \to \cR{U_{c0}}$ densely defined by
\begin{eqnarray}\label{U-c}
U_{c0} \colon (V_1-T_1)T_2h\oplus (V_2-T_2)h\mapsto (V_1-T_1)h\oplus (V_2-T_2)T_1 h.
\end{eqnarray}
We add an infinite-dimensional summand  $\ell^2$ (if necessary) to the external direct sum
$\cL_1\oplus\cL_2$ to ensure that $U_{c0}$ can be extended to a unitary map $U_c$ on the
spaces $\cF_c: =\cL_1 \oplus \cL_2$ (in case the $\ell^2$-summand is not required) or $\cF_c: =
\cL_1 \oplus \cL_2 \oplus \ell^2$ (otherwise).  Let $P_c$ denote the orthogonal projection of $\cF_c$ onto $\cL_1$:
\begin{align}
& P_c \colon f \oplus g \mapsto f \oplus 0 \text{ if } \cF_c = \cL_1 \oplus \cL_2,  \notag \\
& P_c \colon f \oplus g \oplus h  \mapsto f \oplus 0 \oplus 0 \text{ if } \cF_c = \cL_1 \oplus \cL_2 \oplus \ell^2.
\label{Pc}
\end{align}

\begin{definition}  \label{D:Ando-c}
{\em{The tuple $(\cF_c,\Lambda_c,P_c,U_c)$ of Hilbert spaces and operators arising as above from a commutative
contractive pair $(T_1, T_2)$ having commutative isometric dilation $(V_1, V_2)$ as above will be called a{ \em{coordinate-free And\^o tuple}} for $(T_1,T_2)$.}}
\end{definition}

It will also be useful to spell out the notion of {\em coordinate-free And\^o tuple for $(T_1^*, T_2^*)$}.
Via a computation similar to what was done in \eqref{IsoUniGen} we see that
\begin{align}
\|(I-VT^*)h\|^2 &=\|(I-V_1T_1^*)T_2^*h\|^2+\|(I-V_2T_2^*)h\|^2 \notag  \\
&=\|(I-V_1T_1^*)h\|^2+\|(I-V_2T_2^*)T_1^*h\|^2.
 \label{IsoUniGen*}
\end{align}
We conclude in particular that there is an isometry $\Lambda_{c*}:\cL_*\to \cL_{1*}\oplus\cL_{2*}$
densely defined as
\begin{equation} \label{Gamma-c*}
\Lambda_{c*}:(I-VT^*)h\mapsto (I-V_1T_1^*)T_2^*h\oplus (I-V_2T_2^*) h.
\end{equation}
Introduce the notation
\begin{align*}
& \cD_{U_{c*0}}  = \operatorname{Ran} \Lambda_{c*} =  \operatorname{clos} \{ (I-V_1T_1^*)T_2^*h
\oplus (I-V_2T_2^*) h \colon h \in \cH\}, \\
& \cR_{U_{c*0}} = \operatorname{clos} \{ (I - V_1 T_1^*) h \oplus (I - V_2 T_2^*) T_1 h \colon h \in \cH\}.
\end{align*}
Then the last equality in the chain of equalities \eqref{IsoUniGen*} implies that
there is unitary $U_{c*0} \colon \cD_{U_{c*0}} \to \cR_{U_{c*0}}$ densely defined by
$$
U_{c*0} \colon (I-V_1T_1^*)T_2^*h \oplus (I-V_2T_2^*) h  \mapsto
(I - V_1 T_1^*) h \oplus (I - V_2 T_2^*) T_1 h.
$$
 Add an infinite dimensional summand to $\cL_{1*}\oplus\cL_{2*}$ if necessary to guarantee that the
isometry $U_{c0*}$ viewed as an operator from $\overline{\operatorname{Ran}}\, \Lambda_{c*}$ into $\cF_{c*}$
can be extended to a unitary operator, denotes as $U_{c*}$, on $\cF_{c*}$, where
$\cF_{c*}$ is either $\cL_{1*} \oplus \cL_{2*}$, in case $\dim (\cD_{U_{c*0}})^\perp =
\dim (\cR_{U_{c*0}})^\perp$ (viewed as subspaces of the ambient space $\cL_{1*} \oplus \cL_{2*}$,
 and $\cF_{c*} =
\cL_{1*} \oplus \cL_{2*} \oplus \ell_+^2$ otherwise.
We fix a choice of such a unitary extension  $U_{c*}$ of the isometry \eqref{Gamma-c*} and denote it also as $U_{c*}$.
We denote by
$P_{c*}$ the orthogonal projection of $\cF_{c*}$ onto its first component, given by
\begin{align*}
& P_{c*} \colon f \oplus g \mapsto f \oplus 0 \text{ in case } \cF_{c*} = \cL_{1*} \oplus \cL_{2*}, \\
& P_{c*} \colon f \oplus g \oplus h \mapsto f \oplus 0 \oplus 0 \text{ in case }
\cF_{c*} = \cL_{1*} \oplus \cL_{2*} \oplus \ell^2.
\end{align*}
We then make the following formal definition.

\begin{definition}  \label{D:Ando-c*}
{\rm Let $(T_1, T_2)$ be a commutative pair of contractions on $\cH$.
Any collection of spaces $(\cF_{c*}, \Lambda_{c*},   P_{c*}, U_{c*} )$ arising via the construction as described in the
previous paragraph will be called a {\em coordinate-free And\^o tuple for $(T_1^*, T_2^*)$}.}
\end{definition}

\begin{remark}   \label{R:Ando-tuple}
{\rm  We point out here that Remark \ref{R:Ando-tuple-unique}  applies equally well to the coordinate-free
And\^o tuples in Definitions \ref{D:Ando-c} and \ref{D:Ando-c*}.   Each construction of an And\^o tuple
 involves an arbitrary choice of unitary extension
of an in general only partially defined isometry ($U_{0}$, $U_{c0}$,and $U_{c0*}$ respectively).
Hence we cannot expect any result to the effect that
And\^o tuples are unique up to a unitary equivalence connecting the various operators and spaces except
in the case when  (or $\cD_{U_0}$ and $\cR_{U_0}$ is the whole space $\cD_{T_1} \oplus \cD_{T_2}$
(respectively, $\cD_{U_{c0}}$ and $\cR_{U_{c0}}$ is the whole space $\cL_1 \oplus \cL_2$,
respectively $\cD_{U_{c*0}}$ and $\cR_{U_{c*0}}$ is the whole space $\cL_{1*} \oplus \cL_{2*}$).
One can use the unitary identification maps $\iota_j$ and $\iota_{j*}$ ($j = 1, 2$) to see that this holds
for one of the three cases (i.e., Definitions \ref{D:Ando-tuple}, \ref{D:Ando-c}, and \ref{D:Ando-c*})
if and only if holds for the all the cases.  As observed in Remark \ref{R:Ando-tuple-unique}, this happens if and only if
both $T = T_1 \cdot T_2$ and $T = T_2 \cdot T_1$ are regular factorizations.
}\end{remark}

\subsection{The  Sch\"affer coordinate-free model for an And\^o dilation}
Let $(T_1,T_2)$ be a pair of commuting contractions on a Hilbert space $\mathcal{H}$ and $(V_1,V_2)$ be an isometric dilation of $(T_1,T_2)$ acting on $\mathcal{K}$. Note that it is highly unlikely for the lifting $V=V_1V_2$ of $T=T_1T_2$ to be minimal, because the space $\mathcal{K}$ contains
$$
\overline{\text {span}}\{V_1^mV_2^n h:h\in\mathcal{H}\text{ and }m,n\geq 0 \},
$$which, in general, is bigger than the resulting space when we choose $m=n$ above. Therefore by part (6) of Proposition \ref{P:dil},
the dilation space $\mathcal{K}$ should, in general, be of the form
\begin{eqnarray}\label{CFdilspace}
\cH \oplus M_+(\cL)  \oplus \cR_0,
\end{eqnarray}
where $\cL = \overline{(V- T) \cH}$ and for some non-zero $\mathcal{R}_0$. Unlike the
single-variable case, there is no canonical way of constructing a minimal isometric dilation of a pair $(T_1,T_2)$ of commuting contractions. As a consequence, construction of an And\^o dilation requires
a non-zero and noncanonical choice of  subspace $\mathcal{R}_0$.
We  fix a choice of coordinate-free And\^o tuple
$(\cF_c, \Lambda_c, P_c, U_c)$ for $(T_1, T_2)$ (see Definition \ref{D:Ando-c}) and then choose
\begin{eqnarray}\label{R0}
\cR_0=\ell^2_+(\cF_c\ominus\Lambda_c\cL).
\end{eqnarray}
Consider the following identification of the space in (\ref{CFdilspace}):
\begin{align}\label{CFLspace}
\cK_c:=\cH \oplus \ell^2_+(\Lambda_c \cL)  \oplus \cR_0=\cH \oplus \ell^2_+(\Lambda_c \cL)  \oplus  \ell^2_+(\cF_c\ominus\Lambda_c\cL)=\cH\oplus \ell^2_+(\cF_c).
\end{align}

Below we show that this choice of $\cR_0$ works, i.e., it is possible to find an And\^o dilation on a
space of the form (\ref{CFdilspace}) with $\cR_0$ as in (\ref{R0}).
Let $S$ denote the forward shift on $\ell^2$. Then $S\otimes I_{\cF_c}$ is the forward shift on $\ell^2_+(\cF_c)
\cong \ell^2_+ \otimes \cF_{c}$. Define $\tilde{V}$ on $\cK_c$ by the following $2\times2$ block operator matrix:
\begin{eqnarray}\label{tildeV}
\tilde{V}:=    \begin{bmatrix}
      T & 0 \\  \Lambda(V-T) & S\otimes I_{\cF_c}\\ \end{bmatrix} \colon
     \begin{bmatrix} \cH  \\ \ell^2_+(\cF_c) \end{bmatrix} \to  \begin{bmatrix} \cH \\ \ell^2_+(\cF_c) \end{bmatrix}.
\end{eqnarray}
A matrix computation verifies that $\tilde{V}$ is an isometry.

We next find a commuting pair $(\tilde{V}_1,\tilde{V}_2)$ of isometries on $\cK_c$ such that $(\tilde{V}_1, \tilde{V}_2)$ is a lift of
$(T_1,T_2)$ and $\tilde{V}_1\tilde{V}_2=\tilde{V}$. The requirement that  $(\tilde{V}_1,\tilde{V}_2)$ to be a lift of $(T_1,T_2)$ forces $(\tilde{V_1}, \tilde{V}_2)$ to have  the form
$$
\tilde{V}_1=\begin{bmatrix}
      T_1 & 0 \\
      C_1 & D_1 \\
     \end{bmatrix}
\text{ and }
\tilde{V}_2=  \begin{bmatrix}
      T_2 & 0 \\
      C_2 & D_2 \\    \end{bmatrix}.
$$
From the fact that the pair $(D_1,D_2)$ is the restriction of $(\tilde{V}_1,\tilde{V}_2)$ to ${\ell^2_+(\cF_c)}$, we see that $(D_1, D_2)$ is a commuting pair of isometries.
Moreover, the condition that  $\tilde{V}_1\tilde{V}_2=\tilde{V}$ forces $(D_1,D_2)$ to satisfy $D_1D_2=S\otimes I_{\cF_c}$. Hence by the result of Berger-Coburn-Lebow \cite{B-C-L}, there exists a projection $P$ and a unitary $U$ on $\cF_c$ such that
$$
(D_1,D_2)=(I_{\ell^2_+}\otimes P^\perp U+S\otimes PU,I_{\ell^2_+}\otimes U^*P+S\otimes U^*P^\perp ).$$
Since $\tilde{V}_1\tilde{V}_2=\tilde{V}$, we have $\tilde{V}_1=\tilde{V}_2^*\tilde{V}$ and $\tilde{V}_2=\tilde{V}_1^*\tilde{V}$, which then give
\begin{align}\label{1*2}
\begin{cases}
\text{(a) } C_1=(I_{\ell^2_+}\otimes U^*P+S\otimes U^*P^\perp)^*\Lambda_c(V-T)=PU\Lambda_c(V-T),  \\
\text{(b) } C_2=(I_{\ell^2_+}\otimes P^\perp U+S\otimes PU)^*\Lambda_c(V-T)=U^*P^\perp\Lambda_c(V-T), \\
\text{(c) } T_1-T_2^*T=(V-T)^*\Lambda_c^*P^\perp U\Lambda_c(V-T),  \text{ and} \\
\text{(d) } T_2-T_1^*T=(V-T)^*\Lambda_c^*U^*P\Lambda_c(V-T).
\end{cases}
\end{align}
We do not know if one can embed an arbitrary projection $P$ and unitary $U$ into an And\^o dilation,
but we show in the following theorem that if  any $(P,U)$ of the form $(P_c, U_c)$ with $P_c$ and $U_c$
as in the coordinate-free And\^o tuple can be embedded in an And\^o dilation for $(T_1, T_2)$.
%
%
%

\begin{thm}
For a pair $(T_1,T_2)$ of commuting contractions, let $(\cF_c,\Lambda_c,P_c,U_c)$ be a coordinate-free And\^o
tuple for $(T_1,T_2)$. The the pair $(V_{1c},V_{2c})$ defined on $\cH\oplus \ell^2_+(\cF_c)$ by
\begin{align}\label{CFSchaf}
\begin{cases}
V_{1c}&= \begin{bmatrix}
      T_1 & 0 \\
      P_cU_c\Lambda_c(V-T) & I_{\ell^2_+}\otimes P_c^\perp U_c+S\otimes P_cU_c  \end{bmatrix}  \text{ and}\\
V_{2c}&= \begin{bmatrix}
      T_2 & 0 \\
      U_c^*P_c^\perp\Lambda_c(V-T) &I_{\ell^2_+}\otimes U_c^*P_c+S\otimes U_c^*P_c^\perp \end{bmatrix}
   \end{cases}
\end{align}
is an isometric dilation of $(T_1,T_2)$.
\end{thm}

\begin{proof}
The pair $(V_{1c},V_{2c})$ is clearly a dilation of $(T_1,T_2)$ because it is a co-extension. We first show that the $(2,1)$-entry in the matrix representation of both $V_{1c}V_{2c}$ and $V_{2c}V_{1c}$ is
$\Lambda_c (V-T)$, i.e.,
\begin{align*}
&P_cU_c\Lambda_c (V-T)T_2+(I_{\ell^2_+}\otimes P_c^\perp U_c+S\otimes P_cU_c)U_c^*
P_c^\perp\Lambda_c (V-T)\\
=&U_c^*P_c^\perp\Lambda_c (V-T)T_1+(I_{\ell^2_+}\otimes U_c^*P_c+S\otimes U_c^*P_c^\perp)
P_cU_c\Lambda_c (V-T)=\Lambda_c (V-T).
\end{align*}
This is true because for all $h\in \mathcal H$ we see that
\begin{align*}
&P_cU_c\Lambda_c (V-T)T_2+(I_{\ell^2_+}\otimes P_c^\perp U_c+S\otimes P_cU_c)U_c^*P_c^\perp\Lambda_c (V-T)h\\
& \quad =P_cU_c\Lambda_c (V-T)T_2+P_c^\perp \Lambda_c (V-T)h\\
& \quad = \Lambda_c (V-T)h\\
&  \quad = U_c^*U_c((V_1-T_1)T_2h\oplus (V_2-T_2)h)\\
& \quad = U_c^*((V_1-T_1)h\oplus 0)+U_c^*(0\oplus (V_2-T_2)T_1h)\\
& \quad = U_c^*P_cU_c\Lambda_c (V-T)h+U_c^*P_c^\perp\Lambda_c (V-T)T_1h\\
& \quad = U_c^*P_c^\perp\Lambda_c (V-T)T_1h+(I_{\ell^2_+}\otimes U_c^*P_c+S\otimes U_c^*P_c^\perp) P_cU_c\Lambda_c (V-T)h.
\end{align*}
This proves the commutativity part because
\begin{align*}
&(I_{\ell^2_+}\otimes P_c^\perp U_c+S\otimes P_cU_c)(I_{\ell^2_+}\otimes U_c^*P_c+S\otimes U_c^*P_c^\perp)\\
=&(I_{\ell^2_+}\otimes U_c^*P_c+S\otimes U_c^*P_c^\perp)( I_{\ell^2_+}\otimes P_c^\perp U_c+S\otimes P_cU_c)=S\otimes I_{\cF_c}.
\end{align*}

It remains to show that $V_{1c}$ and $V_{2c}$ are isometries.
A simple matrix computation shows that $V_{1c}$ is an isometry if and only if the following equalities hold:
\begin{align}
T_1^*T_1+(V-T)\Lambda_c^*(U_c^*P_c)(U_c^*P_c)^*\Lambda_c (V-T)=I_{\mathcal H}
\label{proof-2-1}\\
(V-T)\Lambda_c^*(U_c^*P_c)(I_{\ell^2_+}\otimes P_c^\perp U_c+S\otimes P_cU_c)=0.
\label{proof-2-2}
\end{align}
The first equality is true because for every $h,h'\in \mathcal H$,
\begin{eqnarray*}
\langle P_cU_c\Lambda_c (V-T)h,P_cU_c\Lambda_c (V-T)h' \rangle=\langle (V_1-T_1)h\oplus 0,(V_1-T_1)h'\oplus 0\rangle =\langle D_{T_1}^2h,h' \rangle.
\end{eqnarray*}
As for the second equality, note that our conventions are such that the operator  $U_c^*P_c$ appearing
in the second equality \eqref{proof-2-2} is short hand for the operator
$ P_{\{e_0\}} \otimes U_c^* P_c$ acting on the spaced $\ell^2_+(\cF_c) \cong  \ell^2_+ \otimes \cF_c$.
Let us compute
\begin{align*}
(U_c^* P_c) (I_{\ell^2_+} \otimes P_c^\perp U_c + S \otimes P_c U_c) := &
(P_{\{ e_0\}} \otimes U_c^* P_c) (I_{\ell^2_+} \otimes P_c^\perp U_c + S \otimes P_c U_c) \\
=& P_{\{e_0\}} \otimes U_c^* P_c P_c^\perp U_c + ( P_{\{e_0\}} S \otimes U_c^* P_c U_c)  \\
=& P_{\{e_0\}} \otimes 0 + 0 \otimes U_c^* P_c U_c = 0.
\end{align*}
Hence also
$$
(V-T) \Lambda_c^* \cdot (U_c^* P_c) (I_{\ell^2_+} \otimes P_c^\perp U_c + S \otimes P_c U_c) = 0.
$$
Thus \eqref{proof-2-2} follows as well and we conclude that $V_{1c}$ is an isometry.
 Similarly one can show that $V_{2c}$ is an isometry too. This completes the proof.
\end{proof}

\subsection{The Douglas coordinate-free model for an And\^o dilation}
We now work with the first decomposition of the minimal isometric dilation space as in (\ref{decom4}). As before we first must find a non-zero $\cR_0$ such that the And\^o dilation space is given by
\begin{align}\label{CFDspace*}
\cK=M_+(\cL_*)\oplus \cR\oplus \cR_0,
\end{align}where $\cR$ is the minimal component, i.e., by part 6(b) of Proposition \ref{P:dil},
$$
\cR=\overline{\operatorname{span}}\{V^n P_\cR h: n\geq 0,\;h\in\cH\}.
$$
For this construction we make use of a coordinate-free And\^o tuple $(\cF_{c*}, \Lambda_{c*}, P_{c*}, U_{c*})$ for $(T_1^*, T_2^*)$,
 (see Definition \ref{D:Ando-c*}).
We then choose
$$
\cR_0 = \ell^2_+(\cF_{c*}\ominus\Lambda_{c*}\cL_{*})
$$
and consider the following identification of the space given in (\ref{CFDspace*})
$$
\cK = \ell^2_+(\Lambda_{c*} \cL_*) \oplus \cR \oplus \ell^2_+(\cF_{c*}\ominus\Lambda_{c*}\cL_{*})
=\ell^2_+(\cF_{c*}) \oplus \cR.
$$

Let the operators $Q$, $W_D$ be as in the discussion around \eqref{isom-id} --\eqref{IntofQ}.
Let $S_{\cF_{c*}}$ denote the forward shift on $\ell^2_+(\cF_{c*})$; as was the case
above where coefficient spaces
other than $\cF_{c*}$ came up, when we view $\ell^2_+(\cF_{c*})$ as $\ell^2_+ \otimes \cF_{c*}$, then
we may view $S_{\cF_{c*}}$ as equal to $S \otimes I_{\cF_{c*}}$.
Define $V^D_c$ on $\ell^2_+(\cF_{c*})\oplus\cR$ as
\begin{align}\label{VDc}
V^D_c=S_{\cF_{c*}}\oplus W_D
\end{align}and $\Pi_{D}^c:\cH\to \ell^2_+(\cF_{c*})\oplus\cR$ as
\begin{align}\label{PiDc}
  \Pi_D^ch=\sum_{n=0}^{\infty}S_{\cF_{c*}}^n E_0^* \Lambda_{c*}(I-VT^*)T^{* n}h \oplus Qh
\end{align}
where $E_0^* \colon \cF_{c*} \to \ell^2_+(\cF_{c*})$ is the embedding of $\cF_{c*}$ into the $0$-component
$\{f \oplus 0 \oplus 0 \oplus 0 \oplus \cdots \colon f \in \cF_{c*}\}$ of $\ell^2_+(\cF_{c*})$.

The following computation shows that $V^D_c$ is a lift of $T$, where we use the second equation in (\ref{IntofQ}):
\begin{align*}
\Pi_D^cT^*h & =\sum_{n=0}^{\infty} S_{\cF_{c*}}^n  E_0^* \Lambda_{c*} (I-VT^*)T^{* n+1}h \oplus QT^*h\\
& = S^*_{\cF_{c*}} \left( \sum_{n=0}^\infty S_{\cF_{c*}}^{n+1} E_0^* \Lambda_{c*} (I - V T^*) T^{* n+1} h \right)
\oplus Q T^* h \\
&=S_{\cF_{c*}}^* \left( \sum_{n=1}^{\infty} S_{\cF_{c*}}^n  E_0^* \Lambda_{c*} (I-VT^*)T^{* n}h \right)
   \oplus W_D^*Qh \\
  & = S_{\cF_{c*}}^* \left( \sum_{n=0}^\infty
  S_{\cF_c*}^n E_0^* \Lambda_{c*} (I - V T^*) T^{*n} h \right)  \oplus W_D^* Q h   = (V^{D}_c)^*\Pi_D^c h.
\end{align*}

We shall find an And\^o dilation $(V_{c1}^D,V_{c2}^D)$ of $(T_1,T_2)$ on $\ell^2_+(\cF_{c*})
\oplus\cR$ such
that $V_{c1}^DV_{c2}^D=V^D_c$. To that end, we do further analysis of the operator $Q$
defined as the positive square root of $Q^2 : = \operatorname{SOT-}
\lim_{n \to \infty} T^nT^{* n}$, where here $T=T_1T_2=T_2T_1$. We first note that $Q$
has the following additional properties
\begin{eqnarray}\label{prop1}
T_1Q^2T_1^* \preceq Q^2 \text{ and }T_2Q^2T_2^* \preceq Q^2,
\end{eqnarray}the first one of which follows easily from the following inner product computation for every $h\in\mathcal{H}$ (the proof for the second one is similar):
\begin{eqnarray*}
\langle T_1Q^2T_1^*h,h\rangle = \lim_n\langle T^n(T_1T_1^*){T^*}^nh,h \rangle \leq \lim_n\langle T^n{T^*}^nh,h\rangle = \langle Q^2h,h\rangle.
\end{eqnarray*}We now recall the following simple but telling result of Douglas.

\begin{lemma}[Douglas Lemma, \cite{Douglas}]\label{DougLem}
Let $A$ and $B$ be two bounded operators on a Hilbert space $\mathcal{H}$. Then there exists a contraction $C$ such that
$A=BC$ if and only if $$AA^*\preceq BB^*.$$
\end{lemma}

See the paper \cite{Douglas} for a general version of the above lemma. Inequalities in (\ref{prop1}) and Lemma
 \ref{DougLem} imply that there exist two contractions say $X_1^*$ and $X_2^*$ such that
\begin{eqnarray}\label{X1X2}
X_1^*Q=QT_1^* \text{ and } X_2^*Q=QT_2^*.
\end{eqnarray}
From the equalities in (\ref{X1X2}) and (\ref{defX}) it is clear that $X_1$ and $X_2$ commute and that
\begin{eqnarray}
X^*=X_1^*X_2^*.
\end{eqnarray}
Since $X^*$ is an isometry, both $X_1^*$ and $X_2^*$ are isometries, as a consequence of the general fact that,
whenever $T$ is an isometry with factorization $T = T_1 T_2$ for some commuting contractions $T_1$ and $T_2$,
then in fact $T_1$ and $T_2$ are also isometries; indeed
consider the following norm equalities which we have already seen in
the discussion of And\^o tuples (Definition \ref{D:Ando-tuple}):
$$
\|D_{T_1}T_2h\|^2+\|D_{T_2}h\|^2=\|D_Th\|^2=\|D_{T_1}h\|^2+\|D_{T_2}T_1h\|^2 \text{ for all }h \in \mathcal{H}.
$$
Also, note that the same is true if the word `isometry' is replaced by `unitary' because the above equalities hold for every contraction, in particular, for $T_1^*$ and $T_2^*$ also. We now recall the following result from \cite{sauAndo}.

\begin{lemma}\label{special-ext}
Let $\underline{V}=(V_1,V_2,\dots, V_n,V)$ be a commuting tuple of isometries on a Hilbert space $\mathcal{H}$ such that $V=V_1V_2\cdots V_n$, then $\underline{V}$ has a commuting unitary extension $\underline{Y}=(Y_1,Y_2,\dots, Y_n,Y)$ such that $Y=Y_1Y_2\cdots Y_n$ is the minimal unitary extension of $V$.
\end{lemma}
%

By this lemma we get a commuting unitary extension $(W_{\partial1}^*,W_{\partial2}^*)$ on $\mathcal{R}_D$ of $(X_1^*,X_2^*)$ such that $W_D^*=W_{\partial1}^*W_{\partial2}^*$ is the minimal unitary extension of $X^*=X_1^*X_2^*$, equivalently, $W_D$ is the minimal isometric dilation of $X$, hence
\begin{eqnarray}\label{TheSpace-R}
\mathcal{R}_D=\overline{\text{span}}\{W_D^{ n}x: x\in \overline{\operatorname{Ran}}\,Q\text{ and }n\geq0\}.
\end{eqnarray}
The lemma below shows that the pair $(W_1,W_2)$ depends uniquely on $(T_1,T_2)$.

\begin{lemma}
Let $(T_1,T_2)$ on $\mathcal{H}$ and $(T_1',T_2')$ on $\mathcal{H'}$ be two pairs of commuting contractions and $(W_{\partial1},W_{\partial2})$ on $\mathcal{R}_D$ and $(W_{\partial1}',W_{\partial2}')$ on $\mathcal{R}_D'$ be the respective pairs of commuting unitaries obtained from them as above. If $(T_1,T_2)$ is unitarily equivalent to $(T_1',T_2')$ via the unitary similarity $\phi \colon \mathcal{H} \to \mathcal{H'}$, then $(W_{\partial1},W_{\partial2})$ and $(W_{\partial1}',W_{\partial2}')$ are unitarily equivalent via the induced unitary transformation $\tau_\phi \colon \mathcal{R}_D\to \mathcal{R}_D'$ determined by $\tau_\phi \colon W_D^nQh \to W_D'^n Q'\phi h$.  In particular, if $(T_1, T_2) = (T_1', T_2')$, then
$(W_{\partial1},W_{\partial2}) = (W_{\partial1}',W_{\partial2}')$.
\end{lemma}
\begin{proof}
Let $\phi:\mathcal{H}\to\mathcal{H'}$ be a unitary that intertwines $(T_1,T_2)$ and $(T_1',T_2')$. Let us denote $T=T_1T_2$ and $T'=T_1'T_2'$. Let $Q$ and $Q'$ be the limits of $T^nT^{* n}$ and $T'^nT'^{* n}$, respectively, in the strong operator topology. Clearly, $\phi$ intertwines $Q$ and $Q'$. Therefore $\phi$ takes $\mathcal{Q}\equiv\overline{\operatorname{Ran}}\,Q$ onto $\mathcal{Q'}\equiv\overline{\operatorname{Ran}}\,Q'$. We denote the restriction of $\phi$ to $\overline{\operatorname{Ran}}\,Q$ by $\phi$ itself. Let $(X_1,X_2)$ on $\overline{\operatorname{Ran}}\,Q$ and $(X_1',X_2')$ on $\overline{\operatorname{Ran}}\,Q'$ be the pairs of commuting co-isometries corresponding to the pairs $(T_1,T_2)$ and $(T_1',T_2')$ as in (\ref{X1X2}), respectively. It is easy to see from the definition that
$$
\phi(X_1,X_2)=(X_1',X_2')\phi.
$$Let $(W_{\partial1},W_{\partial2})$ on $\mathcal{R}_D$ and $(W_{\partial1}',W_{\partial2}')$ on $\mathcal{R}_D'$ be the pairs of commuting unitaries corresponding to $(T_1,T_2)$ and $(T_1',T_2')$, respectively. Remembering the formula (\ref{TheSpace-R}) for the spaces $\mathcal{R}_D$ and $\mathcal{R}_D'$, we define $\tau_\phi:\mathcal{R}_D\to\mathcal{R}_D'$ by
$$
\tau_\phi: W_D^{ n}x\mapsto W_D'^{ n}\phi x, \text{ for every } x\in \mathcal{Q} \text{ and } n\geq 0
$$and extend linearly and continuously. Trivially, $\tau_\phi$ is unitary and intertwines $W_D$ and $W_D'$. For a non-negative integer $n$ and $x$ in $\mathcal{Q}$, we have
\begin{eqnarray*}
\tau_\phi W_{\partial1}(W_D^{n}x)=\tau_\phi W_D^{n+1}(W_{\partial2}^*x)&=&W_D'^{ n+1}\phi(X_2^*x)\\
&=&W_D'^{n+1}W_{\partial2}'^*\phi x=W_{\partial1}'W_D'^{n}\phi x=W_{\partial1}'\tau_\phi(W_D^{ n}x).
\end{eqnarray*}A similar computation shows that $\tau_\phi$ intertwines $W_{\partial2}$ and $W_{\partial2}'$ too.
\end{proof}
Define the following two operators on $\ell^2_+(\cF_{c*})\oplus\cR$:
\begin{align}\label{CFDAndo}
\begin{cases}
V^D_{c1}=(I_{\ell^2_+} \otimes U_{c*}^* P_{c*}^\perp + S \otimes U_{c*}^* P_{c*}) \oplus W_{\partial1},\\
V^D_{c2} = (I_{\ell^2_+} \otimes P_{c*} U_{c*} + S \otimes P^\perp_{c*} U_{c*} ) \oplus W_{\partial2}.
\end{cases}
\end{align}

\begin{thm}\label{Thm:CFDAndo}
With the operators $V_{c1}^D$ and $V_{c2}^D$ as defined in (\ref{CFDAndo}), $(\Pi_D^c,V_{c1}^D,V_{c2}^D)$ is a commuting isometric lift of $(T_1,T_2)$.
\end{thm}

\begin{proof}
Note that $(V_{c1}^D, V_{c2}^D)$ is a commutative pair of isometries since it has the form of  a
Berger-Coburn-Lebow model for a commuting pair of isometries.
 It remains to show that, for $j=1,2,$
$V_{cj}^D$ is a lift of $T_j$.  We  do in detail only the case $j=1$ as the other case is similar.
All that remains to show is that
\begin{equation}   \label{toshow}
V_{c1}^{D*} \Pi_D^c h = \Pi_D^c T_1^* h \text{ for all } h \in \cH.
\end{equation}
We use the definitions to compute
\begin{align}\label{step1}
& V_{c1}^{D*} \Pi_D^c h  \\
=& \left( (I_{\ell^2_+} \oplus P^\perp_{c*} U_{c*}) + (S^* \otimes P_{c*} U_{c*}) \right)
\sum_{n=0}^\infty S_{\cF_{c*}}^n E_0^* \Lambda_{c*} (I - V T^*)T^{*n} h \oplus W_{\partial1}^* Q h   \notag \\
=&  \left( \sum_{n=0}^\infty S_{\cF_{c*}}^n E_0^* P^\perp_{c*} U_{c*} E_0^* \Lambda_{c*}
(I - VT^*) T^{*n} h + \sum_{n=0}^\infty S_{\cF_c}^n E_0^* U_{c*} \Lambda_{c*} (I - VT^*) T^{*n+1} h \right) \oplus
QT_1^* h.
\notag
\end{align}
As a side computation let us note that
\begin{align} \label{side1}
& P_{c*}^\perp U_{c*} \Lambda_{c*} (I - VT^*) T^{*n} h
= P_{c*}^\perp U_{c*} \left( (I - V_1 T_1^*) T_2^* T^{*n} h \oplus ( I - V_2 T_2^*) T^{*n} h \right)  \\
& \quad = P_{c*}^\perp \left( (I - V_1 T_1^*) T^{*n} h \oplus (I - V_2 T_2^*) T_1^* T^{*n} h \right) \notag \\
& \quad  = 0 \oplus (I - V_2 T_2^*) T_1^* T^{*n} h  \notag
\end{align}
while
\begin{align}
\nonumber& P_{c*} U_{c*} \Lambda_{c*} (I - V T^*) T^{*n+1} h\\
&=P_{c*} U_{c*} \left( (I - V_1 T_1^*) T_2^* T^{* n+1} h \oplus (I - V_2 T_2^*) T^{*n+1} h \right)  \label{side2} \\
&=  P_{c*} \left( (I - V_1 T_1^*) T^{*n+1} h \oplus (I - V_2 T_2^*) T_1^* T^{*n+1} h \right) \notag  \\
&= ( I - V_1 T_1^*) T^{*n+1} h \oplus 0.  \notag
\end{align}
Plugging \eqref{side1} and \eqref{side2} back into \eqref{step1} then leads us to
\begin{align*}
V_{c1}^{D*} \Pi_D^c h & = \sum_{n=0}^\infty S_{\cF_{c*}}^n E_0^* \left( (I - V_1 T_1^*) T^{*n+1} h
\oplus (I - V_2 T_2^*) T_1^* T^{*n} h \right)  \\
& = \sum_{n=0}^\infty S_{\cF_{c*}}^n E_0^* \left( (I - V_1 T_1^*) T_2^* T^{*n} T_1^* h
\oplus (I - V_2 T_2^*) T^{*n} T_1^* h \right) \\
& = \sum_{n=0}^\infty S_{\cF_{c*}}^n E_0^* \Lambda_{c*} T_1^* h = \Pi_D^c T_1^* h
\end{align*}
and \eqref{toshow} now follows.
\end{proof}

We can convert the preceding analysis to a functional-model form as follows.  Define the $Z$-transform
$Z_+ \colon \ell^2_+(\cF_{c*}) \to H^2(\cF_{c*})$ by
$$
  Z_+ \colon \{ f_n \}_{n \ge 0} \mapsto \sum_{n=0}^\infty z^n f_n .
$$
Then $Z_+$ is unitary from $\ell^2_+(\cF_{c*})$ onto $H^2(\cF_{c*})$ and intertwines the respective shift operators:
$$
   M_z Z_+ = Z_+ S_{\cF_{c*}}.
$$
Let $\omega_D$ and $\omega_{NF}$ be the unitaries as defined in (\ref{omega-ext}) and (\ref{omegaNF}), respectively. We observed in \S \ref{S:SNFmodel} that the unitary
$$
u:=\omega_{NF}\omega_D^*:\cR_D\to\overline{\Delta_TL^2(\cD_T)}
$$intertwines $W_D$ and $M_{\zeta}|_{\overline{\Delta_TL^2(\cD_T)}}$. Consequently, the operator
\begin{align*}
\cZ:=Z_+\oplus u:\ell^2_+(\cF_{c*})\oplus\cR_D\to H^2(\cF_{c*})\oplus\overline{\Delta_TL^2(\cD_T)}
\end{align*}has the following intertwining property
\begin{align*}
  \cZ V_c^D=(M_z\oplus M_{\zeta}|_{\overline{\Delta_TL^2(\cD_T)}})\cZ.
\end{align*}
We adopt the following notations
\begin{align}\label{WandVNFs}
\begin{cases}
(W_{\sharp1},W_{\sharp 2},M_{\zeta}|_{\overline{\Delta_TL^2(\cD_T)}}):=(uW_{\partial1}u^*,uW_{\partial2}u^*,uW_Du^*),\\
(V_{\flat1},V_{\flat2},V_{\flat}):=(\cZ V_{c1}^D\cZ^*,\cZ V_{c2}^D\cZ^*,\cZ V_c^D\cZ^*)
\end{cases}
\end{align}
and see that for every $f=\sum_{n=0}^{\infty}z^na_n$ in $H^2(\cF_{c*})$ and $g$ in $\overline{\Delta_TL^2(\cD_T)}$,
\begin{align*}
\cZ V_{c1}^D\cZ^*(f\oplus g)&=\cZ V_{c1}^D\cZ^*\left(\sum_{n=0}^{\infty}z^na_n\oplus g\right)\\
&=\cZ (I_{\ell^2_+} \otimes U_{c*}^* P_{c*}^\perp + S \otimes U_{c*}^* P_{c*})\{a_n\}_{n\geq 0}\oplus W_{\sharp1}g\\
&=\cZ\{U_{c*}^* P_{c*}^\perp a_n+U_{c*}^* P_{c*}a_{n-1}\}_{n\geq0}\oplus W_{\sharp1}g\quad[a_{-1}:=0]\\
&=M_{U_{c*}^* P_{c*}^\perp + z U_{c*}^* P_{c*}} \oplus W_{\sharp1}(f\oplus g).
\end{align*}
Therefore (after a similar computation with $\cZ V_{c2}^D\cZ^*$) we obtain
\begin{align}\label{NFAndo}
\nonumber&(V_{\flat1},V_{\flat2},V_{\flat})=(V_{\flat1},V_{\flat2},V_{\flat1}V_{\flat2})\\
=&(M_{U_{c*}^* P_{c*}^\perp + z U_{c*}^* P_{c*}} \oplus W_{\sharp1},M_{P_{c*} U_{c*} + z P^\perp_{c*} U_{c*}} \oplus W_{\sharp2},M_z\oplus M_{\zeta}|_{\overline{\Delta_TL^2(\cD_T)}}).
\end{align} Finally define an isometry $\Pi_{\flat}:\cH\to H^2(\cF_{c*})\oplus\overline{\Delta_TL^2(\cD_T)}$ as
\begin{align}\label{NF-iso}
\Pi_{\flat} h=\cZ\Pi^c_Dh=\sum_{n=0}^{\infty}z^n\Lambda_{c*} (I-VT^*)T^{* n}h\oplus u(Qh).
\end{align}Therefore by Theorem \ref{Thm:CFDAndo} we have proved the following theorem, which gives a Sz.-Nagy--Foias model for an And\^o dilation for a pair of commuting contractions.

\begin{thm}\label{Thm:NFAndo}
Let $(T_1,T_2)$ be a pair of commuting contractions and  let $(\cF_{c*},\Lambda_{c*},P_{c*},U_{c*})$ be a coordinate-free And\^o tuple for $(T_1^*, T_2^*)$. Then with the notations as in (\ref{NFAndo}) and (\ref{NF-iso}) the triple $(\Pi_{\flat},V_{\flat1},V_{\flat2})$ is an isometric lift of $(T_1,T_2)$.
\end{thm}

\section{Characteristic/Admissible triples and functional model for a commuting pair of contractions} \label{S:NFmodel}

\subsection{Characteristic triples and functional models}

Let $(T_1,T_2)$ be a pair of commuting contractions and $(\cF_{c*},\Lambda_{c*},P_{c*},U_{c*})$ be a
coordinate-free And\^o tuple for $(T_1^*,T_2^*)$. Define a pair $(G_{\sharp1},G_{\sharp2})$ of contractions
on $\cD_{T^*}$ as
\begin{align}\label{Gs}
  (G_{\sharp1},G_{\sharp2}):=(\iota_*\Lambda_{c*}^*P_{c*}^\perp U_{c*}\Lambda_{c*}\iota_*^*,\iota_*\Lambda_{c*}^*U_{c*}^*P_{c*}\Lambda_{c*}\iota_*^*).
\end{align}Let the pair $(W_{\sharp1},W_{\sharp2})$ of commuting unitaries be as defined in (\ref{WandVNFs}). The objective of this section is to study the triple $((G_{\sharp1},G_{\sharp2}),(W_{\sharp1},W_{\sharp2}),\Theta_T)$, where $\Theta_T$ is the characteristic triple for the contraction $T=T_1T_2$. We start by giving the triple a name.

\begin{definition}
The triple $((G_{\sharp1},G_{\sharp2}),(W_{\sharp1},W_{\sharp2}),\Theta_T)$ is called the characteristic triple for $(T_1,T_2)$.
\end{definition}

Recall that a And\^o tuple for $(T_1^*, T_2^*)$ depends on an arbitrary choice of unitary extension of a given
partially defined isometric map, and hence is highly nonunique except in the case where $T = T_1 \cdot T_2$
is a regular factorization.  Thus, as the definition of characteristic triple for a commutative contractive pair
$(T_1, T_2)$ depends on a choice of And\^o tuple for $(T_1^*, T_2^*)$, it would appear that $(T_1, T_2)$
does not uniquely determine its characteristic triple, specifically the component ${\mathbb G} =
(G_{\sharp 1}, G_{\sharp 2})$.  This presumption however remarkably turns out to be
wrong, as noted in the following result.

\begin{proposition}  \label{P:fundamental}  Let $(T_1, T_2)$ be a commutative contractive operator pair.
Then the operators $G_{\sharp 1}, G_{\sharp 2}$ in a characteristic triple
$((G_{\sharp 1}, G_{\sharp 2}), (W_{\sharp 1}, W_{\sharp 2}, \Theta_T)$ are uniquely determined from
$(T_1, T_2)$ as solutions of the equations
\begin{equation}   \label{fundamental}
T_i^* - T_j T^* = D_{T^*} G_{\sharp i } D_{T^*} \text{ where } (i,j) = (1,2) \text{ or } (2,1).
\end{equation}
\end{proposition}

\begin{remark} \label{R:uniqueness}
{\rm As the operators $G_{\sharp 1}$ and $G_{\sharp 2}$ are by definition operators on the
space $\cD_{T^*} = \overline{\operatorname{Ran}} \, D_{T^*}$, any solutions of the equations \eqref{fundamental},
assuming that such exist, must be unique.  Operators $G_{\sharp 1}$ and $G_{\sharp 2}$ satisfying equations of this
type appear in the theory of $\Gamma$-contractions (commutative operator pairs having the symmetrized
bidisk as a complete spectral set) and of tetrablock contractions (commutative triples of operators having
the tetrablock domain as a complete spectral set)---we refer to \cite{B-P-SR},  \cite{sir's tetrablock paper}
for further details.
}\end{remark}

\begin{proof}[Proof of Proposition \ref{P:fundamental}.]
We sketch the proof of \eqref{fundamental} only for the case $(i,j) = (1,2)$ as the case $(i,j) = (2,1)$
is similar.

Let $(V_1,V_2)$ be an arbitrary And\^o dilation for $(T_1,T_2)$ and set $V=V_1V_2$. Then for
$h, h' \in \cH$ we compute
\begin{align*}
\langle D_{T^*}G_{\sharp1}D_{T^*}h,h'\rangle_\cH &=
\langle P_{c*}^\perp U_{c*}\Lambda_{c*}\iota_*^*D_{T^*}h, \Lambda_{c*}\iota_*^*D_{T^*}h'\rangle_{\cF_{c*}}
\text{ (by \eqref{Gs})}\\
&=\langle P_{c*}^\perp U_{c*}\Lambda_{c*}(I_{\cH}-VT^*)h,\Lambda_{c*}(I_{\cH}-VT^*)h'\rangle_{\cF_{c*}}
 \text{ (by Proposition \ref{P:dil})}\\
&=\langle (I_{\cH}-V_2T_2^*)T_1^*h,(I_{\cH}-V_2T_2^*)h' \rangle_{\cL_{2*}} \\
& = \langle D_{T_2} T_1^* h, D_{T_2} h' \rangle_{\cF_{c*}}
                           \text{ (by Proposition \ref{P:dil} with $T_2$ in place of $T$)} \\
& = \langle (I - T_2^* T_2) T_1^* h, h' \rangle_\cH = \langle (T_1^* - T_2 T^*) h, h' \rangle_\cH
\end{align*}
and the result follows.
\end{proof}

The characteristic triple $((G_{\sharp1},G_{\sharp2}),(W_{\sharp1},W_{\sharp2}),\Theta_T)$ for a commuting
pair of contractions $(T_1, T_2)$ is the invariant leading to the construction of a Sz.-Nagy--Foias-type
functional model for $(T_1, T_2)$ as follows.

\begin{thm}\label{Thm:SNFmodelPair}
Let $(T_1,T_2)$ be a pair of commuting contractions and let its characteristic triple be $((G_{\sharp1},G_{\sharp2}),(W_{\sharp1},W_{\sharp2}),\Theta_T)$. Then the Sz.-Nagy--Foias model space
\begin{equation}  \label{HNF'}
\cH_{NF}=\left(H^2(\cD_{T^*})\oplus\overline{\Delta_T(L^2(\cD_T))})\right)\ominus\{\Theta_Tf\oplus\Delta_Tf:f\in H^2(\cD_T)\}
\end{equation}
is coinvariant under
$$
(M_{G_{\sharp1}^*+zG_{\sharp2}}\oplus W_{\sharp1},M_{G_{\sharp2}^*+zG_{\sharp1}}\oplus W_{\sharp2},M_z\oplus M_{\zeta}|_{\overline{\Delta_T(L^2(\cD_T))}})
$$and $(T_1,T_2,T_1T_2)$ is unitarily equivalent to
\begin{align}\label{NFmodelPair}
P_{\cH_{NF}}(M_{G_{\sharp1}^*+zG_{\sharp2}}\oplus W_{\sharp1},M_{G_{\sharp2}^*+zG_{\sharp1}}\oplus W_{\sharp2},M_z\oplus M_{\zeta}|_{\overline{\Delta_T(L^2(\cD_T))}})|_{\cH_{NF}}.
\end{align}
\end{thm}

\begin{proof}
Let us denote
\begin{align}\label{Vs}
\nonumber\underline{V}_{NF}:=&(S^{NF}_1,S^{NF}_2,V_{NF})\\=&(M_{G_{\sharp1}^*+zG_{\sharp2}}\oplus W_{\sharp1},M_{G_{\sharp2}^*+zG_{\sharp1}}\oplus W_{\sharp2},M_z\oplus M_{\zeta}|_{\overline{\Delta_T(L^2(\cD_T))}}).
\end{align}
Note that with the triple $(\Pi_{\flat},V_{\flat1},V_{\flat2})$ as in Theorem \ref{Thm:NFAndo} and
with the isometry $L$ given by
$$
L:=(I_{H^2}\otimes\Lambda_{c*}\iota_*^*\oplus I_{\overline{\Delta_T(L^2(\cD_T))}}):H^2(\cD_{T^*})\oplus\overline{\Delta_T(L^2(\cD_T))}\to H^2(\cF_{c*})\oplus\overline{\Delta_T(L^2(\cD_T))},
$$
it is a direct check using the definitions that
\begin{equation}  \label{check}
L^*(V_{\flat1},V_{\flat2},V_{\flat})L=(S^{NF}_1,S^{NF}_2,V_{NF})=\underline{V}_{NF}
\text{ and }   L\Pi_{NF}=\Pi_{\flat}.
\end{equation}
By Theorem \ref{Thm:NFAndo} we also have
\begin{equation}   \label{intertwine1}
(V_{\flat 1}, V_{\flat 2}, V_\flat)^* \Pi_\flat = \Pi_\flat (T_1^*, T_2^*, T^*).
\end{equation}
Using the second equation in \eqref{check} then leads to
$$
  (V_{\flat 1}, V_{\flat 2}, V_\flat)^* L \Pi_{NF} = L \Pi_{NF} (T_1^*, T_2^*, T^*).
$$
Multiplying by $L^*$ on the left and using that $L$ is an isometry then gives us
$$
 L^* (V_{\flat 1}, V_{\flat 2}, V_\flat )^* L \Pi_{NF} = \Pi_{NF} (T_1^*, T_2^*, T^*).
$$
Using the first equation in \eqref{check} then gives us
\begin{equation}\label{Eg1}
 (S^{NF}_1,S^{NF}_2,V_{NF})^*\Pi_{NF}=\Pi_{NF}(T_1,T_2,T_1T_2)^*.
\end{equation}
We have seen in \S \ref{S:SNFmodel} that $\operatorname{Ran}\Pi_{NF}$ is $\cH_{NF}$ as  in
\eqref{HNF'}.   Hence  from the intertwining (\ref{Eg1}) we see that the unitary transformation
$\Pi_{NF} \colon \cH \to \cH_{NF}$ establishes the unitary equivalence between the operator tuple
\eqref{NFmodelPair} on $\cH_{NF}$ and $(T_1, T_2, T)$ on $\cH$, and the theorem follows.
\end{proof}

\subsection{Characteristic triple as a complete unitary invariant}
It was proved by Sz.-Nagy--Foias  (see \cite[Chapter VI]{Nagy-Foias} that the characteristic function $\Theta_T$
for a c.n.u.\ contraction $T$ is a complete unitary invariant. This means that two c.n.u.\ contractions $T$ and $T'$
are unitarily equivalent if and only if their characteristic functions {\em{coincide}} in the sense that there exist
unitary operators $u: \mathcal{D}_T \to \mathcal{D}_{T'}$ and $u_{*}: \mathcal{D}_{T^*} \to \mathcal{D}_{{T'}^*}$ such that the following diagram commutes for every $z\in\mathbb D$:
\begin{align}\label{coindiagram}
\begin{CD}
\mathcal{D}_T @>\Theta_T(z)>> \mathcal{D}_{T^*}\\
@Vu VV @VVu_{*} V\\
\mathcal{D}_{T'} @>>\Theta_{T'}(z)> \mathcal{D}_{{T'}^*}
\end{CD}.
\end{align}
Theorem \ref{UnitaryInv} below shows that such a result holds for characteristic triples of pairs of
commuting contractions also. First we define a notion of coincidence for such a triple.

A contractive analytic function $(\cD,\cD_*,\Theta)$ is a $\cB(\cD,\cD_*)$-valued analytic function on $\mathbb{D}$ such that
$$
\|\Theta(z)\|\leq 1 \text{ for all } z\in\mathbb{D}.
$$
Such a function is called {\it purely contractive} if $\Theta(0)$ does not preserve the norm of any nonzero
vector, i.e.,
\begin{equation}   \label{pureCAF}
\|\Theta(0)\xi\|_{\cD_*}<\|\xi\|_{\cD} \text{ for all nonzero }\xi\in\cD.
\end{equation}
We note that a Sz.-Nagy--Foias characteristic function $\Theta_T$ is always purely contractive
(see \cite[Section VI.1]{Nagy-Foias}),  and that it is
always the case that a general contractive analytic function $(\cD, \cD_*, \Theta)$ has a block diagonal
decompostiton  $\Theta = \Theta' \oplus \Theta^0$ where $(\cD', \cD_*', \Theta')$ is a unitary constant function
and $(\cD^0, \cD_*^0, \Theta^0)$ is purely contractive.  A key easily checked property of this decomposition
is the following:

\begin{obs}   \label{O:reduction}
 The model space
$$
\cH_{NF} = \begin{bmatrix} H^2(\cD_*) \\ \overline{ \Delta_\Theta L^2(\cD)} \end{bmatrix}
 \ominus \begin{bmatrix} \Theta \\ \Delta_\Theta \end{bmatrix} H^2(\cD)
$$
and the associated model operator
$$
T_{NF} = P_{\cH_{NF}} \begin{bmatrix}  M_z & 0 \\ 0 & M_\zeta \end{bmatrix} \big|_{\cH_{NF}}
$$
remain exactly the same (after some natural identification of respective coefficient spaces) when $\Theta$ is replaced by $\Theta^0$.
\end{obs}
  Thus only completely contractive
analytic functions are relevant when discussing Sz.-Nagy--Foias functional models.
For the moment we consider only purely contractive
analytic functions.

\begin{definition}\label{coincidence}
Let $(\mathcal{D},\mathcal{D}_*,\Theta)$, $(\mathcal{D'},\mathcal{D'_*},\Theta')$ be two purely contractive
analytic functions, $\mathbb{G}=(G_1,G_2)$ on $\mathcal{D}_*$, $\mathbb{G}'=(G_1',G_2')$ on $\mathcal{D'_*}$
be two pairs of contractions and $\mathbb{W}=(W_1,W_2)$ on $\overline{\Delta_\Theta L^2(\mathcal{D})}$,
$\mathbb{W}'=(W_1',W_2')$ on $\overline{\Delta_{\Theta'} L^2(\mathcal{D'})}$ be two pairs of commuting
unitaries such that their product is $M_{\zeta}$ on the respective spaces. We say that the two triples $(\mathbb{G},\mathbb{W},\Theta)$ and $(\mathbb{G}',\mathbb{W}',\Theta')$ coincide, if \
\begin{itemize}
  \item[(i)] $(\mathcal{D},\mathcal{D_*},\Theta)$ and $(\mathcal{D'},\mathcal{D'_*},\Theta')$ coincide, i.e.,
there exist unitary operators $u: \mathcal{D} \to \mathcal{D'}$ and $u_{*}: \mathcal{D}_{*} \to \mathcal{D'}_{*}$
such that the diagram (\ref{coindiagram}) commutes with $\Theta_T$ and $\Theta_{T^*}$ in place of $\Theta$
and $\Theta'$, respectively;
\item[(ii)] the unitary operators $u$, $u_*$ have the following intertwining properties:
\begin{eqnarray}
\begin{cases}
\mathbb{G}'=(G_1',G_2')=u_*\mathbb{G}u_*^*=(u_*G_1u_*^*,u_*G_2u_*^*) \text{ and }\\
\mathbb{W}'=(W_1',W_2')=\omega_u\mathbb{W}\omega_u^*=(\omega_uW_1\omega_u^*,\omega_uW_2\omega_u^*),
\end{cases}
\end{eqnarray}
where $\omega_u:\overline{\Delta_{\Theta} L^2(\mathcal{D})}\to\overline{\Delta_{\Theta'} L^2(\mathcal{D'})}$
is the following unitary map induced by $u$:
\begin{eqnarray}\label{omega-u}
\omega_u:=(I_{L^2}\otimes u)|_{\overline{\Delta_{\Theta} L^2(\mathcal{D})}}.
\end{eqnarray}
\end{itemize}
\end{definition}

We shall use the following uniqueness result from \cite{sauAndo} later in this section. For a pair $(T_1,T_2)$ of commuting contractions on $\mathcal{H}$ and $\underline{T}=(T_1,T_2,T_1T_2)$, let
\begin{eqnarray*}
\mathcal{U}_{\underline{T}}&:=&
\{(\Pi,\mathcal{K}, \underline{V}): \underline{V}=(S_1,S_2,V),\; \Pi:\mathcal{H}\to\mathcal{K} \text{ is an isometry such that }\\&& \underline{V}^*\Pi=\Pi\underline{T}^*, \;(\Pi, V)\text{ is the minimal isometric dilation of } T=T_1T_2,\\&&
(S_1,V),\; (S_2,V) \text{ are commuting and } S_1=S_2^*V.\}
\end{eqnarray*}
We next exhibit a concrete example of a member of $\mathcal{U}_{\underline{T}}$ for a given pair $(T_1,T_2)$ of commuting contractions.

\begin{example}\label{TheEg}
{\rm Let $((G_{\sharp1},G_{\sharp2}),(W_{\sharp1},W_{\sharp2}),\Theta_T)$ be the characteristic triple for $(T_1,T_2)$. Let $\Pi_{NF}:\cH\to H^2(\cD_{T^*})\oplus\overline{\Delta_T(L^2(\cD_T))}$ be the isometry as in (\ref{Unf&Pinf}), i.e.,
\begin{align*}
\Pi_{NF}h=\sum_{n=0}^{\infty}z^nD_{T^*}T^{*n}h\oplus u(Qh)=U_{\text{min}}\Pi_Dh.
\end{align*}
Consider the triple $(\Pi_{NF},\cK_{NF},\underline{V}_{NF})$, where
$\underline{V}_{NF}=(S^{NF}_1,S^{NF}_2,V_{NF})$ is as in (\ref{Vs}). It can be checked easily from the definition
that $(S_1^{NF},V_{NF})$ and $(S_2^{NF},V_{NF})$ are commuting and that $S_1^{NF}={S_2^{NF}}^*V_{NF}$.
We have also seen in
\S \ref{S:SNFmodel} that $(\Pi_{NF}, V_{NF})$ is the Sz.-Nagy--Foias minimal isometric dilation of
$T=T_1T_2$.  Therefore with $\cK_{NF}=H^2(\cD_{T^*})\oplus\overline{\Delta_T(L^2(\cD_T))}$, we conclude
that the triple
$(\mathcal{K}_{NF}, \Pi_{NF}, \underline{V}_{NF})$ is in $\mathcal{U}_{\underline{T}}$ for $(T_1,T_2)$.
}\end{example}

The theorem below, proved in \cite{sauAndo}, asserts that any triple $(\cK, \Pi, V)$ in $\cU_T$ is unitarily
equivalent as an element of $\cK_T$ to the model triple  $(\mathcal{K}_{NF},  \Pi_{NF}, \underline{V}_{NF})$.

\begin{thm}\label{uniqueness}
  For a pair $\underline{T}:=(T_1,T_2)$ of commuting contractions, the family $\mathcal{U}_{\underline{T}}$ is
   a singleton set under unitary equivalence, i.e., if $(\mathcal{K}, \Pi,  \underline{V})$ and
   $(\mathcal{K}' ,  \Pi',   \underline{V'})$ are in $\mathcal{U}_T$, then there exists a unitary
   $U:\mathcal{K}\to\mathcal{K'}$ such that
  \begin{eqnarray}\label{IntertwiningU}
  U\underline{V}=\underline{V'}U \text{ and }U(\Pi h)=\Pi'h, \text{ for all }h\in \mathcal{H}.
  \end{eqnarray}
\end{thm}

As we have seen in Theorem \ref{T:min-iso-lift}, if $(\Pi,V)$ and $(\Pi',V')$ are any two minimal isometric
dilations of a contraction $T$, then
there is a {\em{unique}} unitary $U$ that intertwines $V$ and $V'$ and $U\Pi=\Pi'$. This means that if $V=V'$, then
$U$ has to be the identity operator.  Let now $(\Pi,\mathcal{K}, \underline{V})$ and
$(\Pi',\mathcal{K'}, \underline{V'})$ be two members in $\mathcal{U}_{\underline{T}}$ for a pair $(T_1,T_2)$
of commuting contractions such that the last entries of $\underline{V}$ and $\underline{V'}$ are the same.
Then by (\ref{IntertwiningU}) we see that  this forces $\underline{V}=\underline{V'}$,
and hence the following corollary is easily obtained.

\begin{corollary}\label{Uniqueness-Cor}
For a pair of commuting contractions $(T_1,T_2)$, let  the triples $(\mathcal{K}, \Pi, \underline{V})$ and
$(\mathcal{K'}, \Pi', \underline{V'})$ be in $\mathcal{U}_{\underline{T}}$ such that the last entries of $\underline{V}$ and $\underline{V'}$ are the same. Then $\underline{V}=\underline{V'}.$
\end{corollary}

\begin{thm}\label{UnitaryInv}
Let $(T_1,T_2)$ and $(T_1',T_2')$ be two pairs of commuting contractions such that their products $T=T_1T_2$ and $T'=T_1'T_2'$ are c.n.u.\ contractions. Then $(T_1,T_2)$ and $(T_1',T_2')$ are unitarily equivalent if and only if their characteristic triples coincide.
\end{thm}
\begin{proof}
Let $(T_1,T_2)$ on $\mathcal{H}$ and $(T_1',T_2')$ on $\mathcal{H}'$ be unitarily equivalent via the unitary operator $U:\mathcal{H}\to\mathcal{H}'$ and $((G_{\sharp1},G_{\sharp2}),(W_{\sharp1},W_{\sharp2}),\Theta_T)$ and $((G'_{\sharp1},G'_{\sharp2}),(W_{\sharp1}',W_{\sharp2}'),\Theta_{T'})$ be their characteristic triples, respectively. It is easy to see that $UD_T=D_{T'}U$ and $UD_{T^*}=D_{T'^*}U$ and that the unitaries
\begin{eqnarray}
u:=U|_{\mathcal{D}_T}: \mathcal{D}_{T} \to \mathcal{D}_{T'}\text{ and }u_{*}:=U|_{\mathcal{D}_{T^*}}: \mathcal{D}_{T^*} \to \mathcal{D}_{T'^{*}}
\end{eqnarray} have the following property:
\begin{eqnarray}\label{coincd}
u_*\Theta_T=\Theta_{T'}u.
\end{eqnarray} Hence $\Theta_T$ and $\Theta_{T'}$ coincide. We now show that the unitary $u_*$ above plays the role in unitary equivalence of $(G_{\sharp1},G_{\sharp2})$ and $(G_{\sharp1}',G_{\sharp2}')$.

To this end we  note that by Proposition \ref{P:fundamental} the following set of equations is satisfied:
$$
 T_i^* - T_j T^* = D_{T^*} G_{\sharp i} D_{T^*},   \quad
 T_i^{\prime*} - T'_j T^{\prime *} = D_{T^{\prime *}} G'_{\sharp i} D_{T^{\prime *}} \text{ for }
 (i,j) = (1,2) \text{ or } (2,1).
 $$
It then follows that
\begin{eqnarray}\label{fundequiv}
u_*(G_{\sharp1},G_{\sharp2})=(G_{\sharp1}',G_{\sharp2}')u_*.
\end{eqnarray} We have noticed in Example \ref{TheEg} that for a pair $(T_1,T_2)$ of commuting contractions, the triple $(\Pi_{NF},\mathcal{K}_{NF}, \underline{V}_{NF})$ is always in $\mathcal{U}_{\underline{T}}$. Let $(\Pi'_{NF},\mathcal{K}'_{NF}, \underline{V'}_{NF})$ be the corresponding triple for $(T_1',T_2')$. Let us denote by $\Pi''$ the following isometry
\begin{align}\label{Pi''}
\Pi'':=((I_{H^2}\otimes u_*^*)\oplus \omega_u^*)\Pi_{NF}'U:\mathcal{H}\to H^2(\mathcal{D}_{T^*})\oplus\overline{\Delta_{T}L^2(\mathcal{D}_{T})}
\end{align}where $\omega_u:\overline{\Delta_{T}L^2(\mathcal{D}_T)}\to\overline{\Delta_{T'}L^2(\mathcal{D}_{T'})}$ is the unitary $\omega_u=(I_{L^2}\otimes u)|_{\overline{\Delta_{T}L^2(\mathcal{D}_T)}}$. We observe that the triple $(\Pi'',\mathcal{K}_{NF},\underline{V''})$ is also in $\mathcal{U}_{\underline{T}}$, where with  $(W_1'',W_2'')=\omega_u^*(W_{\sharp1}',W_{\sharp2}')\omega_u$,
$$
\underline{V''}:=(M_{G_{\sharp1}^*+zG_{\sharp2}}\oplus W_1'',M_{G_{\sharp2}^*+zG_{\sharp1}}\oplus W_2'',M_z\oplus M_{\zeta}|_{\overline{\Delta_T(L^2(\cD_T))}}),
$$
because using (\ref{fundequiv}) we have for $(i,j)=(1,2)$ or $(2,1)$
\begin{eqnarray*}
\Pi''T_i^*&=&((I_{H^2}\otimes u_*^*)\oplus \omega_u^*)\Pi_{NF}'{T'_i}^*U \text{ (by \eqref{Pi''})} \\
&=&((I_{H^2}\otimes u_*^*)\oplus \omega_u^*)(M_{{G'_{\sharp i}}^*+z{G'_{\sharp j}}}^*\oplus {W'_{\sharp i}}^*)\Pi_{NF}'U\\
&=&(M_{G_{\sharp i}^*+zG_{\sharp j}}\oplus {W''_i}^*)((I_{H^2}\otimes u_*^*)\oplus \omega_u^*)\Pi_{NF}'U
\text{ (by \eqref{fundequiv}).}
\end{eqnarray*}
Now since the last entry of $\underline{V''}$ is the same as that of $\underline{V}_{NF}$, applying Corollary \ref{Uniqueness-Cor}, we get
$$
(W_1'',W_2'')=\omega_u^*(W_{\sharp1}',W_{\sharp2}')\omega_u=(W_{\sharp1},W_{\sharp2}),
$$
which together with equations (\ref{coincd}) and (\ref{fundequiv}) establish the first part of the theorem.


 Conversely, let $((G_{\sharp1},G_{\sharp2}),(W_{\sharp1},W_{\sharp2}),\Theta_T)$ and $((G'_{\sharp1},G'_{\sharp2}),(W_{\sharp1}',W_{\sharp2}'),\Theta_{T'})$ be the characteristic triples of $(T_1,T_2)$ and $(T_1',T_2')$, respectively and suppose the respective characteristic triples coincide. Thus there exist unitaries $u:\mathcal{D}_T\to\mathcal{D}_{T'}$ and $u_*:\mathcal{D}_{T^*}\to\mathcal{D}_{T'^*}$ such that part $(i)$ and part $(ii)$ in Definition \ref{coincidence} hold. Let $\omega_u$ be the unitary induced by $u$ as defined in $(\ref{omega-u})$. Then it is easy to see that the unitary
\begin{eqnarray}\label{unitary-coin}
(I_{H^2}\otimes u_*)\oplus \omega_u:H^2(\mathcal{D}_{T^*})\oplus \overline{\Delta_TL^2(\mathcal{D}_T)}\to H^2(\mathcal{D}_{T'^*})\oplus \overline{\Delta_{T'}L^2(\mathcal{D}_{T'})}
\end{eqnarray}intertwines
\begin{align*}
    &\underline{V}=(M_{G_{\sharp1}^*+zG_{\sharp2}}\oplus W_{\sharp1},M_{G_{\sharp2}^*+zG_{\sharp1}}\oplus W_{\sharp 2},M_z\oplus M_{\zeta}|_{\overline{\Delta_T(L^2(\cD_T))}}) \text{ and }\\
    &\underline{V'}=(M_{G_{\sharp1}'^*+zG'_{\sharp2}}\oplus W_{\sharp1}',M_{G_{\sharp2}'^*+zG'_{\sharp1}}\oplus W_{\sharp2}',M_z\oplus M_{\zeta}|_{\overline{\Delta_{T'}(L^2(\cD_{T'}))}}).
\end{align*}Also, the unitary in $(\ref{unitary-coin})$ clearly takes the space $\{\Theta_Tf\oplus \Delta_T\tilde{f}: f\in H^2(\mathcal{D}_T)\}$ onto $\{\Theta_{T'}f\oplus \Delta_{T'}\tilde{f}: f\in H^2(\mathcal{D}_{T'})\}$ and hence
$$
\big(H^2(\mathcal{D}_{T^*})\oplus \overline{\Delta_T L^2(\mathcal{D}_T)}\big)\ominus \{\Theta_Tf\oplus \Delta_T\tilde{f}: f\in H^2(\mathcal{D}_T)\}
$$onto
$$
\big(H^2(\mathcal{D}_{T'^*})\oplus \overline{\Delta_{T'} L^2(\mathcal{D}_{T'})}\big)\ominus \{\Theta_{T'}f\oplus \Delta_{T'}\tilde{f}: f\in H^2(\mathcal{D}_{T'})\}.
$$This implies that the functional models for $(T_1,T_2)$ and $(T_1',T_2')$ as in (\ref{NFmodelPair}) are unitarily equivalent and hence by (\ref{NFmodelPair}) the pairs $(T_1,T_2)$ and $(T_1',T_2')$ are unitarily equivalent also.
\end{proof}
%

\subsection{Admissible triples}
In this subsection, we consider general contractive analytic functions $(\cD, \cD_*, \Theta)$ and do not insist
that $\Theta$ be purely contractive.
We start with a  contractive analytic function $(\cD,\cD_*,\Theta)$, a pair of commuting unitaries $(W_1,W_2)$ and a pair of contractions $(G_1,G_2)$ and investigate when the triple $((G_1,G_2),(W_1,W_2),\Theta)$ gives rise to a pair $(T_1,T_2)$ of commuting contractions such that $T=T_1T_2$ is completely non-unitary.

We note from equation (\ref{Eg1}) that the characteristic triple $((G_{\sharp1},G_{\sharp2}),(W_{\sharp1},W_{\sharp2}),\Theta_T)$ for a pair $(T_1,T_2)$ of commuting contractions satisfies the following set of {\em admissibility
conditions}, where $(i,j)=(1,2),(2,1)$ and we write $(\cD,\cD_*,\Theta)$ in place of $(\cD_{T},\cD_{T^*},\Theta_T)$ and $\Delta_\Theta$ in place of $\Delta_T$:
\smallskip

\begin{definition}  \label{D:admis-cond} {\rm  \textbf{Admissibility conditions:}
 \begin{enumerate}
\item each of $M_{G_{\sharp i}^* + z G_{\sharp j}} \oplus W_{\sharp i}$ is a contraction,
\item $W_{\sharp1}W_{\sharp2}=W_{\sharp2}W_{\sharp1}=M_{\zeta}|_{\overline{\Delta_\Theta L^2(\cD)}}$,
\item $\cQ_{NF}=(\operatorname{Ran}\Pi_{NF})^\perp=\{\Theta f\oplus\Delta_\Theta f:f\in H^2(\cD)\}$ is
jointly invariant under
$(M_{G_{\sharp1}^* + z G_{\sharp2}} \oplus W_{\sharp1},M_{G_{\sharp2}^* + z G_{\sharp1}} \oplus W_{\sharp2},M_z\oplus M_{\zeta}|_{\overline{\Delta_\Theta L^2(\cD)}})$ and
\item with $\cH_{NF}=H^2(\cD_{*})\oplus\overline{\Delta_\Theta L^2(\cD)}\ominus\cQ_{NF}$ we have
$$(M_{G_{\sharp i}^* + z G_{\sharp j}}^*\oplus {W_{\sharp i}}^*)(M_{G_{\sharp j}^* + z G_{\sharp i}}^*\oplus {W_{\sharp j}}^*)|_{\cH_{NF}}=(M_z^*\oplus M_{\zeta}^*|_{\overline{\Delta_\Theta L^2(\cD)}})|_{\cH_{NF}}.$$
\end{enumerate}
}\end{definition}

This motivates us to define the following.
\begin{definition}
Let $(\mathcal{D},\mathcal{D}_*,\Theta)$ be a  contractive analytic function and $(G_1,G_2)$ on $\mathcal{D}_*$
be a pair of contractions. Let $(W_1,W_2)$ be a pair of commuting unitaries on $\overline{\Delta_\Theta L^2(\mathcal{D})}$. We say that the triple $((G_1,G_2),(W_1,W_2),\Theta)$ is admissible if it satisfies
the admissibility conditions $(1)$--$(4)$  in Definition \ref{D:admis-cond}.
We then say that the triple
\begin{align}\label{AdmisFuncModel}
\nonumber&({\bf T}_1,{\bf T}_2,{\bf T}_1{\bf T}_2)\\
&:=P_{\mathcal{H}_\Theta}(M_{G_1^*+zG_2}\oplus W_1,M_{G_2^*+zG_1}\oplus W_2,M_z\oplus M_{\zeta}|_{\overline{\Delta_\Theta(L^2(\cD))}})|_{\mathcal{H}_\Theta}
\end{align}
is the functional model associated with the admissible triple $((G_1,G_2),(W_1,W_2),\Theta)$.
\end{definition}

Let us say that the admissible triple $({\mathbb G}, {\mathbb W}, \Theta)$ is {\em pure} if its
last component $\Theta$ is a purely contractive analytic function.  Then we have the following analogue
of Observation \ref{O:reduction} for the Sz.-Nagy--Foias model.

\begin{proposition}   \label{P:pure-adm-triple}
We suppose that $((G_1, G_2), (W_1, W_2), \Theta)$ is an admissible triple and that $\Theta$ has a
(possibly nontrivial) decomposition $\Theta = \Theta' \oplus \Theta^0$ with $(\cD', \cD'_*, \Theta')$ a unitary
constant function and $\Theta^0$ a purely contractive analytic function.
Then there is an admissible triple of the form $((G_1^0, G_2^0),
(W_1^0, W_2^0), \Theta^0)$ so that the functional model for $((G_1, G_2), (W_1, W_2), \Theta)$
is unitarily equivalent to the functional model for $((G_1^0, G_2^0), (W_1^0, W_2^0), \Theta^0)$.
\end{proposition}

We shall refer to  $((G_1^0, G_2^0), (W_1^0, W_2^0), \Theta^0)$ as the {\em pure part} of
$((G_1, G_2), (W_1, W_2), \Theta)$.

\begin{proof}  We suppose that $((G_1, G_2), (W_1, W_2), \Theta)$ is an admissible triple and that $\Theta$
has a (possibly nontrivial) decomposition $\Theta = \Theta' \oplus \Theta^0$ with
$(\cD', \cD'_*, \Theta')$ a unitary constant function and $(\cD^0, \cD_*^0, \Theta^0)$ a purely contractive
analytic function.  Let
$$
\cH_\Theta = \cK_\Theta \ominus \begin{bmatrix} \Theta \\ \Delta_\Theta \end{bmatrix} H^2(\cD)
$$
be the Sz.-Nagy--Foias functional model space associated with $\Theta$
(and hence also the functional model space associated with the admissible triple
$((G_1, G_2), (W_1, W_2), \Theta)$), and let
$$
({\mathbf T}_1^*, {\mathbf T}_2^*, {\mathbf T}^*) =
\left((M_{G_1^* + z G_2} \oplus W_1)^*, (M_{G_2^* + z G_1} \oplus W_2)^*,
(M_z \oplus M_\zeta|_{\overline{\Delta_\Theta L^2(\cD)}})^* \right) \big|_{\cH(\Theta)}
$$
be the associated functional-model triple of contraction operators. (with ${\mathbf T} = {\mathbf T}_1 {\mathbf T}_2$).
As a result of \cite[Theorem VI.3.1]{Nagy-Foias}, we know that ${\mathbf T}$ is c.n.u.\ with
characteristic function $\Theta_{\mathbf T}$ coinciding with $\Theta^0$.  Thus the characteristic triple for
$({\mathbf T}_1^*, {\mathbf T}_2^*, {\mathbf T}^*)$ has the form
$$
\widetilde \Xi : = ((\widetilde G_1, \widetilde G_2), (\widetilde W_1, \widetilde W_2), \Theta_{\mathbf T})
$$
and by Theorem \ref{Thm:SNFmodelPair} it follows that $({\mathbf T}_1, {\mathbf T}_2, {\mathbf T})$
is unitarily equivalent to the model operators
associated with $\widetilde \Xi$.  As already noted, $\Theta_{\mathbf T}$ coincides with $\Theta^0$;
hence there are unitary operators $u \colon \cD_T \to \cU^0$, $u_* \colon \cD_{T^*} \to \cU_*^0$ so that
$$
 \Theta^0(z) u = u_* \Theta_T(z) \text{ for all } z \in {\mathbb D}.
$$
Define operators $G_1^0$, $G_2^0$ on $\cD_*^0$ and $W_1^0$, $W_2^0$ on
$\overline{\Delta_{\Theta^0} L^2(\cD^0)}$ by
$$
  G_i^0 = u_* \widetilde G_i u_*^*, \quad W_i^0 = (u \otimes I_{L^2} \widetilde W_i (u^* \otimes I_{L^2})
$$
for $i = 1,2$.  Then by construction the triple
$$
  \Xi^0 = \left( (G_1^0, G_2^0), (W_1^0, W_2^0), \Theta^0 \right)
 $$
 coincides with $\widetilde \Xi$ and hence is also admissible.  Then by Theorem \ref{UnitaryInv}
 the commutative contractive pair $({\mathbf T}_1, {\mathbf T}_2)$ is also unitarily equivalent
 to the functional-model commutative contractive pair associated with the admissible triple $\Xi^0$.
This completes the proof of Proposition \ref{P:pure-adm-triple}.
\end{proof}

For $\Theta$ a purely contractive analytic function, we have the following result.

\begin{thm}\label{AdmisCharc}
Let $(\mathcal{D},\mathcal{D}_*,\Theta)$ be a purely contractive analytic function, $(G_1,G_2)$ on $\mathcal{D}_*$ be a pair of contractions and $(W_1,W_2)$ on $\overline{\Delta_\Theta L^2(\mathcal{D})}$ be a pair of
commuting unitaries such that their product is $M_{\zeta}|_{\overline{\Delta_\Theta L^2(\cD)}}$.  Then the  triple
$((G_1,G_2),(W_1,W_2),\Theta)$ is admissible if and only if it is the characteristic triple for
some pair of commuting contractions with their product being a c.n.u. contraction. In fact,
$((G_1,G_2),(W_1,W_2),\Theta)$ coincides with the characteristic triple of its functional model as
defined in (\ref{AdmisFuncModel}).
\end{thm}

\begin{proof}
We have already observed that the characteristic triple of a pair $(T_1,T_2)$ of commuting contractions with $T=T_1T_2$ being a c.n.u.\ contraction is indeed a pure admissible triple (since characteristic functions
$\Theta_T$ are necessarily purely contractive).

Conversely suppose that $((G_1,G_2),(W_1,W_2),\Theta)$ is a pure admissible triple. This means
that the pair $({\bf T}_1,{\bf T}_2)$ defined on
$$
\mathcal{H}_\Theta:=\big(H^2(\mathcal{D_*})\oplus \overline{\Delta_\Theta L^2(\mathcal{D})}\big)\ominus \{\Theta f\oplus \Delta_\Theta f:f\in H^2(\mathcal{D})\}
$$
by
$$
({\bf T}_1,{\bf T}_2):=P_{\mathcal{H}_\Theta}(M_{G_1^*+zG_2}\oplus W_1,M_{G_2^*+zG_1}
\oplus W_2)|_{\mathcal{H}_\Theta}
$$
is a commuting pair of contractions and such that their product is given by
\begin{align}\label{producT}
{\bf T}:={\bf T}_1{\bf T}_2=
P_{\mathcal{H}_{\Theta}}(M_z\oplus M_{\zeta}|_{\overline{\Delta_\Theta(L^2(\cD))}})|_{\mathcal{H}_\Theta}.
\end{align}
By the Sz.-Nagy--Foias model theory for a single contraction operator $T$ (see \cite[Theorem VI.3.1]{Nagy-Foias}),
we conclude that ${\bf T}$ is a c.n.u.\ contraction.
We claim that the triple $((G_1,G_2),(W_1,W_2),\Theta)$
coincides with the characteristic triple for $({\bf T}_1,{\bf T}_2)$, which we assume to be
$((G'_1,G'_2),(W'_1,W'_2),\Theta_{{\bf T}})$. Since $\Theta$ is a purely contractive analytic function,
by (\ref{producT}) and Theorem VI.3.1 in  \cite{Nagy-Foias}, we conclude that $\Theta$ coincides
with $\Theta_{{\bf T}}$.  By definition this  means
that there exist unitaries $u:\mathcal{D}\to\mathcal{D}_{\bf T}$ and $u_*:\mathcal{D}_*\to\mathcal{D}_{{\bf T}^*}$
 such that $\Theta_{{\bf T}}u=u_*\Theta$. Clearly the unitary operator $u_*\oplus\omega_u$ takes
 $H^2(\mathcal{D_*})\oplus\overline{\Delta_\Theta L^2(\mathcal{D})}$ onto
 $H^2(\mathcal{D}_{{\bf T}^*})\oplus\overline{\Delta_{{\bf T}} L^2(\mathcal{D}_{{\bf T}})}$.
 Denote by $\tau$ the restriction of $u_*\oplus\omega_u$ to $\mathcal{H}_\Theta$.
 The following commutative diagram, where $i$ and $i'$ are the inclusion maps and
 $K_\Theta:= H^2(\mathcal{D_*})\oplus \overline{\Delta_\Theta L^2(\mathcal{D})}$,
$$
\begin{CD}
\mathcal{H}_\Theta @> i>> \mathcal K_\Theta\\
@V\tau VV @VVu_*\oplus\omega_u V\\
\mathcal{H}_{NF}@>>i'> \mathcal{K}_{NF}
\end{CD}
$$
shows that if we denote by $\underline{V}$ and $\underline{V'}$ the respective triples
\begin{align*}
&\underline{V}=(M_{G_1^*+zG_2}\oplus W_1,M_{G_2^*+zG_1}\oplus W_2,M_z\oplus M_{\zeta}|_{\overline{\Delta_\Theta(L^2(\cD))}}),  \\
&\underline{V'}=(M_{G_1'^*+zG_2'}\oplus W_1',M_{G_2'^*+zG_1'}\oplus W_2',M_z\oplus M_{\zeta}|_{\overline{\Delta_{{\bf T}}(L^2(\cD_{\bf T}))}}),
\end{align*}
then $(i,\mathcal{K}_\Theta, \underline{V})$ and $(i'\circ\tau,\mathcal{K}_{NF}, \underline{V'})$ both are in $\mathcal{U}_{\underline{T}}$.
Applying the uniqueness result Theorem \ref{uniqueness}, we get a unitary from $\mathcal{K}_\Theta$ onto $\mathcal{K}_{NF}$ that intertwines $\underline{V}$ and $\underline{V'}$ and the restriction of which to $\mathcal{H}_\Theta$ is $\tau$. Since the last entries of $\underline{V}$ and $\underline{V'}$ are the minimal isometric dilations of $T=T_1T_2$, such a unitary is in fact unique as a consequence of Theorem \ref{T:min-iso-lift}. Since $u_*\oplus\omega_u$ is one such unitary, we get
$$
(u_*\oplus\omega_u)\underline{V}=\underline{V'}(u_*\oplus\omega_u).
$$
Consequently $((G_1,G_2),(W_1,W_2),\Theta)$ coincides with $((G'_1,G'_2),(W'_1,W'_2),\Theta_{{\bf T}})$
and the theorem follows.
\end{proof}

The results of this and the previous subsection can be stated more succinctly in the language of Category Theory
as follows.

\begin{proposition}   \label{P:catlan}
Define the following categories:
\begin{enumerate}
\item[(i)]
Let ${\mathfrak C}_1$ be the category of all commuting pairs of contraction operators ${\mathbf T} = (T_1, T_2)$
where we set $T = T_1 \cdot T_2 = T_2 \cdot T_1$ and we assume that $T$ is c.n.u.
\item[(ii)] Let ${\mathfrak C}_2$ be the category of all purely contractive admissible triples
$({\mathbb G}, {\mathbb W}, \Theta)$.
\end{enumerate}
Define functors ${\mathfrak f} \colon {\mathfrak C}_1 \to {\mathfrak C}_2$ and
${\mathfrak g} \colon {\mathfrak C}_2 \to {\mathfrak C}_1$  by
\begin{align*}
& {\mathfrak f} \colon {\mathbf T} \mapsto ({\mathbb G}, {\mathbb W}, \Theta)_{\mathbf T} =
\text{ characteristic triple for } {\mathbf T}, \\
& {\mathfrak g}  \colon ({\mathbb G}, {\mathbb W}, \Theta) \mapsto {\mathbf T}_{{\mathbb G}, {\mathbb W}, \Theta}
= \text{functional-model commutative contractive pair} \\
& \quad \quad \quad \quad \text{ associated with } ({\mathbb G}, {\mathbb W}, \Theta)
\text{ as in } \eqref{AdmisFuncModel}.
\end{align*}
Then, for ${\mathbf T}, {\mathbf T}' \in {\mathfrak C}_1$ and $({\mathbb G}, {\mathbb W}, \Theta),
({\mathbb G}', {\mathbb W}', \Theta')  \in {\mathfrak C}_2$, we have
 \begin{enumerate}
\item  ${\mathbf T} \underset{u}\cong {\mathbf T'} \Leftrightarrow
{\mathfrak f}({\mathbf T}) \underset{c}\cong {\mathfrak f}({\mathbf T}')$,

\smallskip

\item  $({\mathbb G}, {\mathbb W}, \Theta) \underset{c} \cong ({\mathbb G}', {\mathbb W}', \Theta') \Leftrightarrow
{\mathfrak g}({\mathbb G}, {\mathbb W}, \Theta) \underset{u} \cong
{\mathfrak g}({\mathbb G}', {\mathbb W}', \Theta')$,

\smallskip

\item ${\mathfrak g} \circ {\mathfrak f}({\mathbf T}) \underset{u}\cong {\mathbf T}$,

\smallskip

\item ${\mathfrak f} \circ {\mathfrak g}(({\mathbb G}, {\mathbb W}, \Theta)) \underset{c}\cong
({\mathbb G}, {\mathbb W}, \Theta)$
\end{enumerate}
where $\underset{u} \cong$ denotes {\em unitary equivalence of operator tuples} and
$\underset{c} \cong$ denotes {\em coincidence of admissible triples}.
\end{proposition}

\subsection{Concrete Examples}
In this subsection, we shall exhibit several examples of characteristic triples. In particular, in the following result we characterize when a triple of the form $(g_1,g_2,\vartheta)$ can be admissible, where $g_1,g_2$ are scalars and $\vartheta$ is a Blaschke function with simple non-zero zeros.

\begin{proposition}  \label{P:concrete}
Let $(g_1,g_2)\in\mathbb{C}$ and  $\vartheta$ be a scalar-valued inner function.   If
\begin{enumerate}
\item  If  $\vartheta(z)=z^N$  for some positive integer $N$,  or
\item  If $\vartheta(z)=z^N\prod_{j=1}^M\frac{z-a_j}{1-\overline{a_j}z}$
for some positive integers $N$ and $M$ and distinct nonzero $a_j \in {\mathbb D}$ for $j=1, \dots, M$,
\end{enumerate}
then $(g_1,g_2,\vartheta)$ is admissible if and only if
\begin{equation}  \label{criterion}
\text{either } (g_1, g_2) \text{ or } (g_2,g_1) \text{  is in  } \overline{\mathbb{D}}\times \{0\}.
\end{equation}
\end{proposition}

\begin{proof}
Let us define $\varphi_1(z)=\overline{g_1}+zg_2$ and $\varphi_2(z)=\overline{g_2}+zg_1$.
For a function $\varphi$, let us denote by $\widetilde{\varphi}$ the function
$\widetilde{\varphi}(z)=\overline{\varphi(\bar{z})}$.
Then note that for non-zero $z$,
$$
z \widetilde{\varphi_1}(\frac{1}{z})=z(g_1+\overline{g_2}\frac{1}{z})
= z g_1 + \overline{g_2} =\varphi_2(z).
$$
By definition, admissibility of $(g_1,g_2,\vartheta)$ means the following three conditions hold true:
\begin{enumerate}
\item the multiplication operators $M_{\varphi_1}$ and $M_{\varphi_2}$ are contractions on $H^2(\mathbb{D})$;
\item $\mathcal{M}:=\vartheta H^2(\mathbb{D})$ is joint $(M_{\varphi_1},M_{\varphi_2})$-invariant; and
\item $M_{\varphi_1}^*M_{\varphi_2}^*|_{\mathcal{M}^\perp}=M_{\varphi_2}^*M_{\varphi_1}^*|_{\mathcal{M}^\perp}=M_z^*|_{\mathcal{M}^\perp}.$
\end{enumerate}
It is well known that condition (1) is equivalent to $\|\varphi_i\|_{\infty,\mathbb{D}}$ being at most $1$, for $i=1,2$.
We see that condition (2) is automatic since we are in the scalar case. It remains to analyze condition (3).
For this analysis we must handle the two cases $\vartheta(z) = z^N$ and $\vartheta(z) = z^N \prod_{j=1}^M
\frac{z-a_j}{1 - \overline{aJ} z}$ separately.

\smallskip

\noindent
\textbf{Case 1: $\theta(z) = z^N$.}  For simplicity we consider only the case $N=2$ as the case $N=1$ is trivial
and $N\ge 3$ is similar but with a heavier notational burden.  Then $\cM^\perp$ has a orthonormal basis consisting
of $\{1, z\}$.  With respect to this orthonormal basis, the relevant operators have matrix representations
\begin{align*}
& M_{\varphi_1}^*|_{\cM^\perp} = \sbm{ \overline{\varphi_1(0)} & \overline{ \varphi'_1(0)} \\ 0 & \overline{\varphi_1(0)} }
= \sbm{ g_1 & \overline{g_2} \\ 0 & g_1 }, \quad
M_{\varphi_2}^*|_{\cM^\perp} = \sbm{ \overline{\varphi_2(0)} & \overline{ \varphi'_2(0)} \\ 0 & \overline{\varphi_2(0)} }
= \sbm{ g_2 & \overline{g_1} \\ 0 & g_2 }  \\
& M_z^*|_{\cM^\perp} = \sbm{  0 & 1 \\ 0 & 1 }.
\end{align*}
Condition (3) then comes down to the requirement that
$$
\sbm{ g_1 & \overline{g_2} \\ 0 & g_1 } \sbm{ g_2 & \overline{g_1} \\ 0 & g_2 } = \sbm{0 & 1 \\ 0 & 0}
$$
which is to say
$$
\sbm{ g_1 g_2 & | g_1 |^2 + | g_2 |^2 \\ 0 & g_1 g_2 } = \sbm{0 & 1 \\ 0 & 0}
$$
which in turn reduces to the condition that either $(g_1, g_2)$ or $(g_2, g_1)$ is in ${\mathbb T} \times \{0\}$.
This completes the verification for Case 1.

\smallskip

\noindent
\textbf{Case 2:  $\vartheta(z) =  z^N \prod_{j=1}^M \frac{ z - a_j}{1 - \overline{a_j} z}$.}
In this case
$$
\cM^\perp=\operatorname{span}\{1,k_{a_j}:1\leq j \leq M\}.
$$
Let us impose the convention that $a_0 = 0$.  Then
condition (3) is equivalent to the set of interpolation constraints:
\begin{align} \label{Interpol}
&\varphi_1(a_j)\varphi_2(a_j)=a_j, \text{ for } 0 \leq j \leq N \notag \\
& \Leftrightarrow (\overline{g_1}+a_jg_2)(\overline{g_2}+a_jg_1)=a_j \text{ for } 0\leq j \leq N \notag \\
& \Leftrightarrow (|g_1|^2+|g_2|^2)a_j+\overline{g_1g_2}+g_1g_2a_j^2=a_j\text{ for  } 0\leq j \leq N.
\end{align}
Putting $j=0$ in (\ref{Interpol}), we get $g_1g_2=0$, which implies that either or both of $g_1$, $g_2$ are $0$. But $g_1=0=g_2$ violates the last equivalent statement in \eqref{Interpol} once $j > 0$. Therefore we assume that $g_1=0$ and $g_2$ is any non-zero number of modulus at most one.  Then the final equivalent statement in  (\ref{Interpol}) collapses to $|g_2|^2 a_j = a_j$.  As $a_j \ne 0$ for $j > 0$, we conclude that $|g_2|^2 = 1$,
i.e., $g_2 \in {\mathbb T}$.  This completes the verification of Case 2.
\end{proof}

\begin{remark}   \label{R:interpretation}  {\rm We note the following interpretation of the criterion \eqref{criterion}
for admissibility of a triple of the form $(g_1, g_2, \vartheta(z))$ with $g_1, g_2$ complex numbers and $\vartheta$
a scalar inner function:}  Suppose that $T$ is a c.n.u.\ contraction operator on a Hilbert space $\cH$
 with characteristic function $\vartheta$,
and the criterion for complex numbers $g_1$, $g_2$ to be such that $(g_1, g_2, \vartheta)$ is an admissible
triple is given by criterion \eqref{criterion}.  Suppose that $T_1$, $T_2$ is a commutative pair of contractions  on $\cH$
such that $T = T_1 \cdot T_2$.  Then there is a number $\omega$ on the unit circle so that
$T_1 = \omega I_{{\mathbb C}^N}$, $T_2 =\overline{\omega} T$, or the reverse.
\end{remark}

We next give a perhaps somewhat less trivial example of a characteristic triple for a commutative contractive
pair $(T_1, T_2)$ on ${\mathbb C}^2$.

\begin{example}
{\rm Let $a,b,x,y$ be any four complex numbers with moduli strictly less than one such that $ay=bx$.
One can check that the triple
$$
\left(\begin{bmatrix}
      0 & \frac{a(1-|y|^2)}{1-|ay|^2} \\
      \frac{b(1-|x|^2)}{1-|bx|^2} & 0 \\
     \end{bmatrix},\begin{bmatrix}
      0 & \frac{x(1-|b|^2)}{1-|xb|^2} \\
      \frac{y(1-|a|^2)}{1-|ay|^2} & 0 \\
     \end{bmatrix},\frac{z-ay}{1-\overline{ay}z}I_{\mathbb{C}^2}\right)
$$
is the characteristic triple for the commutative, contractive pair
\begin{align}\label{NRegExamp}
T_1=\begin{bmatrix}
      0 & a \\
      b & 0 \\
     \end{bmatrix},\;
T_2=\begin{bmatrix}
      0 & x \\
      y & 0 \\
     \end{bmatrix}.
\end{align}
}\end{example}

\section{Characterization of invariant subspaces for pairs of commuting contractions}  \label{S:invsub}
In this section we characterize invariant subspaces for pairs $(T_1,T_2)$ of commuting contractions
such that $T=T_1T_2$ is a c.n.u.\ contraction. Sz.-Nagy and Foias characterized how invariant subspaces for
c.n.u.\ contractions arise in the functional model. They showed that invariant subspaces of a c.n.u. contraction $T$
are in one-to-one correspondence with regular factorizations of the characteristic function of $T$.  A minor
complication in the theory is that the factors in a regular factorization of a purely contractive analytic
function need not again be purely contractive.  We now recall
their result  as we shall have use of it later in this section.

\begin{thm}[Sz.-Nagy--Foais]\label{NFcnu}
Let $(\cD,\cD_*,\Theta)$ be a purely contractive analytic function and $\bf T$ be the contraction on
\begin{align}\label{cnuMspace}
\bH=H^2(\cD_*)\oplus\overline{\Delta_{\Theta} L^2(\cD)}\ominus\{\Theta f\oplus\Delta_{\Theta} f: f\in H^2(\cD)\}
\end{align}defined by
\begin{align}\label{NFMop}
  {\bf T}=P_{\bH}\big(M_z\oplus M_{\zeta}\big)|_{\bH}.
\end{align}A subspace $\bH'$ of $\bH$ is invariant under ${\bf T}$ if and only if there exist  contractive analytic functions $(\mathcal{D},\mathcal{F},\Theta')$, $(\mathcal{F},\mathcal{D}_{*},\Theta'')$ such that
\begin{align}\label{cnuRegFact}
  \Theta=\Theta''\Theta'
\end{align}is a regular factorization, and with the unitary $Z$ as in (\ref{Z}) we have
\begin{align}\label{cnuH'}
\bH'=\{\Theta''f\oplus Z^{-1}(\Delta_{\Theta''}f\oplus g):\;f\in H^2(\mathcal{F}),g\in\overline{\Delta_{\Theta'}L^2(\mathcal{D})}\}\\\nonumber
\ominus\{\Theta h\oplus\Delta_{\Theta}h:h\in H^2(\mathcal{D})\}
\end{align}
and
\begin{align}\label{cnuH''}
\bH'' := \bH\ominus\bH'=H^2(\cD_*)\oplus &Z^{-1}(\overline{\Delta_{\Theta''}L^2(\mathcal{F})}\oplus\{0\})\\\nonumber
&\ominus\{\Theta''f\oplus Z^{-1}(\Delta_{\Theta''}f\oplus 0):f\in H^2(\mathcal{F})\}.
\end{align}
Moreover, the characteristic function of ${\mathbf T}|_{{\mathbb H}'}$ coincides with the purely contractive part
of $\Theta'$, and the characteristic function of $P_{{\mathbb H}''} {\mathbf T}|_{{\mathbb H}''}$
coincides with the purely contractive part of $\Theta''$.
\end{thm}

Let $T$ be a c.n.u.\ contraction such that $T=T_1T_2$ for a pair $(T_1,T_2)$ of commuting contractions.
It is natural that one would need more conditions than (\ref{cnuH'}) and (\ref{cnuH''}) for an invariant
subspace of $T$ to be jointly invariant under $(T_1,T_2)$.

\begin{thm}\label{T:inv-cnu}
Let $(\cD,\cD_*,\Theta)$ be a pure contractive analytic function and let  the triple $((G_1,G_2),(W_1,W_2),\Theta)$
be admissible. Define the pair $({\bf T_1},{\bf T_2})$ of commuting contractions on
\begin{align}\label{cnuMspace}
\bH=H^2(\cD_*)\oplus\overline{\Delta_{\Theta} L^2(\cD)}\ominus\{\Theta f\oplus\Delta_{\Theta} f: f\in H^2(\cD)\}
\end{align}by
\begin{align}\label{Mop}
 ({\bf T_1},{\bf T_2})=P_{\bH}\big(M_{G_1^*+zG_2}\oplus W_1,M_{G_2^*+zG_1}\oplus W_2\big)|_{\bH}.
\end{align}
A subspace $\bH'$ of $\bH$ is jointly invariant under $({\bf T_1},{\bf T_2})$ if and only if there exist two contractions $G'_1$, $G'_2$ in $\mathcal{B}(\mathcal{F})$, a pair $(W'_1,W'_2)$ of unitary operators
on $\overline{\Delta_{\Theta'}L^2(\cD)}$ with the property
\begin{align}\label{propW'}
W'_1W'_2=W'_2W'_1=M_{\zeta}|_{\overline{\Delta_{\Theta'}L^2(\cD)}}
\end{align}
along with two contractive analytic functions $(\mathcal{D},\mathcal{F},\Theta')$,
$(\mathcal{F},\mathcal{D}_{*},\Theta'')$ such that
\begin{eqnarray*}
\Theta=\Theta''  \Theta'
\end{eqnarray*}
is a regular factorization and also, with  $Z$ the pointwise unitary operator  as in (\ref{Z}),
\begin{align}\label{H'}
\bH'=\{\Theta''f\oplus Z^{-1}(\Delta_{\Theta''}f\oplus g):\;f\in H^2(\mathcal{F}),g\in\overline{\Delta_{\Theta'}L^2(\cD)}\}\\\nonumber
\ominus\{\Theta h\oplus\Delta_{\Theta}h:h\in H^2(\cD)\},
\end{align}
\begin{align}\label{H''}
\bH'':=\bH\ominus\bH'=&\;H^2(\cD_*)\oplus Z^{-1}(\overline{\Delta_{\Theta''}L^2(\mathcal{F})}\oplus\{0\})\\&\nonumber\ominus\{\Theta''f\oplus Z^{-1}(\Delta_{\Theta''}f\oplus 0):f\in H^2(\mathcal{F})\},
\end{align}
and for every $f\in H^2(\mathcal{F})$ and $g\in\overline{\Delta_{\Theta'}L^2(\cD)}$
\begin{align}\label{ExtraCond}
\begin{bmatrix}
      M_{G_i^*+zG_j} & 0 \\
      0 & W_i \\
     \end{bmatrix}
\begin{bmatrix}
\Theta''f\\Z^{-1}(\Delta_{\Theta''}f\oplus g)
\end{bmatrix}
=\begin{bmatrix}
\Theta''M_{G'_i+zG'_j}f\\Z^{-1}(\Delta_{\Theta''}M_{G'_i+zG'_j}f\oplus W'_i g)
\end{bmatrix},
\end{align}where $(i,j)=(1,2),(2,1)$.
\end{thm}

\begin{proof}
We first prove the easier part---the proof of sufficiency.   Suppose that
$$
((G_1,G_2),(W_1,W_2),\Theta)
$$
is a purely contractive admissible triple (i.e., $( \cD, \cD_*, \Theta)$ is a purely contractive analytic function)
such that $\Theta$ has a regular factorization $\Theta = \Theta'' \Theta'$ with $(\cD, \cF, \Theta')$ and
$(\cF, \cD_*, \Theta'')$ contractive analytic functions.  We suppose also that $G_1'$ and $G_2'$ are
contraction operators on $\cF$, $W_1'$, $W_2'$ are unitary operators on $\overline{\Delta_{\Theta'} L^2(\cD)}$
so that \eqref{propW'} and \eqref{ExtraCond} hold.  Then we have all the ingredients to define ${\mathbb H}'$
and ${\mathbb H}''$ as in \eqref{H'} and \eqref{H''}.  Note next that ${\mathbb H}'$ is indeed a subspace of
${\mathbb H}$.  We wish to show that the space $\bH'$ given in (\ref{H'}) is
jointly invariant under the pair $({\bf T_1},\bf{T_2})$ defined in (\ref{Mop}). Firstly, it is easy to see that $\bH'$
is a subspace of $\bH$. Since the operator
\begin{align}\label{iso-cnu}
{\mathcal I} \colon H^2(\mathcal{F})\oplus\overline{\Delta_{\Theta'}L^2(\cD)} &\to
H^2(\cD_*)\oplus\overline{\Delta_{\Theta}L^2(\cD)}\\\nonumber
f\oplus g &\mapsto \Theta''f\oplus Z^{-1}(\Delta_{\Theta''}f\oplus g)
\end{align}
is an isometry, the space
$$
\{\Theta''f\oplus Z^{-1}(\Delta_{\Theta''}f\oplus g):f\in H^2(\mathcal{F})
\text{ and } g \in \overline{\Delta_{\Theta'}L^2(\cD)}\}
$$
is closed and by (\ref{ExtraCond}) we see that it is jointly invariant under
$$
(M_{G_1^*+zG_2}\oplus W_1,M_{G_2^*+zG_1}\oplus W_2,M_z\oplus M_{\zeta}).$$
We also see that
$$
\operatorname{Ran}\, \mathcal{I}=\big(H^2(\cD_*)\oplus \overline{\Delta_{\Theta} L^2(\cD)} \big)
\ominus(\bH\ominus\mathbb{H}').
$$
Now the sufficiency follows from the definition of $({\bf T_1},\bf{T_2})$ and from the general fact that if $V$
is an operator on $\mathcal{K}$ containing $\mathcal{H}$, $V(\cK\ominus\cH)\subset\cK\ominus\cH$,
and $V^*|_{\mathcal{H}}=T^*$, then for a subspace $\mathcal{H}'$ of $\mathcal{H}$,
$$
V(\mathcal{K}\ominus(\mathcal{H}\ominus\mathcal{H}'))\subseteq (\mathcal{K}\ominus(\mathcal{H}\ominus\mathcal{H}')) \text{ if and only if } T(\mathcal{H}')\subseteq\mathcal{H}'.
$$

Now we show that the conditions are necessary. The first step of the proof is an application of
Theorem \ref{NFcnu}. Indeed, if $\bH'\subset\bH$ is jointly invariant under $({\bf T_1},\bf{T_2})$,
then it is also invariant under the product ${\bf T_1T_2}$ and by definition of admissibility
$$
{\bf T}={\bf T_1T_2}={\bf T_2T_1}=P_{\bH}\big(M_z\oplus M_{\zeta}\big)|_{\bH}.
$$
Hence by Theorem \ref{NFcnu}, there exist two contractive analytic functions
$$
(\mathcal{D},\mathcal{F},\Theta'), \quad (\mathcal{F},\mathcal{D}_{*},\Theta'')
$$
such that $\Theta=\Theta''\Theta'$ is a regular factorization and the
spaces $\bH'$ and $\bH''$ are realized as in (\ref{H'}) and (\ref{H''}), respectively. It only remains to
produce contraction operators $G_1'$, $G_2'$ on $\cF$ and unitary operators $W_1'$, $W_2'$ on
$\overline{\Delta_{\Theta'} L^2(\cD)}$ so that conditions \eqref{propW'} and \eqref{ExtraCond} hold.
Note that, once we have found
$G_1'$, $G_2'$, $W_1'$, $W_2'$, verification of \eqref{ExtraCond} breaks up into three linear pieces,
where $(i,j) = (1,2)$ or $(2,1)$:
\begin{align}
& M_{G_i^* + z G_j} \Theta'' f = \Theta'' M_{G_i' + z G'_j} f \text{ for all } f \in H^2(\cF), \label{verify1} \\
& W_i (Z^{-1} (\Delta_{\Theta''} f \oplus 0)= Z^{-1} (\Delta_{\Theta''} M_{G_i' + z G_j'} f \oplus 0)
\text{ for all } f \in H^2(\cF),  \label{verify2} \\
& W_i Z^{-1} (0 \oplus g) = Z^{-1}(0 \oplus  W_i' g) \text{ for all } g \in \overline{\Delta_{\Theta'} L^2(\cD)}.
\label{verify3}
\end{align}

As a first step, we define operators $X_i$ on $H^2(\mathcal{F})$ and $W_i'$ on
$\overline{\Delta_{\Theta'}L^2(\cD)}$, for $i=1,2$, such that for every $f\in H^2(\mathcal{F})$ and
$g\in\overline{\Delta_{\Theta'}L^2(\cD)}$,
\begin{align}\label{Intertwining2}
\nonumber\begin{bmatrix}
      M_{G_i^*+zG_j} & 0 \\
      0 & W_i \\
     \end{bmatrix}\mathcal{I}
\begin{bmatrix}
f\\g
\end{bmatrix}&=
\begin{bmatrix}
M_{G_i^*+zG_j} & 0 \\
      0 & W_i \\
\end{bmatrix}
\begin{bmatrix}
\Theta''f\\Z^{-1}(\Delta_{\Theta''}f\oplus g)
\end{bmatrix}\\
&=\begin{bmatrix}
\Theta''X_if\\Z^{-1}(\Delta_{\Theta''}X_if\oplus W_i'g)
\end{bmatrix}=
\mathcal{I}
\begin{bmatrix}
     X_i & 0 \\
      0 & W_i' \\
\end{bmatrix}
\begin{bmatrix}
f\\g
\end{bmatrix},
\end{align}
where $\cI$ is the isometry as defined in (\ref{iso-cnu}). The operators $X_1,X_2$ and $W_1',W_2'$ are
well-defined because the operator $\mathcal{I}$ is an isometry. Indeed, it follows that $X_1,X_2$ and
$W_1',W_2'$ are contractions. Since the unitary $Z$ commutes with $M_{\zeta}$, it is easy to see from the
definition of $\mathcal I$ that it has the following intertwining property
\begin{equation} \label{Intertwining1}
  \mathcal{I} (M_z \oplus M_{\zeta}|_{\overline{\Delta_{\Theta'}L^2(\cD)}}) =
  (M_z\oplus M_{\zeta}|_{\overline{\Delta_{\Theta}L^2(\cD)}})\mathcal{I}.
\end{equation}
From the intertwining properties (\ref{Intertwining1}) and (\ref{Intertwining2}) of $\mathcal{I}$, we get for $i=1,2$
$$
(X_i\oplus W_i')(M_z\oplus M_{\zeta}|_{\overline{\Delta_{\Theta'}L^2(\cD)}})=
(M_z\oplus M_{\zeta}|_{\overline{\Delta_{\Theta'}L^2(\cD)}})(X_i\oplus W_i'),
$$
which implies that $(X_1,X_2)=(M_{\varphi_1},M_{\varphi_2})$, for some $\varphi_1$ and $\varphi_2$ in
$L^\infty(\mathcal{B}(\mathcal{F}))$. We next show that $\varphi_1$ and $\varphi_2$ are actually linear pencils.
 Toward this end,  notice from (\ref{Intertwining2}) that
\begin{align}\label{subeqn1}
  \begin{bmatrix}
     M_{\varphi_1} & 0 \\
      0 & W_1' \\
\end{bmatrix}&= \mathcal{I}^*
\begin{bmatrix}
      M_{G_1^*+zG_2} & 0 \\
      0 & W_1 \\
\end{bmatrix}\mathcal{I} \\\label{subeqn}
\begin{bmatrix}
     M_{\varphi_2} & 0 \\
      0 & W_2' \\
\end{bmatrix}&= \mathcal{I}^*
\begin{bmatrix}
      M_{G_2^*+zG_1} & 0 \\
      0 & W_2 \\
\end{bmatrix}\mathcal{I}
\end{align}
Now multiplying (\ref{subeqn1}) on the left by $M_z^*\oplus M_{\zeta}^*|_{\overline{\Delta_{\Theta'}L^2(\cD)}}$,
then using the intertwining property (\ref{Intertwining1}) of $\cI$ and then remembering that $(W_1,W_2)$ is a commuting pair of unitaries such that $W_1W_2=M_{\zeta}^*|_{\overline{\Delta_{\Theta}L^2(\cD)}}$, we get
$$
\begin{bmatrix}
     M_{z}^* & 0 \\
      0 & M^*_{\zeta}|_{\overline{\Delta_{\Theta'}L^2(\cD)}} \\
\end{bmatrix}
\begin{bmatrix}
     M_{\varphi_1} & 0 \\
      0 & W_1' \\
\end{bmatrix}
=\mathcal{I}^*
\begin{bmatrix}
      M^*_{G_2^*+zG_1} & 0 \\
      0 & W_2^* \\
\end{bmatrix}\mathcal{I}=
\begin{bmatrix}
     M_{\varphi_2}^* & 0 \\
      0 & W'^*_2 \\
\end{bmatrix}.
$$
Consequently, $M_{\varphi_2}=M_{\varphi_1}^*M_z$. A similar argument as above yields
$M_{\varphi_1}=M_{\varphi_2}^*M_z$. Considering these two relations and the power series
expansions of $\varphi_1$ and $\varphi_2$, we get
\begin{equation}   \label{varphi}
\varphi_1(z)=G_1'^*+zG_2' \text{ and }\varphi_2(z)=G_2'^*+zG_1',
\end{equation}
for some $G_1',G_2'\in\mathcal{B}(\mathcal{F})$. The fact that $M_{\varphi_1}$ (and $M_{\varphi_2}$)
is a contraction implies that $G_1'$ and $G_2'$ are contractions too.  Recalling \eqref{Intertwining2}
and the substitution $(X_1, X_2) = (M_{\varphi_1}, M_{\varphi_2})$ where $\varphi_1$ and $\varphi_2$
are given by \eqref{varphi},  we see that we have established \eqref{verify1} with the choice of $G_1'$, $G_2'$
as in \eqref{varphi}.

Next note that the bottom component of  \eqref{Intertwining2} gives us
\begin{equation}          \label{ConseqInt2}
W_iZ^{-1}(\Delta_{\Theta''}f\oplus g)=Z^{-1}(\Delta_{\Theta''}X_if\oplus W_i'g).
\end{equation}
for all $f \in H^2(\cF)$, $g \in \overline{\Delta_{\Theta'} L^2(\cD)}$, and $i=1,2$.
In particular, setting $g=0$  and recalling that $X_i = M_{\varphi_i} = M_{G_i^{\prime *} + z G_j'}$, we get
\begin{equation}   \label{verify2'}
W_i  Z^{-1}  (\Delta_{\Theta''}   f \oplus 0)       = Z^{-1} (\Delta_{\Theta''} M_{G_i^{\prime *} + z G_j'} f \oplus 0),
\end{equation}
thereby verifying \eqref{verify2}. We next consider \eqref{ConseqInt2} with $f=0$ and $g$ equal to a general element
of $\overline{\Delta_{\Theta'} L^2(\cD)}$ to get
\begin{equation}   \label{verify3'}
W_i  Z^{-1}(0\oplus g)=Z^{-1}(0\oplus W'_i g),
\end{equation}
thereby verifying \eqref{verify3} and hence also completing the proof of \eqref{ExtraCond}.

It remains to show that $(W_1', W_2')$ is a commuting pair of unitary operators satisfying condition
\eqref{propW'}.  Toward this goal, let us rewrite \eqref{verify3'} in the form
\begin{equation}  \label{verify3''}
 Z W_i Z^{-1} (0 \oplus g) = 0 \oplus W_i' g.
\end{equation}
which implies that, for $i=1,2$,
$$
ZW_iZ^{-1}(\{0\}\oplus\overline{\Delta_{\Theta'}L^2(\cD)})\subseteq(\{0\}\oplus\overline{\Delta_{\Theta'}L^2(\cD)}).
$$
On the other hand, using \eqref{verify2'} and noting that $M_{\zeta}|_{\overline{\Delta_{\Theta}L^2(\cD)}}$ commutes with $Z,W_1,W_2$ and $\Delta_{\Theta''}$, we get for every $f\in H^2(\mathcal{F})$ and $n\geq 0$
$$
ZW_iZ^{-1}(\Delta_{\Theta''}e^{-int}f\oplus 0)=(\Delta_{\Theta''}e^{-int}X_if\oplus 0),
$$
which implies that $ZW_iZ^{-1}(\overline{\Delta_{\Theta}L^2(\cD)}\oplus \{0\})\subseteq
(\overline{\Delta_{\Theta'}L^2(\cD)}\oplus\{0\})$, for $i=1,2$. We conclude that $Z^{-1} ( \{0\} \oplus
\overline{\Delta_{\Theta'} L^2(\cD)})$ is a reducing subspace for the pair of unitaries $(W_1, W_2)$
and hence $(W_1, W_2) |_{Z^{-1} (\{0\} \oplus \overline{\Delta_{\Theta'} L^2(\cD)}}$ is a pair of
commuting unitary operators.  The intertwining \eqref{verify3''} shows that the pair $(W_1', W_2')$ on
$\overline{\Delta_{\Delta'} L^2(\cD)}$  is jointly unitarily equivalent to the commutative unitary pair
$(W_1, W_2)|_{Z^{-1}(\{0\} \oplus \overline{\Delta_\Theta'} L^2(\cD)}$ and hence is itself a commutative
unitary pair.  Furthermore, since $W_1 W_2 = M_\zeta$ in particular on $Z^{-1} ( \{0\} \oplus \overline{\Delta_{\Theta'}
L^2(\cS)})$ and $M_\zeta$ commutes past $Z$ and $Z^{-1}$, we conclude that
condition \eqref{propW'} holds as well.
This completes the proof of the necessary part.
\end{proof}

As we see from the last part of the statement of Theorem \ref{T:inv-cnu}, Sz.-Nagy and Foias went on to
prove that, under the conditions of Theorem \ref{NFcnu}, the characteristic functions
of $\bf T|_{\bH'}$ and $P_{\bH\ominus\bH'} {\bf T}|_{\bH\ominus\bH'}$ coincide with the purely contractive parts of
$\Theta'$ and $\Theta''$, respectively. Below we find an analogous result (at least for the first part of this statement)
for pairs of commuting contractions.
The strategy of the proof is the same as that of Sz.-Nagy--Foias, namely: application of model theory.

\begin{thm}
Under the conditions of Theorem \ref{T:inv-cnu}, let $\bH'$ be a joint invariant subspace of $\bH$ induced by the
regular factorization $\Theta=\Theta''\Theta'$. Then with the notations as in Theorem \ref{T:inv-cnu},
the triple $((G_1',G_2'),(W_1',W_2'),\Theta')$ is admissible and its purely contractive part coincides
with the characteristic triple for $({\bf T_1},{\bf T_2})|_{\bH '}$.
\end{thm}

\begin{proof}
With the isometry $\cI$ as in (\ref{iso-cnu}), define a unitary $U:=\cI^*|_{\operatorname{Ran }\cI}$. Therefore
\begin{align}\label{iso-cnu^*}
\nonumber
U:\{\Theta''f\oplus Z^{-1}(\Delta_{\Theta''}f\oplus g):f\in H^2(\mathcal{F}),\;g\in\overline{\Delta_{\Theta'}L^2(\cD)}\}
&\to H^2(\mathcal{F})\oplus\overline{\Delta_{\Theta'}L^2(\cD)} \\
U \colon \Theta''f\oplus Z^{-1}(\Delta_{\Theta''}f\oplus g) \mapsto  f\oplus g.
\end{align}
For every $g\in H^2(\cD)$,
\begin{eqnarray}\label{ortho}
U(\Theta g\oplus\Delta_{\Theta}g)=U(\Theta''\Theta'g\oplus Z^{-1}(\Delta_{\Theta''}\Theta'g\oplus\Delta_{\Theta'}g))
=\Theta'g\oplus\Delta_{\Theta'}g,
\end{eqnarray}which implies that $U$ takes $\bH'$ as given in (\ref{H'}) onto the Hilbert space
\begin{align}\label{fracH'}
\mathfrak{H}':=H^2(\mathcal{F})\oplus\overline{\Delta_{\Theta'}L^2(\cD)}\ominus\{\Theta'g\oplus
\Delta_{\Theta'}g:g\in H^2(\cD)\}
\end{align}
The basis of the proof is the following unitary equivalences:
\begin{align}
 \label{goalcnu}  U(M_{G_i^*+zG_j}\oplus W_i) U^*&=M_{G_i'^*+zG_j'}\oplus W_i' \text{ for } (i,j)=(1,2),(2,1), \\
 \label{goalcnu1} U(M_z\oplus M_{\zeta})U^*&=(M_z\oplus M_{\zeta}).
\end{align}
To verify \eqref{goalcnu}--\eqref{goalcnu1}, proceed as follows.
Since $M_{\zeta}$ commutes with $Z$ and $\Delta_{\Theta''}$, (\ref{goalcnu1}) follows easily.
We establish equation (\ref{goalcnu}) only for $(i,j)=(1,2)$ and omit the proof for the other case because it is similar. For $f\in H^2(\mathcal{F})$ and $g\in\overline{\Delta_{\Theta'}L^2(\cD)}$,
\begin{align*}
&U(M_{G_1^*+zG_2}\oplus W_1)U^*(f\oplus g)=U(M_{G_1^*+zG_2}\oplus W_1)(\Theta''f\oplus Z^{-1}(\Delta_{\Theta''}f\oplus g))\\
=&\; U(\Theta''M_{G_1'^*+zG_2'}f\oplus Z^{-1}(\Delta_{\Theta''}M_{G_1'^*+zG_2'}f\oplus W_1'g))\;[\text{by }(\ref{ExtraCond})]\\
=&\; M_{G_1'^*+zG_2'}f\oplus W_1'g
\end{align*}
and \eqref{goalcnu} also follows.

We now show that the triple $((G_1',G_2'),(W_1',W_2'),\Theta')$ is admissible. Recall that in the course of the
proof of Theorem \ref{T:inv-cnu}, we saw that both $G_1'$ and $G_2'$ are contractions and that $(W_1',W_2')$
is a pair of commuting unitaries satisfying (\ref{propW'}). From (\ref{goalcnu}) we see that for every $f\in H^2(\cF)$
and $(i,j)=(1,2),(2,1)$,
\begin{align*}
(M_{G_i'^*+zG_j'}\oplus W_i') (\Theta'f\oplus\Delta_{\Theta'}f)
=&\;U(M_{G_i^*+zG_j}\oplus W_i) U^*(\Theta'f\oplus\Delta_{\Theta'}f)\\
=&\;U(M_{G_i^*+zG_j}\oplus W_i)\big(\Theta''\Theta'f\oplus Z^{-1}(\Delta_{\Theta''}\Theta'f\oplus \Delta_{\Theta'}f\big)\\
=&\;U(M_{G_i^*+zG_j}\oplus W_i)\big(\Theta f\oplus \Delta_{\Theta}f\big).
\end{align*}
From the admissibility of $((G_1,G_2),(W_1,W_2),\Theta)$, we know that each of the contraction operators
$(M_{G_i^*+zG_j}\oplus W_i)$ takes the space
$\{\Theta f\oplus \Delta_{\Theta}f:f\in H^2(\cD)\}$ into itself. Therefore from the last term of the above computation and (\ref{ortho}), we see that for each $(i,j)=(1,2),(2,1)$,
$$
(M_{G_i'^*+zG_j'}\oplus W_i')\big(\{\Theta' f\oplus \Delta_{\Theta'}f:f\in H^2(\cD)\}\big)\subset\{\Theta' f\oplus \Delta_{\Theta'}f:f\in H^2(\cD)\}.
$$
From (\ref{goalcnu}) it is also clear that for each $(i,j)=(1,2),(2,1)$, the operators $(M_{G_i'^*+zG_j'}\oplus W_i')$ are contractions and that with $\mathfrak{H}'$ as in (\ref{fracH'})
\begin{align*}
&(M_{G_i'^*+zG_j'}\oplus W_i')^*(M_{G_j'^*+zG_i'}\oplus W_j')^*|_{\mathfrak{H}'}=U(M_{G_i^*+zG_j}
\oplus W_i)^*(M_{G_j^*+zG_i}\oplus W_j)^*|_{\bH}\\
=&\;U(M_z\oplus M_{\zeta})|_{\bH}=U(M_z\oplus M_{\zeta})U^*|_{\mathfrak{H}'}=
(M_z\oplus M_{\zeta})|_{\mathfrak{H}'} \quad [\text{by } (\ref{goalcnu1})].
\end{align*}
This completes the proof of admissibility of $((G_1',G_2'),(W_1',W_2'),\Theta')$.

And finally to prove the last part we first observe that
\begin{align*}
({\bf T_1},{\bf T_2})=P_{\bH'}(M_{G_1^*+zG_2}\oplus W_1,M_{G_2^*+zG_1}\oplus W_2)|_{\bH'}.
\end{align*}
Now from equations (\ref{goalcnu}) and (\ref{goalcnu1}) again and from the fact that
$U(\bH')=\mathfrak{H}'$ (hence $UP_{\bH'}=P_{\mathfrak{H}'}U$), we conclude that
\begin{align*}
({\bf T_1},{\bf T_2},{\bf T_1T_2})|_{\bH'}=P_{\bH'}(M_{G_1^*+zG_2}\oplus W_1,M_{G_2^*+zG_1}\oplus W_2,M_z\oplus M_{\zeta})|_{\bH'}
\end{align*}
is unitarily equivalent to the functional model associated to $((G_1',G_2'),(W_1',W_2'),\Theta')$, i.e.,
\begin{align*}
P_{\mathfrak H'}(M_{G_1^*+zG_2}\oplus W_1,M_{G_2^*+zG_1}\oplus W_2,M_z\oplus M_{\zeta})|_{\mathfrak H'}
\end{align*}
via the unitary $U|_{\bH'}:\bH'\to\mathfrak{H}'$. Therefore appeal to Theorem \ref{UnitaryInv},
Theorem \ref{AdmisCharc} and Proposition \ref{P:pure-adm-triple} completes the proof.
\end{proof}

In case  the purely contractive analytic function $(\cD,\cD_{*},\Theta)$ is inner, the results above are much
simpler, as in the following statement.

\begin{thm}
Let $(\cD,\cD_*,\Theta)$ be an inner function and $((G_1,G_2),\Theta)$ be an admissible pair.
Define the pair $({\bf T_1},{\bf T_2})$ of commuting contractions on
\begin{align}\label{cnuMspace}
\bH=H^2(\cD_*)\ominus\{\Theta f: f\in H^2(\cD)\}
\end{align}by
\begin{align}\label{Mop}
 ({\bf T_1},{\bf T_2})=P_{\bH}\big(M_{G_1^*+zG_2},M_{G_2^*+zG_1}\big)|_{\bH}.
\end{align}
A subspace $\bH'$ of $\bH$ is jointly invariant under $({\bf T_1},{\bf T_2})$ if and only if there exist two inner functions
$(\mathcal{D},\mathcal{F},\Theta')$, $(\mathcal{F},\mathcal{D}_{*},\Theta'')$ such that
\begin{eqnarray*}
\Theta=\Theta''\Theta'
\end{eqnarray*}
is a regular factorization,
\begin{align}\label{H'}
\bH'=\{\Theta''f:f\in H^2(\mathcal{F})\}\ominus\{\Theta h:h\in H^2(\cD)\},
\end{align}
\begin{align}\label{H''}
\bH'':=\bH\ominus\bH'=&\;H^2(\cD_*)\ominus\{\Theta''f:f\in H^2(\mathcal{F})\},
\end{align}and two contractions $G'_1,G'_2$ in $\mathcal{B}(\mathcal{F})$ such that
\begin{align}\label{ExtraCond'}
M_{G_i^*+zG_j}M_{\Theta''}=M_{\Theta''}M_{G'_i+zG'_j}.
\end{align}
Moreover, the pair $((G_1',G_2'),\Theta')$ coincides with the characteristic pair for $(T_1,T_2)|_{\mathcal{H}'}$.
\end{thm}

\end{document}